\documentclass[12pt,reqno]{amsart}
\usepackage[dvipsnames]{xcolor}
\usepackage{graphicx}
\usepackage{tikz-cd}
\usepackage{amsmath, amssymb}
\usepackage{mathrsfs}
\usepackage{amsthm}
\usepackage{multicol}
\usepackage{color}
\usepackage{bm}
\usepackage{mathtools}
\usepackage{nameref}
\usepackage{url}
\usepackage[colorlinks=true,linkcolor=blue!60!black,citecolor=teal!80!black,urlcolor=RoyalBlue]{hyperref}
\usepackage[utf8x]{inputenc} 
\usepackage[left=2.5cm,right=2.5cm,top=2.5cm,bottom=2.5cm]{geometry}
\usepackage{enumitem}
\usepackage{accents}
\usepackage{import}

\usepackage[color=red!20!,bordercolor=red, shadow]{todonotes}
\newtheorem{thm}{Theorem}[section]
\newtheorem{cor}[thm]{Corollary}
\newtheorem{lemma}[thm]{Lemma}
\newtheorem{prop}[thm]{Proposition}

\newtheorem{dfn&prop}[thm]{Definition and Proposition}
\newtheorem{dfn&thm}[thm]{Definition and Theorem}

\theoremstyle{definition}
\newtheorem{defn}[thm]{Definition}
\newtheorem{observation}[thm]{Observation}

\newtheorem*{structure}{Structure of the article}

\newtheorem{discussion}[thm]{}

\theoremstyle{remark}
\newtheorem*{Claim}{Claim}

\newtheorem*{remark}{Remark}

\numberwithin{equation}{section}

\newtheorem*{thmIntroOrb}{Theorem \ref{thm_intro_orb}}
\newtheorem*{thmEstimates}{Theorem \ref{prop_expansion_Intro}}
\newtheorem*{examples}{Examples}

\newenvironment{subproof}{\begin{proof}[Proof of claim.]}{%
\end{proof}}

\newcommand{\Deriv}{{\rm D}}
\def\phi{\varphi}

\def\B{{\mathcal{B}}}

\def\C{{\mathbb{C}}}
\def\D{{\mathbb{D}}}
\def\N{{\mathbb{N}}}
\def\Z{{\mathbb{Z}}}
\def\R{{\mathbb{R}}}
\def\H{{\mathbb{H}}}

\def\Or{{\mathcal{O}}}
\def\Ort{{\widetilde{\mathcal{O}}}}

\newcommand{\dist}{\operatorname{dist}}
\newcommand{\e}{\operatorname{e}}
\newcommand{\degr}{\operatorname{deg}}

\newcommand{\lcm}{\operatorname{lcm}}

\newcommand{\id}{\operatorname{id}}
\newcommand{\Ima}{\operatorname{Im}}
\newcommand{\Rea}{\operatorname{Re}}

\newcommand{\Crit}{\operatorname{Crit}}

\newcommand*{\defeq}{\mathrel{\vcenter{\baselineskip0.5ex \lineskiplimit0pt
	\hbox{\scriptsize.}\hbox{\scriptsize.}}}%
=}

\newcommand{\eqdef}{=\mathrel{\vcenter{\baselineskip0.5ex \lineskiplimit0pt
	\hbox{\scriptsize.}\hbox{\scriptsize.}}}}

\title{Orbifold expansion and entire functions with bounded Fatou components}

\author[L. Pardo-Sim\'{o}n]{Leticia Pardo-Sim\'{o}n}
\address{Institute of Mathematics of the Polish Academy of Sciences \\ ul.  \'Sniadeckich 8 \\ 00-656 Warsaw \\ Poland \\ 
\textsc{\newline \indent 
\href{https://orcid.org/0000-0003-4039-5556%
}{\includegraphics[width=1em,height=1em]{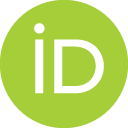} {\normalfont https://orcid.org/0000-0003-4039-5556}}
}}
\email{l.pardo-simon@impan.pl}
\subjclass[2010]{Primary 37F10; secondary 30D05, 57R18.}
\begin{document}

\begin{abstract}
Many authors have studied the dynamics of \emph{hyperbolic} transcendental entire functions; these are those for which the postsingular set is a compact subset of the Fatou set. Equivalenty, they are characterized as being \textit{expanding}. Mihaljevi\'c-Brandt studied a more general class of maps for which finitely many of their postsingular points can be in their Julia set, and showed that these maps are also expanding with respect to a certain orbifold metric. In this paper we generalise these ideas further, and consider a class of maps for which the postsingular set is not even bounded. We are able to prove that these maps are also expanding with respect to a suitable orbifold metric, and use this expansion to draw conclusions on the topology and dynamics of the maps. In particular, we generalize existing results for hyperbolic functions, giving criteria for the boundedness of Fatou components and local connectivity of Julia sets. As part of this study, we develop some novel results on \textit{hyperbolic orbifold metrics}. These are of independent interest, and may have future applications in holomorphic dynamics.	
\end{abstract}
\maketitle

\section{Introduction}
In the study of \textit{dynamical systems}, the notion of \textit{expansion} in its various forms is fundamental. For a holomorphic function, expansion has frequently been understood in terms of a conformal metric defined on a neighbourhood of its Julia set. More specifically, a polynomial $p$ is \textit{hyperbolic} if it is expanding with respect to a hyperbolic metric induced on a neighbourhood of its compact Julia set $J(p)$. This is equivalent to saying that every critical value of $p$ belongs to the basin of attraction of a periodic cycle, and in particular its orbit lies in the Fatou set $F(p)$ \cite[Theorem 1, p. 21]{Orsaynotes}. As a consequence of this expansion, whenever the Julia set of a hyperbolic polynomial is connected, then it is also locally connected \cite[Proposition 4, p. 19]{Orsaynotes}. For transcendental entire maps, infinity is an essential singularity and thus their Julia sets are no longer compact. Still, with slight modifications on the notion of expansion, that in particular requires the hyperbolic metric to be defined in a punctured neighbourhood of infinity, a definition and characterization of \textit{hyperbolic transcendental maps} are analogous to those in the polynomial case; see \cite[Theorem and Definition 1.3]{Lasse_Dave_ClassB} and Definition \ref{def_hyperb}. Again, expansion of hyperbolic transcendental maps was used in \cite{lasseNuriaWalter} to draw conclusions on the topology of their Julia and Fatou sets.

Regarding expansion arguments, it is crucial for a hyperbolic map $f$ that both its set of singular values $S(f)$, that is, the closure of the set of singularities of the inverse $f^{-1}$, and the closure of its forward orbit, called the \textit{postsingular set} $P(f) \defeq\overline{\bigcup_{n\geq 0}f^n (S(f))}$, are contained in its Fatou set. This is because then, all iterates of $f$ act as a covering map in a neighbourhood of $J(f)$, where the hyperbolic metric sits, that does not intersect $P(f)$. Even if such a neighbourhood no longer exists for \textit{subhyperbolic polynomials}, that is, those for which $P(p)\cap F(p)$ is compact and $P(p)\cap J(p)$ is finite, still Douady and Hubbard were able to extend these ideas to this more general setting. More precisely, inspired by work of Thurston \cite{thurstonGeom}, they overcame the presence of postsingular points in the Julia sets of subhyperbolic polynomials by considering $J(p)$ as a subset of a Riemann orbifold on which $p$ acts as an orbifold covering map. In particular, they proved subhyperbolic polynomials to be expanding with respect to a corresponding \text{orbifold metric} \cite{Orsaynotes}, see \S\ref{sec_back_orb} for basic definitions on orbifold metrics. Thanks to this expansion, they showed that the aforementioned result on local connectivity of Julia sets for hyperbolic polynomials generalizes to subhyperbolic ones \cite[Proposition 4, p. 19]{Orsaynotes}.

The notion of subhyperbolicity for transcendental entire functions was first introduced by Mihaljevi\'c-Brandt \cite{helenaSemi}. A transcendental entire map $f$ is said to be \textit{subhyperbolic} if $P(f)\cap F(f)$ is compact and $P(f)\cap J(f)$ is finite. For a transcendental entire function, the presence of asymptotic values or critical points with arbitrarily large local degree in its Julia set, prevents its Julia set to be successfully considered a subset of an orbifold \cite[Proposition 3.6]{helenaSemi}. However, orbifold expansion is achieved in \cite[Theorem 4.1]{helenaSemi} for subhyperbolic functions for which this does not occur. That is, subhyperbolic maps with \textit{bounded criticality} on their Julia sets, which are called \textit {strongly subhyperbolic}.

Note that since the postsingular set of subhyperbolic transcendental maps is bounded, all these maps belong to the broadly studied \textit{Eremenko-Lyubich class} $\mathcal{B}$. This class consists of all transcendental entire functions with bounded singular set, and this resemblance to the polynomial case has made this class a target of study in transcendental dynamics. Moreover, the fact that for subhyperbolic maps the postsingular set is also bounded is decisive in the arguments concerning estimates on orbifold metrics in \cite{Orsaynotes, helenaSemi}. In this paper, we generalize strongly subhyperbolic functions to a class of functions that contain critical values escaping to infinity, and thus their postsingular set might be \textbf{unbounded}. More precisely, we say that a transcendental entire function $f$ is \textit{postcritically separated} if $P(f)\cap F(f)$ is compact and $P_J \defeq J(f)\cap P(f)$ is discrete. If in addition $f$ has bounded criticality in $J(f)$, there is a uniform bound for the number of critical points in the orbit of any $z\in J(f)$, and there is $\epsilon>0$ so that for any distinct $z,w\in P_J$, $\vert z-w\vert\geq \epsilon \max\{\vert z \vert, \vert w \vert\}$, we say that $f$ is \textit{strongly postcritically separated.} See Definition \ref{def_strongps}. We note that these maps might have unbounded singular set, but expansion is achieved for those that are additionally in $\B$.

To each strongly postcritically separated map $f\in \B$, we associate a pair of orbifolds $\Or$ and $\Ort$ whose \textit{underlying surfaces} are respective neighbourhoods of $J(f)$, and so that we can \textit{extend} $f$ to be an orbifold covering map between them. With slight abuse of notation, we also denote this map between orbifolds by $f$, see \S\ref{sec_back_orb} for more details. Moreover, for a holomorphic (orbifold) map $f\colon \Ort \rightarrow \Or$ between two hyperbolic orbifolds, we define the hyperbolic derivative with respect to their orbifold metrics as 
\begin{equation*}
\Vert \Deriv f(z)\Vert^\Or_{\Ort}\defeq\vert f'(z)\vert \cdot\frac{\rho_{\Or}(f(z))}{\rho_{\Ort}(z)},
\end{equation*}
whenever the quotient is defined. If, in addition, both $z$ and $f(z)$ belong to the underlying surface of $\Or$, we also abbreviate $\Vert \Deriv f(z)\Vert_{\Or} \defeq  \Vert \Deriv f(z)\Vert^\Or_{\Or}$.

\begin{thm}[Orbifold expansion for strongly postcritically separated maps] \label{thm_main_intro_Orb} Let $f\in \B$ be a strongly postcritically separated map. Then, there exist a constant $\Lambda > 1$ and a pair of hyperbolic orbifolds $\Or$ and $\Ort$ such that $f \colon\Ort\to\Or$ is an orbifold covering map,
\begin{equation}\label{eq_exp_lambda}
\Vert \Deriv f(z)\Vert_{\Or}=\vert f'(z)\vert \cdot\frac{\rho_{\Or}(f(z))}{\rho_{\Or}(z)} \geq\Lambda
\end{equation}
whenever the quotient is defined, and $J(f)$ is contained in the underlying surfaces of both $\Or$ and $\Ort$.
\end{thm}
Theorem \ref{thm_main_intro_Orb} has allowed us to provide in the sequel paper \cite{mio_splitting} a complete description of the topological dynamics of certain transcendental functions in class $\B$ with critical values that escape to infinity, being this the first result of the kind. Namely, the results in \cite{mio_splitting} hold for
strongly postcritically separated functions satisfying some additional condition that guarantees the existence of \textit{dynamic rays}, see Definition \ref{def_ray}, in their Julia sets. In this paper, 
we use Theorem \ref{thm_main_intro_Orb} to generalize some of the results in \cite{lasseNuriaWalter} on the topology of Julia and Fatou sets for hyperbolic functions, to the larger class of strongly postcritically separated maps. The first one is a generalization of \cite[Theorem 1.2]{lasseNuriaWalter}:
\begin{thm}[Bounded Fatou components] \label{thm_intro1.2} 
Let $f\in \B$ be strongly postcritically separated. Then the following are equivalent:
\begin{enumerate}[label=(\alph*)]
\item \label{itema_intro1.2} every component of $F(f)$ is a bounded Jordan domain;
\item \label{itemb_intro1.2} $f$ has no asymptotic values and every component of $F(f)$ contains at most finitely many critical points.
\end{enumerate}
\end{thm}

As a consequence of this theorem, we obtain the following result on local connectivity of Julia sets, that generalizes \cite[Corollary 1.8]{lasseNuriaWalter}.

\begin{cor}[Bounded degree implies local connectivity]\label{cor_intro1.8} Let $f\in \B$ be strongly postcritically separated with no asymptotic values. Suppose that there is a uniform bound on the number of critical points, counting multiplicity, in the Fatou components of $f$. Then $J(f)$ is locally connected.
\end{cor}

The next result provides further sufficient conditions for the local connectivity of Julia sets. Compare to \cite[Corollary 1.9(a)]{lasseNuriaWalter}.
\begin{cor}[Locally connected Julia sets]\label{cor_intro1.9} Let $f\in \B$ be strongly postcritically separated with no asymptotic values, suppose that every component of $F(f)$ contains at most one critical value, and that the multiplicity of the critical points of $f$ is uniformly bounded. Then $J(f)$ is locally connected.
\end{cor}

In order to successfully associate orbifolds to a holomorphic function so that some analogue of Theorem \ref{thm_main_intro_Orb} holds, the set of ramified points of the orbifolds, and hence the set of singularities of the corresponding orbifold metrics, must contain the postsingular points of the function that are also in its Julia set; see the discussion at the beginning of \S\ref{sec_orbifolds}. Since for strongly postcritically separated functions, these points might tend to the essential singularity at infinity, we require global estimates of the densities of metrics on hyperbolic orbifolds, in particular generalizing some known estimates for metrics on hyperbolic domains. These estimates, that hold for orbifolds $\Ort$, $\Or$ for which the inclusion $\Ort \hookrightarrow\Or$ is holomorphic,  come in terms of the \textit{boundary of $\Ort$ in $\Or$}, denoted $\mathbb{B}^\Or_\Ort$. The set $\mathbb{B}^\Or_\Ort$ consists of boundary points of the underlying surface of $\Ort$ that are in $\Or$, together with those points that have greater ramification value in $\Ort$ than in $\Or$; see Definition \ref{def_setD}.

\begin{thm}[Estimates on relative densities]\label{prop_expansion_Intro}
Let $\Ort$ and $\Or$ be hyperbolic orbifolds such that the inclusion $\Ort \hookrightarrow \Or$ is holomorphic, and denote by $\rho_{\Ort}$ and $\rho_\Or$ their respective densities. If $R$ is the $\Or$-distance between $z\in \Ort$ and $\mathbb{B}^\Or_\Ort$, then
\begin{equation}
1< \frac{e^R}{\sqrt{e^{2R}-1}} \leq \frac{\rho_{\Ort}(z)}{\rho_{\Or}(z)}\leq 1+ \frac{2}{e^R-1},
\end{equation}
whenever the quotient is defined.
\end{thm}
We also show in this paper that whenever singularities of an orbifold metric are ``continuously perturbed'', the orbifold metric of the new orbifold is a ``continuous perturbation'' of the metric of the original orbifold; see Theorem \ref{thm_cont_or}. In particular, that result has the following implication.
\begin{thm}[Distances are uniformly bounded across certain orbifolds]\label{thm_intro_orb} Given a compact set $A\subset U$ and constants $\epsilon>0$ and $c, M\in \N_{\geq 1}$, there is a constant $R\defeq R(U,A,\epsilon,c,M)>0$ such that for every orbifold $\Or$ with underlying surface $U$ and at most $M$ ramified points, each with ramification value at most $c$, and such that the Euclidean distance between any two of them is at least $\epsilon$, it holds that $$d_\Or(p,q)<R \quad \text{ for every } \quad p,q \in A.$$
\end{thm}

Finally, we introduce in section \ref{sec_homotopies} a modified notion of \textit{homotopy} for which we obtain in Proposition \ref{cor_homot} an analogue of the Homotopy Lifting property for certain class of curves that contain postsingular points. Moreover, we show in Corollary \ref{cor_homot2} that if $U$ is a bounded set of a hyperbolic orbifold such that $P(f)\subset U$ is finite and there is a dynamic ray of $f$ landing at each of those postsingular points, then there exists a constant $\mu$ such that for any piece of dynamic ray of $f$ contained in $U$, we can find a curve on its ``modified homotopy class'' with orbifold length at most $\mu$. This result is of great value for expansion arguments, and, in particular, it is used in \cite{mio_splitting}.

\begin{structure}
In \S\ref{sec_setting} we provide the formal definition of strongly postcritically separated maps, their basic properties and give some examples. \S \ref{sec_back_orb} includes background on Riemann orbifolds, and \S\ref{sec_orb_estimates} studies orbifold metrics and contains the proofs of Theorems \ref{prop_expansion_Intro} and \ref{thm_intro_orb}. Using these results, in \S\ref{sec_orbifolds} we construct for each strongly postcritically separated map a pair of dynamically associated orbifolds and prove Theorem \ref{thm_main_intro_Orb}. \S\ref{sec_Fatou} contains the proofs of the results on Fatou components and local connectivity of Julia sets, that is, Theorem~\ref{thm_intro1.2} and Corollaries \ref{cor_intro1.8} and \ref{cor_intro1.9}. These proofs will easily follow from the study of \textit{periodic} Fatou components: Theorem \ref{thm_my1.10} gives several conditions equivalent to the boundedness of a periodic Fatou component. Finally, \S\ref{sec_homotopies} includes results regarding curves in homotopy classes with a uniform bound on their (orbifold) length.
\end{structure}

\noindent \textbf{Basic notation} As introduced throughout this section, the Fatou and Julia set of an entire function $f$ are denoted by $F(f)$ and $J(f)$ respectively. Moreover, we define its \textit{escaping set} as $I(f)\defeq \{z\in \C : f^n(z)\rightarrow \infty \text{ as } n\rightarrow \infty \}$. The set of critical values of $f$ is $CV(f)$, that of asymptotic values is $AV(f)$, and the set of critical points is $\Crit(f)$. The set of singular values of $f$ is $S(f)$, and $P(f)$ is its postsingular set. Moreover, we let $P_{J}\defeq P(f)\cap J(f)$ and $P_{F}\defeq P(f)\cap F(f)$. We denote the complex plane by $\C$, the Riemann sphere by $\widehat{\C}$ and the upper half-plane by $\H$. A disc of radius $\epsilon$ centred at a point $p$ will be $\D_\epsilon(p)$, the unit disc centred at $0$ will be abbreviated as $\D$, and $\D^\ast\defeq \D\setminus \{0\}$. We indicate the closure of a domain either by $\overline{U}$ or cl$(U)$ in such a way that it will be clear from the context, and these closures must be understood to be taken in $\C$. $A\Subset B$ means that $A$ is compactly contained in $B$. The annulus with radii $a<b \in \C$ will be denoted by $A(a,b)\defeq\lbrace w\in \C : a< \vert w \vert < b \rbrace.$ For a holomorphic function $f$ and a set $A$, $\text{Orb}^{-}(A)$ and $\text{Orb}^{+}(A)$ are the backward and forward orbit of $A$ under $f$; that is, $\text{Orb}^{-}(A)\defeq \bigcup^{\infty}_{n=0} f^{-n}(A)$ and $\text{Orb}^{+}(A)\defeq \bigcup^{\infty}_{n=0} f^{n}(A).$

\textbf{Acknowledgements.} I am very grateful to my supervisors Lasse Rempe and Dave Sixsmith for their continuous help and support. I also thank Daniel Meyer, Phil Rippon and the referee for very valuable comments.
\section{Strongly postcritically separated functions}\label{sec_setting}
\begin{defn}[Postcritically separated, subhyperbolic and hyperbolic maps]\label{def_hyperb} A transcendental entire function $f$ is \textit{postcritically separated} if $P_{J}\defeq P(f)\cap J(f)$ is discrete and $P_{F}\defeq P(f)\cap F(f)$ is compact. In the particular case when $P(f)\cap J(f)$ is finite, $f$ is called \textit{subhyperbolic}, and when $P(f)\cap J(f)=\emptyset$, $f$ is \textit{hyperbolic}.
\end{defn}

\begin{observation}[Dichotomy of points in $P_J$] \label{rem_setting} If $f$ is postcritically separated, then any $p\in P_J$ is either (pre)periodic, or it escapes to infinity: indeed, if $p\notin I(f)$, then there exists a subsequence of points in the orbit of $p$ that lies in a bounded set, and by discreteness of $P_J$ on that set, the claim follows. By the same argument, if in addition $f\in \mathcal{B}$, then there can be at most finitely many points in $S(f)\cap I(f)$.
\end{observation}

Recall that for a holomorphic map $f:\widetilde{S}\rightarrow S$ 
between Riemann surfaces, the \emph{local degree} of $f$ 
at a point $z_0\in \widetilde{S}$, denoted by $\deg(f,z_0)$, is the unique integer $n\geq 1$ 
such that the local power series development of $f$ is of the form
\begin{equation*}
f(z)=f(z_0) + a_n (z-z_0)^n + \text{(higher terms)},
\end{equation*}
where $a_n\neq 0$. Thus, $z_0\in \C$ is a critical point of $f$ if and only if $\degr(f,z_0)>1$. We say that $f$ has \textit{bounded criticality} in a set $A$ if $\text{AV}(f) \cap A=\emptyset$ and in addition, there exists a constant $M<\infty$ such that $\degr(f,z)<M \text{ for all } z \in A.$

\begin{defn}[Strongly postcritically separated functions]\label{def_strongps}
A postcritically separated transcendental entire map $f$ is \textit{strongly postcritically separated} with parameters $(c,\epsilon)$ if:
\begin{enumerate}[label=(\alph*)]
\item \label{itema_defsps} $f$ has bounded criticality in $J(f)$;
\item \label{itemb_defsps} for each $z\in J(f)$, $\#(\text{Orb}^+(z)\cap \Crit(f)) \leq c$;
\item \label{itemd_defsps} for all distinct $z,w\in P_J$, $\vert z-w\vert\geq \epsilon \max\{\vert z \vert, \vert w \vert\}$.
\end{enumerate}
\end{defn}
\begin{observation}[Separation of points in $P_J$]\label{obs_separation}
In the definition of strongly postcritically separated map, \ref{itemd_defsps} has the following implication: for every constant $K>0$, there exists a constant $M>0$, depending only on $\epsilon$ and $K$, so that 
\begin{equation}\label{itemc_defsps} 
\#(P_J \cap \overline{A(r,Kr)})\leq M \quad \text{ for all }r>0. 
\end{equation}
To see this, note that the annulus $\overline{A(1,K)}$ admits at most some number $M$ of points in $P_J$, so that these points are at pairwise distance at least $\epsilon$. Moreover, by \ref{itemd_defsps}, for any $r>0$ and all distinct $w,z\in (P_J \cap A(r,Kr))$, $\vert z/r-w/r\vert>\epsilon$. The combination of these two facts implies \eqref{itemc_defsps}. In particular, since, as $r$ increases, the annuli ``$A(r,Kr)$'' are of greater area, the orbit of any point $z\in S(f)\cap I(f)$ must converge to infinity at more than a \textit{constant rate}, i.e., for all $C\in \R^+$, there must exist $n\geq 0$ so that $\vert f^{n+1}(z)\vert> \vert f^n(z)\vert +C$.
\end{observation}

\begin{remark}
When $f$ is subhyperbolic and Definition \ref{def_strongps}\ref{itema_defsps} holds, $f$ is called \textit{strongly subhyperbolic} \cite[Definition 2.11]{helenaSemi}. Note that for subhyperbolic maps, conditions \ref{itemb_defsps} and \ref{itemd_defsps} in Definition \ref{def_strongps} are trivially satisfied, and thus any strongly subhyperbolic map is a strongly postcritically separated one. 
\end{remark}

\begin{figure}[htb]
\begingroup%
\makeatletter%
\providecommand\color[2][]{%
\errmessage{(Inkscape) Color is used for the text in Inkscape, but the package 'color.sty' is not loaded}%
\renewcommand\color[2][]{}%
}%
\providecommand\transparent[1]{%
\errmessage{(Inkscape) Transparency is used (non-zero) for the text in Inkscape, but the package 'transparent.sty' is not loaded}%
\renewcommand\transparent[1]{}%
}%
\providecommand\rotatebox[2]{#2}%
\newcommand*\fsize{\dimexpr\f@size pt\relax}%
\newcommand*\lineheight[1]{\fontsize{\fsize}{#1\fsize}\selectfont}%
\ifx\svgwidth\undefined%
\setlength{\unitlength}{425.19685039bp}%
\ifx\svgscale\undefined%
\relax%
\else%
\setlength{\unitlength}{\unitlength * \real{\svgscale}}%
\fi%
\else%
\setlength{\unitlength}{\svgwidth}%
\fi%
\global\let\svgwidth\undefined%
\global\let\svgscale\undefined%
\makeatother%
\resizebox{0.75\textwidth}{!}{
\begin{picture}(1,0.46666667)%
\lineheight{1}%
\setlength\tabcolsep{0pt}%
\put(0,0){\includegraphics[width=\unitlength,page=1]{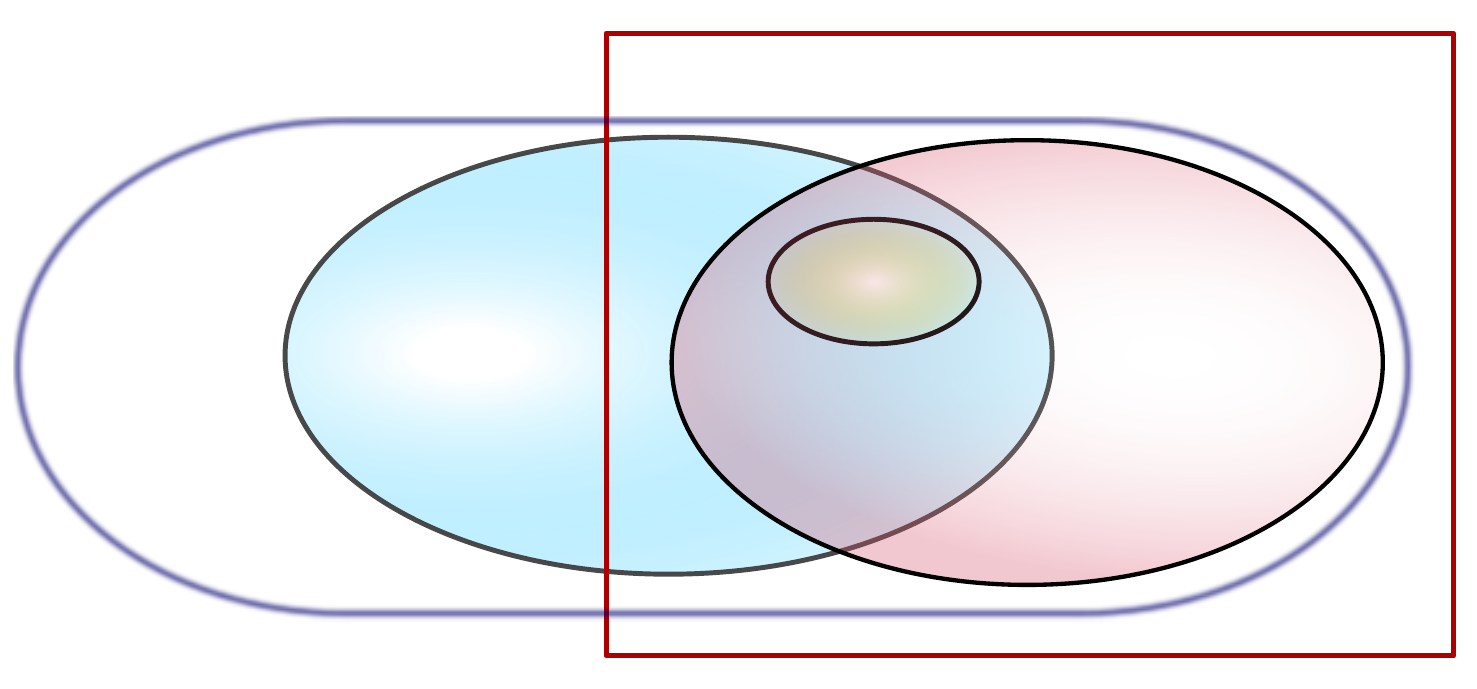}}%
\put(0.53420635,1.1414881){\color[rgb]{0,0,0}\makebox(0,0)[lt]{\begin{minipage}{0.29734129\unitlength}\raggedright \end{minipage}}}%
\put(0.58208333,1.05077378){\color[rgb]{0,0,0}\makebox(0,0)[lt]{\begin{minipage}{0.06551586\unitlength}\raggedright \end{minipage}}}%
\put(0.68287697,1.1540873){\color[rgb]{0,0,0}\makebox(0,0)[lt]{\begin{minipage}{0.46869054\unitlength}\raggedright \end{minipage}}}%
\put(0.86682536,1.13644841){\color[rgb]{0,0,0}\makebox(0,0)[lt]{\begin{minipage}{0.09323413\unitlength}\raggedright \end{minipage}}}%
\put(-0.30742064,0.7660318){\color[rgb]{0,0,0}\makebox(0,0)[lt]{\begin{minipage}{0.8617857\unitlength}\raggedright \end{minipage}}}%
\put(0.7509127,1.68073414){\color[rgb]{0,0,0}\makebox(0,0)[lt]{\begin{minipage}{0.1915079\unitlength}\raggedright \end{minipage}}}%
\put(0.23697899,0.2922147){\color[rgb]{0,0,0}\makebox(0,0)[lt]{\begin{minipage}{0.14788915\unitlength}\raggedright \end{minipage}}}%
\put(0.52724487,0.27173968){\color[rgb]{0,0,0}\makebox(0,0)[lt]{\lineheight{1.25}\smash{\begin{tabular}[t]{l}Hyperbolic\end{tabular}}}}%
\put(0.5874264,0.20245962){\color[rgb]{0,0,0}\makebox(0,0)[t]{\lineheight{0.75}\smash{\begin{tabular}[t]{c}Strongly \\subhyperbolic\end{tabular}}}}%
\put(0.73976232,0.21684647){\color[rgb]{0,0,0}\makebox(0,0)[lt]{\lineheight{0.75}\smash{\begin{tabular}[t]{l}Subhyperbolic\end{tabular}}}}%
\put(0.31330284,0.25982596){\color[rgb]{0,0,0}\makebox(0,0)[t]{\lineheight{1.25}\smash{\begin{tabular}[t]{c}Strongly \\postcritically\\separated\end{tabular}}}}%
\put(0.10891233,0.23704712){\color[rgb]{0,0,0}\makebox(0,0)[t]{\lineheight{1.25}\smash{\begin{tabular}[t]{c}Postcritically \\separated\end{tabular}}}}%
\put(0.60350199,0.41073413){\color[rgb]{0.68627451,0,0}\makebox(0,0)[lt]{\lineheight{1.25}\smash{\begin{tabular}[t]{l}\textbf{Class $\mathcal{B}$}\end{tabular}}}}%
\put(1.90952372,-0.53191121){\color[rgb]{0,0,0}\makebox(0,0)[lt]{\begin{minipage}{0.32481974\unitlength}\raggedright \end{minipage}}}%
\end{picture}%
}
\endgroup%
\caption{Illustration of the relationships between the classes of functions defined in this section.}
\label{fig:inclusions}
\end{figure}
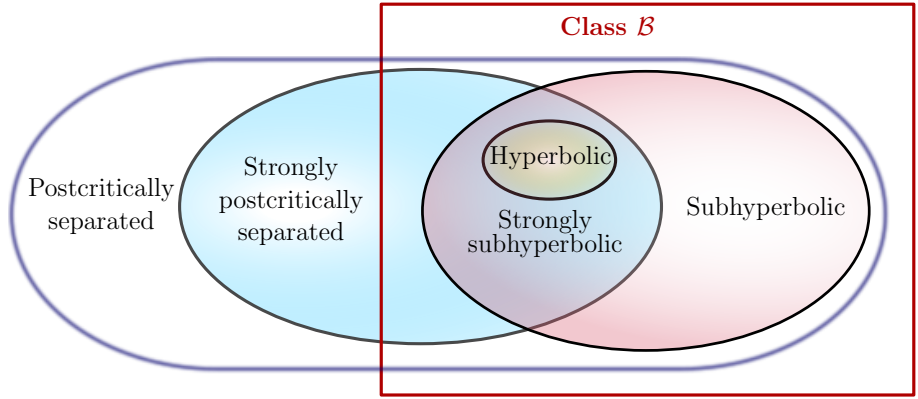

\begin{remark}If $f$ is a strongly postcritically separated map, then so is $f^n$ for all $n\geq~1$. This follows from the facts that $\text{AV}(f^n)=\bigcup^{n-1}_{i=0} f^i(\text{AV}(f))$, $\text{CV}(f^n)=\bigcup^{n-1}_{i=0} f^i(\text{CV}(f))$, $J(f^n)=J(f)$ and $P(f^n)=P(f)$. 
\end{remark}

\begin{examples}The following functions belong to the classes of maps just defined:
\begin{itemize}[wide=0pt, leftmargin=\dimexpr\labelwidth + 2\labelsep\relax]
\item The exponential map is a postcritically separated map in class $\B$ that is neither strongly postcritically separated nor subhyperbolic, since its asymptotic value $0$ escapes to infinity and is in its Julia set; see for example \cite{Lasse_exp_chaotic}.
\item The function $f(z)\defeq \pi\sinh(z)$ has only two critical values and no asymptotic values. Moreover, $P(f)=\{0,\pm\pi i\} \subset J(f)$. Thus, $f$ is strongly subhyperbolic, and hence strongly postcritically separated; see \cite[Appendix A]{helenaSemi} for a description of the dynamics of this map.
\item For the function $f(z)\defeq \cosh(z)$, $S(f)=\text{CV}(f)=\{-1,1\}\subset I(f)$. Moreover, $f\in \B$ and is strongly postcritically separated, but not subhyperbolic; see \cite{mio_cosine} for more details on the dynamics of this map. In particular, Theorem \ref{thm_main_intro_Orb} applies to it.
\item Let $\text{erf}$ denote the error function \cite[p. 297]{abramowitz_stegun}, and let $\alpha\in \C$ be a complex solution to $\text{erf}(\alpha)=1$. In particular, we set $\alpha\approx 5.902- 0.262 i$. Let $g\colon \C\rightarrow \C$ given by
$$g(z)\defeq \frac{2 i \Ima(\alpha)}{\sqrt{\pi}}\int_0^z \e^{-w^2}dw + \Rea(\alpha)= i \Ima(\alpha)\text{erf}(z) + \Rea(\alpha).$$
Then $S(g)=\text{AV}(g)=\{\alpha, \overline{\alpha}\}$, where $\overline{\alpha}$ is the complex conjugate of $\alpha$. Since $\text{erf}(\overline{\alpha})=\overline{\text{erf}(\alpha)}=\overline{\alpha}$, both asymptotic values are fixed points in $J(g)$. Hence, $g$ is postcritically separated but not strongly postcritically separated; see \cite[p. 7]{Dave_infinityconnected} for more details on functions constructed this way.	
\item The function $\cosh(z) -1$ has as singular set two critical values, namely the point $0$, which is superattracting, and the point $-2$, which belongs to the escaping set of the function. Hence, this is another example of a strongly postcritically separated function in class $\B$.
\end{itemize}
\end{examples}
For a transcendental entire map $f$, we denote by $\mathcal{A}(f)$ the set of all points whose forward orbit converges to some attracting cycle of $f$. The following property will be of use to us when $f\in \B$ is postcritically separated, since as we shall see in Lemma \ref{lem_deff}, in that case $P_F \Subset F(f)=\mathcal{A}(f)$. By \textit{Jordan domain} we mean a complementary component of a Jordan curve on the sphere that is also a simply connected domain in $\C$. In particular, it might be bounded or unbounded.
\begin{prop}[Compact subsets of attracting basins {\cite[Proposition 3.1]{helena_thesis}}] \label{prop_Jordan}
Let $f$ be a transcendental entire function and let $C \subset \mathcal{A}(f)$ be a compact set. Then there exist bounded Jordan domains $U_1,\ldots, U_n$ compactly contained in pairwise different components of $\mathcal{A}(f)$ such that if $U \defeq \bigcup^n_{i=1} U_i$, then 
$$f(U) \Subset U \Subset \mathcal{A}(f) \quad \text{ and } \quad \text{Orb}^+(C) \Subset U.$$
\end{prop}

The types of Fatou components that might occur for postcritically separated functions follow from classical results; see for example \cite[\S4 and Theorem 6]{Bergweilermerophormic} for definitions and a classification of periodic Fatou components. In particular, we say that a wandering domain $U$ is \textit{escaping} if $U \subseteq I(f)$. We denote by $P(f)′$ the derived set of $P(f)$, that is, the set of its finite limit points.
\begin{lemma}[Fatou components for postcritically separated maps] \label{lem_deff} Let $f$ be postcritically separated. Then $F(f)$ is either empty or might consist of a collection of attracting basins, Baker domains and escaping wandering domains. The number of attracting basins must be finite, and in the two latter cases, the domains do not contain singular values. In particular, $P_F$ is contained in a finite union of attracting basins, and every periodic cycle in $J(f)$ is repelling. If in addition $f \in \B$, $F(f)$ is either empty or a finite union of attracting basins.
\end{lemma}
\begin{remark}
We do not claim the existence of examples of postcritically separated functions with wandering domains nor Baker domains. Instead, this lemma discards the existence of some types of Fatou components for these functions. 
\end{remark}

\begin{proof}[Proof of Lemma \ref{lem_deff}]
Compactness of $P_F$ excludes parabolic components: suppose that $F(f)$ had a parabolic component $U$ of period $p$, with a parabolic fixed point $z_0 \in \partial U$ such that for every $z \in U$, $f^{np}(z) \rightarrow z_0$ as $n \rightarrow\infty$. Then, since every cycle of immediate attracting or parabolic basins contains a singular value, \cite[Theorem 7]{Bergweilermerophormic}, there would exist $w\in S(f)$ that would also belong to some component in the cycle of $U$, and hence $f^j(w)\in U$ for some $0\leq j< p$. But $\text{Orb}^+(w)\subset P_F$, and simultaneously, $\text{Orb}^+(w)$ would contain the subsequence $f^{p+j}(w), f^{2p+j}(w), \ldots$ converging to $z_0 \notin F(f)$, which would contradict compactness of $P_F$. Hence, $F(f)$ does not contain parabolic components. Note that by our assumptions of the discreteness of $P_J$ and the compactness of $P_F$, $J(f)\cap P(f)'=\emptyset$. Thus, since every boundary point of each Siegel disk is a limit point of $P(f)$, see \cite[Corollary 14.4]{milnor_book}, Siegel disks cannot occur for $f$.

If $U$ is a wandering domain of $f$, since $J(f)\cap P(f)'=\emptyset$, by \cite{BergweilerWandering93}, the only possible limit function of $\{f^n\vert_U\}_{n\geq 0}$ is infinity, and so $U$ must be escaping. 
Since $P_F$ is compact, $I(f)\cap P_F=\emptyset,$ and so if Baker or escaping wandering domains occur for $f$, they cannot contain singular orbits. Hence, $P_F\subset \mathcal{A}(f)$, and by Proposition \ref{prop_Jordan}, $P_F$ is contained in finitely many attracting basins. Since each cycle of attracting periodic components must contain a postsingular point, there cannot be any further attracting basins of $F(f)$. We have already discarded parabolic cycles in $J(f)$, as there are no parabolic components in $F(f)$. If $z_0$ was an irrationally indifferent periodic point in $J(f)$, then there would be a sequence $\{w_k\}_k \subset P(f)$ converging non-trivially to $z_0$; see \cite[Corollary 14.4]{milnor_book}. Since $P_J$ is discrete and $P_F$ is contained in the union of finitely many attracting basins, this is impossible, and so all periodic cycles in $J(f)$ must be repelling. By \cite{eremenkoclassB}, functions in class $\B$ do not have Baker domains nor escaping wandering domains, and so for postcritically separated functions in this class, only attracting basins can occur.
\end{proof}

\section{Background on Riemann orbifolds}\label{sec_back_orb}
An \emph{orbifold} is a space that is locally represented as a quotient of an open subset $S$ of $\R^n$ by a linear action of a finite group (see \cite[Chapter 13]{thurstonGeom}). For the purposes of this paper, we are only interested in orbifolds modelled over Riemann surfaces. In this case, orbifolds are conveniently totally characterized by the surface $S$ together with a map that ``marks'' a discrete set of points of $S$. For a more detailed introduction to this particular case, we refer the interested reader to \cite[Appendix A]{mcmullen1994complex} and \cite[Chapter 19 and Appendix E]{milnor_book}. For the case when the orbifold is constructed over a $2$-sphere, see also \cite[Appendix A.9]{bonkexpanding}.
\begin{defn}[Riemann orbifold]
A \emph{Riemann orbifold} is a pair $(S,\nu)$ consisting of a Riemann surface $S$, called the \textit{underlying surface}, and a \emph{ramification map}\footnote{Unlike in other texts, we only allow the ramification map to take finite values.}
$\nu \colon S\rightarrow\N_{\geq 1}$ such that the set 
\begin{equation*}
\lbrace z\in S\; :\; \nu(z)>1\rbrace
\end{equation*}
is discrete. A point $z\in S$ for which $\nu(z)>1$ is called a \emph{ramified} or \emph{marked point}, and $\nu(z)$ is its \textit{ramification value}. If $\nu(z)=1$ we say that $z$ is \textit{unramified}. The \emph{signature} of an orbifold is the list of values that the ramification map $\nu$ assumes at
the ramified points, where each of them is repeated as often as it is assumed by $\nu$.
\end{defn} 
\begin{remark} We shall often use the term ``orbifold'' synonymously with ``Riemann orbifold''. Note that a traditional Riemann surface is a Riemann orbifold with ramification map $\nu\equiv 1$. In some cases, we will allow underlying surfaces to be disconnected, and hence certain properties should be understood component-wise. 
\end{remark}	

In order to define holomorphic maps between orbifolds, we recall that $f\colon\widetilde{S}\rightarrow S$ between Riemann surfaces is a \emph{branched covering map} if every $z\in S$ has a connected neighbourhood $U\ni z$ such that $f$ maps any component of $f^{-1}(U)$ onto $U$ as a proper map. 

\begin{defn}[Holomorphic and covering orbifold maps]
Let $\Ort=(\widetilde{S},\tilde{\nu})$ and $\Or=(S, \nu)$ be Riemann orbifolds. A \emph{holomorphic map} $f \colon\Ort\rightarrow\Or$ is a holomorphic map $f \colon\widetilde{S}\rightarrow S$ between the underlying Riemann surfaces such that
\begin{equation}\label{eq_holom_orb}
\nu(f(z)) \text{ divides } \degr(f,z)\cdot \tilde{\nu}(z)\quad \text{ for all } \quad z\in \widetilde{S}.
\end{equation}
If in addition $f \colon\widetilde{S}\rightarrow S$ is a branched covering map such that 
\begin{equation}\label{eq_covering_orb}
\nu(f(z)) = \degr(f,z)\cdot\tilde{\nu}(z)\quad \text{ for all } \quad z\in \widetilde{S},
\end{equation}
then $f \colon\Ort\rightarrow\Or$ is an \emph{orbifold covering map}. If there exists an orbifold covering map between $\Ort$ and $\Or$, $\widetilde{S}$ is simply-connected and $\tilde{\nu}\equiv 1$, then $\Ort$ a \emph{universal covering orbifold} of $\Or$ and $f$ is a \textit{universal covering map}.
\end{defn}
We note that an orbifold covering map needs not be 
a covering map, in the usual sense, %
between the underlying surfaces. In fact, that will be the most frequent case for us.
\begin{observation}[Lifts of covering maps {\cite[Lemma E.2]{milnor_book}}]\label{obs_covering_lifting} Let $\Ort, \Or$ be a pair of orbifolds with universal covering orbifolds. Then $f : \Ort \rightarrow \Or$ is an orbifold covering map if and only if it lifts to a conformal isomorphism between the universal covering orbifolds.
\end{observation}

\begin{remark}
With slight abuse of notation, we will sometimes write $z\in \Or$ to indicate that $z$ belongs to the underlying surface of $\Or$. Similarly, given a holomorphic map $f$ between orbifolds, we also denote by $f$ the holomorphic map between their underlying surfaces, and vice-versa.
\end{remark}

As a generalization of the Uniformization theorem for Riemann surfaces, with only two exceptions, every Riemann orbifold has a universal covering orbifold: 
\begin{thm}[Uniformization of Riemann orbifolds]\label{thm_uniform}
Let $\Or=(S,\nu)$ be a Riemann orbifold for which $S$ is connected. Then $\Or$ has no universal 
covering orbifold if and only if $\Or$ is isomorphic to $\widehat{\C}$ 
with signature $(l)$ or $(l,k)$, where $l\neq k$. In all other cases the 
universal cover is unique up to a conformal isomorphism over the surface $S$, and given by either $\widehat{\C}$, $\C$ or $\D$. In particular, if $S\subsetneq \C$ and $\#(\widehat{\C} \setminus S)>2$, then $\Or$ is covered by $\D$.
\end{thm}

In analogy to Riemann surfaces, we call an orbifold $\Or$ \emph{elliptic, parabolic} or \emph{hyperbolic} if all of its connected components are covered by $\widehat{\C}, \C$ or $\D$ respectively. A more detailed version of this theorem can be found in \cite[Theorem A2]{mcmullen1994complex}. For the proof in the more general case see \cite[Proposition 13.2.4]{thurstonGeom}.

\begin{discussion}[Orbifold metric]
Theorem \ref{thm_uniform} allows us to induce a metric on those orbifolds that have a universal cover as the pushforward of the spherical, Euclidean or hyperbolic metric of their universal cover. More precisely, let $\Or=(S, \nu)$ be an orbifold that has universal covering surface $C \in \{\C, \widehat{\C}, \D\}$, and let $\rho_C(z)\vert dz\vert$ be a complete conformal metric on $C$. By pushing forward this metric by an orbifold covering map, we obtain a Riemannian metric on $\Or$, that we denote by $\rho_{\Or}(w)\vert dw\vert$ and call the \emph{orbifold metric} of $\Or$. If $C\in \{\D, \widehat{\C}\}$, this metric is uniquely determined by normalizing the curvature to $\pm 1$, and for $C=\C$ the metric is well-defined up to a positive scalar multiple. The orbifold metric on $\Or$ determines a metric in the surface $S$ with singularities at the ramified points of $\Or$. More precisely, if $\nu(w_0)=m>1$ for some $w_0 \in S$, then $\rho_{\Or}(w)\vert dw\vert$ has a singularity of the type $\vert w-w_0\vert^{(1-m)/m}$ near $w_0$ in $S$. We then say that $w_0$ is a \textit{cone point}.
\end{discussion}

\begin{remark}[Cone points versus punctures]
There is an advantage to defining an orbifold metric on $S\subsetneq \C$ for which $w_0$ is a cone point over inducing a hyperbolic metric in the punctured surface $S\setminus \{w_0\}$. Even if both of the corresponding densities tend to infinity as we approach $w_0$, contrary to what happens when $w_0$ is a puncture, the orbifold distance from a point of $S$ to the cone point $w_0$ is finite, since $w_0$ is part of the surface. See \cite[pp. 210-211]{milnor_book}, as well as Proposition~\ref{prop_circle} for examples where estimates are computed.
\end{remark}

\begin{remark}[Metrics equivalence] If $\Or=(S, \nu)$, with $S\subset\C$, is an orbifold that admits an orbifold metric in the sense above, then the corresponding induced metric in $S$ is topologically equivalent to the Euclidean metric in $S$. That is, both metrics generate the same topology on $S$. We will use this fact without further comment.
\end{remark}
Let $\Or=(S,\nu)$ be an orbifold with $S \subsetneq \C$ that admits an orbifold metric $\rho_{\Or}(w)\vert dw\vert$. This metric induces an \textit{$\Or$-distance} $d_\Or(x,y)$ between points $x,y \in S$ in the following way. We join $x$ to $y$ by a rectifiable curve $\gamma$ in $S$, and define
the $\Or$-length $\ell_\Or(\gamma)$ of $\gamma$ by
$$\ell_\Or(\gamma) \defeq \int_\gamma \rho_{\Or}(w)\vert dw\vert.$$
Note that the integral is well-defined, since the set of ramified points in $\gamma$, and thus singularities of $\rho_{\Or}$, is finite. (See for example \cite[A.1 and A.10]{bonkexpanding} for more details on conformal and orbifold metrics.) Finally, we set
$$d_\Or(x, y) \defeq \inf\{\ell_\Or(\gamma): \gamma \text{ is a rectifiable curve in } S \text{ joining } x \text{ and } y \}.$$
In particular, for any two subsets $A,B\subset S$, which may be singletons, we denote
$$d_\Or(A,B) \defeq \inf \left\{d_\Or(x,y) : x\in A, y\in B\right\}.$$
The so-called Schwarz lemma or Pick's theorem for hyperbolic surfaces \cite[Theorem~6.4]{beardon_minda} generalizes to hyperbolic orbifolds in the following:
\begin{thm}[Orbifold Pick's theorem] \label{pick}
A holomorphic map between two hyperbolic orbifolds can never 
increase distances as measured in the hyperbolic orbifold metrics. Distances are strictly decreased, unless the map is a covering map, in which case it is a local isometry.
\end{thm}
See \cite[Proposition 17.4]{thurston_2} or \cite[Theorem A.3]{mcmullen1994complex} for more details. Recall that as a consequence of Pick's theorem for hyperbolic surfaces, if $U$ and $V$ are hyperbolic domains with $V \subset U$, the inclusion from $V$ into $U$ is contracting, and so, it holds for their hyperbolic densities $\rho_U$ and $\rho_V$ that $\rho_V(z)\geq \rho_U(z)$ for all $z \in V$. Theorem \ref{pick} has analogous implications, that we shall use. 

\begin{remark}
From now on, we use the notation $\Ort \hookrightarrow \Or$ to indicate the inclusion map between $\Ort$ and $\Or$. We are implicitly stating that such map is well-defined, and in particular, the underlying surface of $\Ort$ is contained in the underlying surface of $\Or$.  
\end{remark}
\begin{cor}[Comparison of orbifold densities]\label{cor_pick}
Let $\Ort$ and $\Or$ be hyperbolic orbifolds for which the inclusion $\Ort \hookrightarrow \Or$ is holomorphic. If $\rho_{\Ort}$ and $\rho_{\Or}$ are the respective densities of their orbifold metrics, then $\rho_{\Ort}(z)\geq \rho_{\Or}(z)$ for all unramified $z \in \Ort$, with strict inequality when the inclusion map is not an orbifold covering map.
\end{cor}

\begin{observation}[Relative densities]\label{obs_pick}
Note that if $\Ort \hookrightarrow \Or$ is holomorphic, then all ramified points of $\Or$ are also ramified points of $\Ort$, and so the quotient $\rho_\Ort(z)/\rho_\Or(z)$ is well-defined for all unramified $z\in \Ort$. Moreover, if $f\colon \Ort \rightarrow \Or$ is a covering map, then by Theorem \ref{pick} it is a local isometry, and so $\rho_\Ort(z)=\vert f'(z)\vert\rho_{\Or}(f(z))$ for all $z\in \Ort\cap \Or$. 
\end{observation}
\noindent By the previous observation and Corollary \ref{cor_pick}, the following holds:
\begin{cor}[Lower bound on hyperbolic derivative]\label{cor_derivative_covering}
Let $f\colon \Ort \rightarrow \Or$ be a covering map between hyperbolic orbifolds for which the inclusion $\Ort \hookrightarrow \Or$ is holomorphic but not a covering. Let $\rho_{\Ort}$ and $\rho_{\Or}$ be the respective densities of their orbifold metrics. Then, for all unramified $z\in \Ort$,
\begin{equation*}
\Vert \Deriv f(z)\Vert_{\Or}=\frac{\vert f'(z)\vert\rho_{\Or}(f(z))}{\rho_{\Or}(z)} =\frac{\rho_\Ort(z)}{\rho_\Or(z)}>1.
\end{equation*}
\end{cor}

\section{Hyperbolic orbifold metrics}\label{sec_orb_estimates}
In the first part of this section we study the relation between the densities of the metrics of two hyperbolic orbifolds whenever one of them is holomorphically embedded in the other. More specifically, let $\Ort\defeq(\widetilde{S},\tilde{\nu})$ and $\Or=(S,\nu)$ be hyperbolic orbifolds such that the inclusion $\Ort\hookrightarrow\Or$ is holomorphic. Note that, in particular, we are assuming that $\widetilde{S}\subseteq S$. Then, recall that by Corollary \ref{cor_pick}, it holds that $\rho_\Ort(z) \geq \rho_\Or(z)$ for all unramified $z\in \Ort$. The intuition behind this fact is the following: since a hyperbolic orbifold metric is defined as a pushforward of the hyperbolic metric in $\D$ with singularities at ramified points, its density tends to infinity both when approaching ramified points, and when tending to the boundary of the underlying surface of the orbifold. Moreover, if $w_0$ is a ramified point, then the density function is of the form $\vert w-w_0\vert^{(1-m)/m}$ near it, where $m$ is its ramification value. Note that for a fixed $w_0$, as $m$ increases, the density function tends ``faster'' to infinity when we approach $w_0$. Hence, since $\Ort\hookrightarrow\Or$ being holomorphic implies that $\widetilde{S}\subseteq S$ and $\tilde{\nu}(z)\geq \nu(z)$ for all unramified $z\in \Ort$, the desired inequality on their densities follows. This motivates the definition of the following set.
\begin{defn}[Boundary of $\Ort$ in $\Or$]\label{def_setD}
Given a pair of orbifolds $\Ort=(\widetilde{S}, \tilde{\nu})$ and $\Or=(S, \nu)$ such that the inclusion $\Ort \hookrightarrow \Or$ is holomorphic, we define the \textit{boundary of $\Ort$ in $\Or$} as the set
$$\mathbb{B}^\Or_\Ort\defeq \partial \widetilde{S} \: \cup \: \big\{ z \in \widetilde{S} : \tilde{\nu}(z) > \nu(z) \big\}.$$
\end{defn}

\begin{remark}
If $\Ort \hookrightarrow \Or$ is holomorphic, then $\mathbb{B}^\Or_\Ort \neq \emptyset$ if and only if $\widetilde{S}\subsetneq S$, or $S=\widetilde{S}$ and the inclusion is not an orbifold covering map. Moreover, in this case, the quotient $\rho_{\Ort}(z)/\rho_{\Or}(z)$ is well-defined for all unramified $z\in \Ort$; see Observation \ref{obs_pick}.
\end{remark}

Under the conditions of Definition \ref{def_setD}, Theorem \ref{prop_expansion_Intro} provides bounds for the quotient of densities in terms of the $\Or$-distance between a point $z\in \widetilde{S}$ and the set $\mathbb{B}^\Or_\Ort$. This is inspired by \cite[Proposition 3.4]{lasseAbscense}, where an analogous result is shown to hold for hyperbolic Riemann surfaces. Let us restate Theorem \ref{prop_expansion_Intro} in a more precise version:

\begin{thmEstimates}[Relative densities of hyperbolic orbifolds]\label{prop_expansion}
Let $\Ort\defeq(\widetilde{S},\tilde{\nu})$ and $\Or\defeq(S,\nu)$ be hyperbolic orbifolds such that the inclusion $\Ort \hookrightarrow \Or$ is holomorphic. Let $z\in \Ort$ be unramified and suppose that $R\defeq d_{\Or}(z,\mathbb{B}^\Or_\Ort)<\infty$. Then,
\begin{equation}\label{eq_prop_expansion}
1< \frac{e^R}{\sqrt{e^{2R}-1}} \leq \frac{\rho_{\Ort}(z)}{\rho_{\Or}(z)}\leq 1+ \frac{2}{e^R-1}.
\end{equation}
\end{thmEstimates}
\begin{remark}
The exact dependence of the bounds on $R$ is not relevant for our purposes, but instead, we are interested in the fact that the quotient of densities depends only on $R$ and is bounded away from $1$; see Figure \ref{fig:lambda}. Still, we point out that the proof will show that the bounds are sharp, in the sense that they can be attained. 
\end{remark}
\begin{figure}[h]
\begin{center}
\includegraphics[width=7cm]{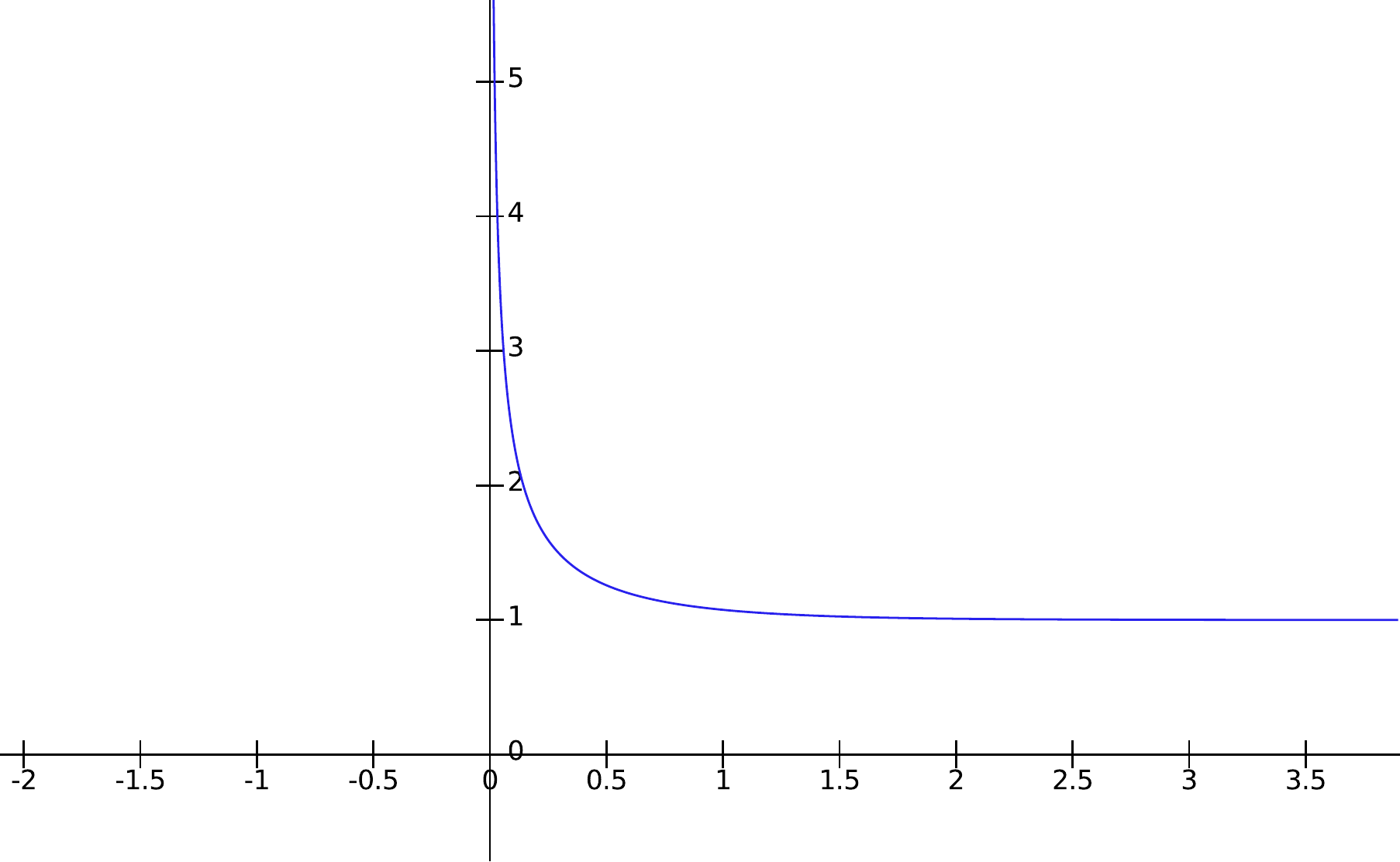}
\caption{Plot of the function $\Lambda(R)\defeq\e^R/\sqrt{e^{2R}-1}$, that provides a lower bound for the quotient of densities in the setting of Theorem \ref{prop_expansion_Intro}. Observe that $\Lambda(R)>1$ for all $R>0$.}
\label{fig:lambda}
\end{center}
\end{figure}
\begin{proof}[Proof of Theorem \ref{prop_expansion_Intro}] 
We can assume without loss of generality that the surfaces $\widetilde{S}$ and $S$ are both connected, since otherwise the same argument applies component-wise. For the point $z$ fixed in the statement of the proposition, let $\pi \colon \D\to \Or$ be a covering map with $\pi(0)=z$. In particular, by definition of orbifold covering map, for any $x\in \D$,
\begin{equation}\label{eq_degD}
\nu(\pi(x))=\deg(\pi, x)\cdot\nu_{\D}(x)=\deg(\pi, x),
\end{equation}
as $\nu_{\D}\equiv 1$ by definition. Since by assumption $z\in \widetilde{S} \subseteq S$, there exists a connected component of $\pi^{-1}(\widetilde{S})$ that contains the point $0$. We shall denote this component by $\widehat{D}$. Since $\Ort \hookrightarrow \Or$ is holomorphic, $\nu(w)$ divides $\tilde{\nu}(w)$ for all $w \in \Ort$, and so, using \eqref{eq_degD}, we can define a ramification map $\hat{\nu} \colon \widehat{D} \rightarrow \N_{\geq 1}$ as 
\begin{equation}\label{degVt}
\hat{\nu}(x)\defeq \dfrac{\tilde{\nu}(\pi(x))}{\deg(\pi,x)}.
\end{equation}
Note that by \eqref{eq_degD} and since $\Ort\hookrightarrow \Or$ is holomorphic, for each $x\in \widehat{D}$, $\hat{\nu}(x)>1$ if and only if $\pi(x) \in \big\{w\in \Ort : \tilde{\nu}(w) > \nu(w) \big\}.$ Since $\big\{w\in \Ort : \tilde{\nu}(w) > \nu(w) \big\}$ is a discrete set, as $\tilde{\nu}$ is the ramification map of an orbifold, $\hat{\Or}\defeq(\widehat{D}, \hat{\nu})$ is also a Riemann orbifold. Observe that by definition, the restriction $\pi\vert_{ \widehat{D}} \colon\hat{\Or} \rightarrow \tilde{\Or} $ is an orbifold covering map.

Since by assumption $d_{\Or}(z,\mathbb{B}^\Or_\Ort)=R$, by definition of the set $\mathbb{B}^\Or_\Ort$, there must exist at least one point $z_2\in \mathbb{B}^\Or_\Ort$ such that $d_\Or(z, z_2)= R$. In particular, $z_2\in S$. Let us connect $z_2$ to $z$ by a geodesic (in the metric of $\Or$) of length $R$. By lifting this geodesic to the unit disc using the map $\pi$, we see using Theorem \ref{pick} that there exists $w\in \text{cl}(\widehat{D})$ such that $\dist_{\D}(0,w)=R$. By pre-composing with a rotation, we can assume that $w$ is a positive real number. We recall that the densities of the hyperbolic metric on $\D_r$ for some $r\in \R^+$ and on $\D^\ast$, are respectively given by
\begin{equation}\label{eq_formulae_densities}
\rho_{\D_r}(x)=\frac{2}{r\left(1-\vert x\vert^{2}/r^2\right)} \quad \text{and} \quad \rho_{\D^\ast}(x)=\frac{1}{\vert x\vert \cdot \vert \log \vert x \vert \vert}.
\end{equation}
Since $\pi$ is a covering map, by Theorem \ref{pick}, $d_\D(x,y)\geq d_\Or(\pi(x), \pi(y))$ for all $x,y \in \D$. Hence, by the choice of $w$, the disc (in the hyperbolic metric on $\D$) of radius $R$ centred at the origin is contained in $\widehat{D}$, and particular is a Euclidean disc of radius $w$. Moreover, by definition of the constant $R$, $\hat{\nu}(z)=1$ for all $z \in \D_{w}\subset \widehat{D}$, and thus, if we regard $\D_{w}$ as a hyperbolic orbifold with ramification map constant and equal to $1$, the inclusion $\D_{w}\hookrightarrow \hat{\Or}$ is holomorphic. In particular, by Corollary \ref{cor_pick}, $\rho_{\hat{\Or}}(x)\leq \rho_{\D_{w}}(x)$ for all $x \in \D_{w}$. Thus, using Theorem \ref{pick}, \eqref{eq_formulae_densities} and recalling that $\pi(0)=z$,
\begin{equation}\label{eq_upperbound}
\frac{\rho_{\Ort}(z)}{\rho_{\Or}(z)} = \frac{ \vert \pi'(0)\vert \cdot \rho_{\Ort}(\pi(0))}{ \vert \pi'(0)\vert\cdot \rho_{\Or}(\pi(0))}=\frac{\rho_{\hat{\Or}}(0)}{\rho_{\D}(0)} \leq
\frac{\rho_{\D_{w}}(0)}{\rho_{\D}(0)} = \frac{1}{w}.
\end{equation}

We have obtained an upper bound for the relative densities at $z$ in terms of the value $w$. In order to get a lower bound, we divide the proof into two cases depending on whether $\hat{\nu}(w)=1$ or $\hat{\nu}(w)> 1$. In the first case, $z_2=\pi(w)\in \partial \widetilde{S}$, and so $w\in \partial \widehat{D}$. In particular, $\widehat{D} \subset \D \setminus\{w\}$, and so the inclusion $\hat{\Or} \hookrightarrow (\D \setminus\{w\}, \rho_{\D \setminus\{w\}})$ is holomorphic, where $\rho_{\D \setminus\{w\}}$ is the constant function equal to $1$. Therefore, by Corollary \ref{cor_pick}, $\rho_{\hat{\Or}}(x)\geq \rho_{{\D \setminus\{w\}}}(x)$ for all unramified $x \in \hat{\Or}$. Consider the Möbius transformation $T\colon \D\rightarrow \D$ given by $T(x)\defeq \frac{x-w}{wx-1}$, which in particular satisfies $T(w)=0$ and $T(0)=w$. The restriction $T\vert_{ \D \setminus\{w\}}$ is a covering map for the orbifold with underlying surface $\D^\ast$ and ramification map constant equal to one. Then, using Theorem \ref{pick}, \eqref{eq_formulae_densities} and \eqref{eq_upperbound},
\begin{equation}\label{eq_lower3}
\frac{\rho_{\Ort}(z)}{\rho_{\Or}(z)} =\frac{\rho_{\hat{\Or}}(0)}{\rho_{\D}(0)} \geq\frac{\rho_{{\D \setminus\{w\}}}(0)}{\rho_{\D}(0)}=\frac{ \vert T'(0)\vert \cdot \rho_{{\D^\ast}}(T(0))}{ \vert T'(0)\vert\cdot \rho_{\D}(T(0))}= \frac{\rho_{\D^{\ast}}(w)}{\rho_{\D}(w)}=\frac{1-w^{2}}{2w \vert\log w \vert }.
\end{equation}

For the second case, that is, whenever $k\defeq\hat{\nu}(w)\geq 2$, we define the orbifold $\Or^k_{w}\defeq (\D, \mu)$ with $\mu(w)=k$ and $\mu\equiv 1$ elsewhere. Then, the inclusion $\hat{\Or} \hookrightarrow \Or^k_{w}$ is holomorphic, and so by Corollary \ref{cor_pick}, $\rho_{\hat{\Or}}(x)\geq \rho_{\Or^k_{w}}(x)$ for all $x \in \hat{D}$. Thus, using \eqref{eq_upperbound},
\begin{equation}\label{eq_lower1}
\frac{\rho_{\Ort}(z)}{\rho_{\Or}(z)}=\frac{\rho_{\hat{\Or}}(0)}{\rho_{\D}(0)} \geq \frac{\rho_{\Or^k_{w}}(0)}{\rho_{\D}(0)}.
\end{equation}
Let $\Or^k_0$ be the orbifold with underlying surface the unit disc and signature $(k)$, being $0$ its only ramified point. Let $f \colon \D\rightarrow \Or^k_{0}$ be the covering map given by $f(x)=x^k$. Then, $T\circ f \colon \D \rightarrow \Or^k_{w}$ is an orbifold covering map, and thus, by Theorem \ref{pick},	
$$\vert f'(x) \vert \rho_{\Or^k_{0}}(f(x))= \rho_\D(x) = \left\vert T'(f(x))\right\vert\cdot \left\vert f'(x) \right\vert \rho_{\Or^k_{w}}(T(f(x))).$$ 
Hence, if we choose any $x \in \D$ such that $f(x)=w$, using that $T(w)=0$, we get that $\rho_{\Or^k_{0}}(w)= \vert T'(w) \vert \rho_{\Or^k_{w}}(0)$. Arguing similarly, $\rho_\D(w)=\vert T'(w) \vert \rho_\D(0)$. Thus, substituting in \eqref{eq_lower1}, 
\begin{equation}\label{eq_anotherlower}
\frac{\rho_{\Ort}(z)}{\rho_{\Or}(z)} \geq \frac{\rho_{\Or^k_{w}}(0)}{\rho_{\D}(0)}= \frac{\rho_{\Or^k_{0}}(w)}{\rho_{\D}(w)}.
\end{equation}

We aim to get a lower bound for $\rho_{\Or^k_{0}}(w)/\rho_{\D}(w)$ independent of the value $k$. We can compute the density of the induced metric in $\Or^k_{0}$ using that $\rho_{\D}(x)=\vert f'(x)\vert \rho_{\Or^k_{0}}(f(x))$ and \eqref{eq_formulae_densities}. Since $f(x)=x^{k}=u$ implies $x^{k-1}=u^{\frac{k-1}{k}}$, we get that for each $u \in \D$,
\begin{equation}\label{eq_density}
\rho_{\Or^k_{0}}(u)=\frac{2}{k\vert u\vert^{\frac{k-1}{k}}(1-\vert u\vert^{\frac{2}{k}})}.
\end{equation}
Thus, if we make the change of variables $q\defeq 1/k$, $r\defeq1/w$, we are aiming to find a lower bound, independent of $q$, for 
\begin{equation}\label{eq_lower2}
\frac{\rho_{\Or^k_{0}}(w)}{\rho_{\D}(w)}=\frac{1-w^{2}}{k w^{\frac{k-1}{k}}(1-w^{\frac{2}{k}})}= \frac{ q(1-r^{-2})r}{r^q-r^{-q}}, \quad \text{ where }\quad q\in (0, 1/2] \quad \text{and} \quad r> 1.
\end{equation}

Observe that for each fixed value of $r$, the last quotient above is strictly decreasing in $q$. This can be seen by considering for each $r> 1$ the functions $f_r\colon (0, 1/2] \rightarrow \R$ given by
$$f_r(q)\defeq \frac{q}{r^q-r^{-q}}=\frac{q}{\sinh(q \log r)}=\frac{s}{\log r\sinh(s)},$$
where we have made the change of variables $s=q\log r$. Let $h(s) \defeq s/\sinh(s)$ and note that $h'(s)=(\sinh(s) - s \cosh(s)) / \sinh^2(s)$ is always negative, as $\tanh(s) < s$ when $s$ is positive. Thus, the same holds for $f_r'(q)$ and so, each function $f_r$ is strictly decreasing in $q$. Substituting in \eqref{eq_lower2},
\begin{equation}\label{eq_quotient_R}
\frac{\rho_{\Or^k_{0}}(w)}{\rho_{\D}(w)} \geq \frac{(1-r^{-2})r}{2(r^{1/2}-r^{-1/2})}=\frac{r^{-1/2} (1+r^{-1})}{2} =\frac{1+w}{2\sqrt{w}} \quad \text{ for each } \quad w<1.
\end{equation}
Thus, putting together equations \eqref{eq_lower3}, \eqref{eq_anotherlower} and	\eqref{eq_quotient_R}, we get that for the point $z$ in the statement,
\begin{equation} \label{eq_lowerfinal}
\frac{\rho_{\Ort}(z)}{\rho_{\Or}(z)}\geq \min \left \{\frac{1-w^{2}}{2w \vert \log w \vert }, \frac{1+w}{2\sqrt{w}} \right\}=\frac{1+w}{2\sqrt{w}}>1.
\end{equation}	

Finally, \eqref{eq_prop_expansion} is obtained recalling that the hyperbolic distance between $0$ and any point $z \in \D$ is given by $\log \frac{1+\vert z\vert}{1-\vert z\vert}$. In our case, $d_{\D}(0,w)=\log \frac{1+w}{1-w}=R$, and so $w=\frac{\e^{R}-1}{\e^{R}+1}$. Substituting accordingly in equations \eqref{eq_upperbound} and \eqref{eq_lowerfinal}, the desired bounds are achieved.
\end{proof}

The second goal of this section is to prove Theorem \ref{thm_intro_orb}. In order to achieve this, we will first prove in Theorem \ref{thm_cont_or}, for orbifolds with the same number of ramified points, all with the same ramification value, that if these ramified points are ``continuously perturbed'', the orbifold metric of the new orbifold is a ``continuous perturbation'' of the metric of the original one. It is possible that these results have appeared before in the literature of orbifolds, but since a reference has not been located, we present proofs that use quasiconformal maps. We refer to \cite{lehto_quasiconformal, vuorinen2006conformal} for definitions.

We start by fixing the type of orbifolds that we shall consider. Namely, those for which their ramified points are at least at a certain (given) Euclidean distance from each other.
\begin{defn}[Orbifolds associated to vectors] \label{def_orb_cont} Given a compact subset $A$ of a Jordan domain $U\subsetneq \C$ and constants $N,X\in \N_{\geq 1}$ and $r>0$, denote
\begin{equation}
\mathcal{T}_r^N(A)\defeq \left\{(w_1, \ldots, w_N)\in A^N : \vert w_i-w_j \vert \geq r \text{ for all } i\neq j \right\}.
\end{equation}
Each $\bm{w}=(w_1, \ldots, w_N)\in \mathcal{T}_r^N(A)$ has an \textit{associated orbifold} $\Or^X_{\bm{w}}\defeq(U, \nu_{\!_{\bm{w}}})$, with 
\begin{align}\label{eq_ram_thm}
\nu_{\!_{\bm{w}}}(z)& \defeq \renewcommand{\arraystretch}{1.5}\left\{\begin{array}{@{}l@{\quad}l@{}}
X & \text{if } z=w_i \text{ for some } 1\leq i \leq N, \\
1 & \text{otherwise.}
\end{array}\right.\kern-\nulldelimiterspace
\end{align}
\end{defn}
\begin{remark}
By Theorem \ref{thm_uniform}, for any orbifold $\Or^X_{\bm{w}}$ as in Definition \ref{def_orb_cont}, since $U\subsetneq \C$, its universal cover is $\D$, and so the distance function $d_{\Or^X_{\bm{w}}}$ is well-defined. Moreover, it is easy to see that $\mathcal{T}_r^N(A)$ is compact.
\end{remark} 

In the following theorem, we see that continuous perturbations of a vector $\bm{w}\in\mathcal{T}_r^N(A)$ lead to continuous perturbations on the distance function $d_{\Or_{\bm{w}}}$ of its associated orbifold. Compare to \cite[Theorem 4.2]{helenaSemi} for a similar argument when a single ramified point of an orbifold is perturbed.
\begin{thm}[Continuity of orbifold metrics under perturbations]\label{thm_cont_or} Let $A$ be a compact subset of a Jordan domain $U$. Let $N,X\in \N_{\geq 1}$ and $r>0$. Then, the function $h​ \colon A^2 \times \mathcal{T}_r^N(A) \to \R$ given by 
$$h(p,q,w_1,\ldots, w_{N})\defeq d_{\Or_{\bm{w}}}(p,q)$$ is continuous, where $\bm{w}\defeq (w_1,\ldots, w_{N})\in \mathcal{T}_r^N(A)$ and $\Or_{\bm{w}}\defeq \Or^X_{\bm{w}} $ is its associated orbifold.
\end{thm}

\begin{proof}
Since the domain of the function $h$ is a metric space, the notions of continuity and sequential continuity for $h$ are equivalent. Thus, we will prove the theorem by showing that for a fixed but arbitrary $\bm{x}\defeq(p,q,w_1,\ldots, w_N) \in A^2 \times \mathcal{T}_r^N(A)$, if $\lbrace \bm{x_k}\defeq(p^k,q^k,w^k_1,\ldots, w^k_N)\rbrace_{k\geq 1}$ is a sequence of points in $A^2 \times \mathcal{T}_r^N(A)$ such that $\bm{x_k} \rightarrow \bm{x}$ as $ k\to \infty$, then $h(\bm{x_k})\to h(\bm{x})$. That is, if $\bm{w}\defeq(w_1,\ldots, w_N)$ and for each $k\geq 1$, $\bm{w_k}\defeq(w^k_1,\ldots, w^k_N)$, then we will prove continuity of $h$ by showing that
\begin{equation}\label{eq_goalcont}
d_{\Or_{\bm{w_k}}}(p^k,q^k) \to d_{\Or_{\bm{w}}}(p,q) \quad \text{ as } \quad k \rightarrow \infty.
\end{equation}

By translation, we may assume without loss of generality that $0 \in A$ and $w_j\neq 0$ for all $1\leq j\leq N$. Then, since $\bm{w}\in \mathcal{T}_r^N(A)$, we can choose $\epsilon<r/2$ so that all disks in the set 
$$\{\D_\epsilon\} \cup \{\D_{\epsilon}(w_j) : 1\leq j\leq N\}$$
are pairwise disjoint and contained in $U$. Since by assumption $\bm{x_k}\rightarrow \bm{x}$ as $k\rightarrow \infty$, there exists $K>0$ such that $w^k_j \subset \D_{\epsilon}(w_j)$ for all $k\geq K$ and $1\leq j\leq N$. Moreover, for each $k\geq K$ and $1\leq j\leq N$, we define a quasiconformal map $\phi_j^k \colon \D_{\epsilon}(w_j) \rightarrow\D_{\epsilon}(w_j)$ that satisfies $\phi_j^k(w^k_j)=w_j$. With that aim, let $H^k_j\colon \D_{\epsilon}(w_j)\rightarrow \H $ be the unique Riemann map such that $H^k_j(w_j)=i$ and so that $H^k_j(w^k_j)$ lies in the positive imaginary axis. Recall that $\H$ denotes the upper half-plane. In particular, $H^k_j(w^k_j)=h^k_j i$, where $h^k_j\defeq\e^{d_j(w^k_j,w_j)}$ and $d_j$ denotes the hyperbolic metric in $\D_\epsilon(w_j)$.
Define $L^k_j:\H\to\H$ as $L^k_j(z)\defeq \Rea(z) + h^k_j \Ima(z) i$. Note that $L^k_j$ is a $h^k_j$-quasiconformal self-map of $\H$. We then define
\begin{align*}
\phi_j^k \colon \D_{\epsilon}(w_j) \rightarrow\D_{\epsilon}(w_j) \quad \text{ as } \quad \phi_j^k\defeq \left(H^k_j\right)^{-1}\circ L^k_j \circ H^k_j. 
\end{align*}
If follows from the definition of the functions involved that $\phi_j^k$ extends continuously to $\partial \D_{\epsilon}(w_j)$ as the identity map. Hence, the map $\phi_k\colon U \rightarrow U$ given by
\begin{align}\label{eq_qc_map}
\phi_k(z)& \defeq \renewcommand{\arraystretch}{1.5}\left\{\begin{array}{@{}l@{\quad}l@{}}
\phi_j^k(z) & \text{if } z\in \D_{\epsilon}(w_j) \text{ for some } 1\leq j\leq N, \\
z & \text{otherwise,}
\end{array}\right.\kern-\nulldelimiterspace
\end{align}
is well-defined and continuous. In fact, $\phi_k$ is a $K(k)$-quasiconformal map, where $K(k)\defeq \max^N_{j=1}\{h^k_j\}$ (see for example \cite{RealFatoubook}), and moreover,
\begin{equation}\label{eq_ident0}
\phi_k \vert_ {\D_\epsilon} \equiv \id\vert_{\D_\epsilon} \quad \text{ for all } k\geq K.
\end{equation}

In particular, $K(k)\rightarrow 1$ and $\phi_k\rightarrow \id$ as $k$ tends to infinity. Let $\pi:\D\to \Or_{\bm{w}}$ and let $\pi_k:\D\to \Or_{\bm{w_k}}$ be orbifold covering maps, normalized such that it holds $\pi(0)=\pi_k(0)=0$ and $\arg(\pi'(0))=\arg(\pi_k'(0))$. Note that as an orbifold map, $\phi_k\colon \Or_{\bm{w_k}} \rightarrow \Or_{\bm{w}}$ is a homeomorphism that preserves ramified points, that is, 
\begin{equation}\label{eq_preserves_degrees}
\nu_{\bm{w}}(z)>1 \iff \nu_{\bm{w_k}}(\phi_k(z))>1.
\end{equation}
Thus, we can \textit{lift} $\phi_k$ to a homeomorphism $\Phi_k\colon \D\rightarrow \D$ such that
\begin{equation}\label{eq_commut_phi_k}
\pi \circ \Phi_k = \phi_k \circ \pi_k \quad \text{ and } \quad \Phi_k(0)=0.
\end{equation}
Note that $\Phi_k$ is also $K(k)$-quasiconformal, since both $\pi$ and $\pi_k$ are holomorphic and $\phi_k$ is a $K(k)$-quasiconformal map. Recall that for every $K$, the space of $K$-quasiconformal self-maps of the disc fixing zero is compact, see \cite[Corollary 4.4.3]{hubbard_teich}. In particular, since $K(k)\rightarrow 1$ as $k$ tends to infinity, the limit function of $\{\Phi_k\}_k$ is $K$-quasiconformal for every $K>1$; that is, it is conformal, and thus a Möbius transformation fixing zero. Note that by \eqref{eq_ident0} and \eqref{eq_commut_phi_k}, 
the maps $\Phi_k$ are all holomorphic on the same neighbourhood of $0$, and by our assumptions on the derivatives of $\pi, \pi_k$ and $\phi_k$ at $0$, we have that $\vert \Phi_k'(0)\vert\rightarrow 1$ as $k\rightarrow \infty$. Consequently, the sequence $\{\Phi_k\}_k$ converges locally uniformly to the identity as $k\rightarrow \infty$. Hence, by \eqref{eq_preserves_degrees} and \eqref{eq_commut_phi_k}, $\pi_k$ converges locally uniformly to $\pi$ as $k\rightarrow \infty$.

Recall that our goal is to prove \eqref{eq_goalcont}. Note that since $\{\pi_k\}_k$ and $\pi$ are orbifold covering maps, by Theorem \ref{pick} they are local isometries. Hence, instead of proving \eqref{eq_goalcont} using the orbifold metrics in $\Or_{\bm{w_k}}$ and $\Or_{\bm{w}}$, we will prove an analogue of \eqref{eq_goalcont} for preimages of the points $p,q,p^k,q^k$ under the covering maps $\{\pi_k\}_k$ and $\pi$. More precisely, let us choose $\delta$ small enough so that there exist respective connected components $V_p$ and $V_q$ of $\pi^{-1}(\D_\delta(p))$ and $\pi^{-1}(\D_\delta(q))$ containing a single preimage of $p$ and of $q$ respectively. That is, $\pi^{-1}(p)\cap V_p \eqdef\{\tilde{p}\}$ and $\pi^{-1}(q)\cap V_q \eqdef\{\tilde{q}\}$. In addition, for each $k\geq K$, let us consider the holomorphic functions $\pi_k^p: V_p \rightarrow U$ and $\pi_k^q: V_q \rightarrow U$ given by $\pi_k^p(z)\defeq \pi_k(z)-p$ and $\pi_k^q(z)\defeq \pi_k(z)-q$. Then, the sequences $\{\pi_k^p\}_k$ and $\{\pi_k^q\}_k$ converge uniformly in compact subsets to the functions $\pi\vert_{ V_p}-p$ and $\pi\vert_{ V_p}-q$, which have respectively unique zeros at $p$ and $q$. Then, by Hurwitz's theorem, for each $k$ large enough, there exist points $\{\tilde{p}_{\pi_k}\}\defeq \pi_k^{-1}(p) \cap V_p\cap \D_{\delta_1}(\tilde{p})$ and $\{\tilde{q}_{\pi_k}\}\defeq \pi_k^{-1}(q) \cap V_q\cap \D_{\delta_1}(\tilde{q})$ for some $\delta_1$ small enough. In particular, 
\begin{equation}\label{eq_Hur_1}
\tilde{p}_{\pi_k}\xrightarrow{ \: k\rightarrow \infty \: } \tilde{p} \quad \text{and} \quad \tilde{q}_{\pi_k} \xrightarrow{ \: k\rightarrow \infty \: } \tilde{q}.
\end{equation}
Note that for each $k\geq K$, $\tilde{p}_{\pi_k}$ is a preimage of $p$ under $\pi_k$, rather than a preimage of $p^k$ under $\pi_k$, and hence the proof is not concluded just yet. However, since by assumption $p^k\xrightarrow{ \: k\rightarrow \infty \: } p$ and $q^k \xrightarrow{ \: k\rightarrow \infty \: } q$, for every $k$ sufficiently large, $\pi_k^{-1}(p^k)\cap V_p \eqdef\{P_k\}$, $\pi_k^{-1}(q^k)\cap V_q \eqdef\{Q_k\}$, and in addition
\begin{equation}\label{eq_Hur_2}
\vert P_k- \tilde{p}_{\pi_k} \vert \xrightarrow{ \: k\rightarrow \infty \: } 0 \quad \text{and} \quad \vert Q_k- \tilde{q}_{\pi_k} \vert \xrightarrow{ \: k\rightarrow \infty \: } 0.
\end{equation}
Thus, as a combination of \eqref{eq_Hur_1} and \eqref{eq_Hur_2}, and using that for compact subsets of $\D$, the Euclidean and hyperbolic metrics are equivalent, we have that
$$d_\D(P_k, Q_k) \rightarrow d_\D(\tilde{p}, \tilde{q}) \quad \text{as} \quad k \rightarrow \infty, $$
which is equivalent to \eqref{eq_goalcont}, as we wanted to show.
\end{proof} 

Theorem \ref{thm_intro_orb} now becomes a consequence of the preceding one together with Theorem \ref{prop_expansion_Intro}. We restate it here for ease of exposition.
\begin{thmIntroOrb}[Distances are uniformly bounded across certain orbifolds] Given a compact subset $A$ of a Jordan domain $U$ and constants $r>0$ and $c, M\in \N_{\geq 1}$, there exists a constant $R\defeq R(U,A,r,c,M) >0$ such that for every orbifold $\Or$ with underlying surface $U$ and at most $M$ ramified points, each with ramification value smaller than or equal to $c$, and such that the Euclidean distance between any two of them is at least $r$, it holds that $$d_\Or(p,q)<R \quad \text{ for every } \quad p,q \in A.$$
\end{thmIntroOrb}
\begin{proof}

For each $N=1, \ldots, M$, we apply Theorem \ref{thm_cont_or} to the compact set $A$, the domain $U\supset A$ and constants $N$ and $X=c!$. Then, Theorem \ref{thm_cont_or} asserts that for each $N$, the function $h \colon A^2 \times \mathcal{T}_r^N(A) \to \R$ of its statement is continuous and defined on a compact set. Hence, for each $N$, there exists a constant $R_N$ such that for any orbifold $\Or_{\bm{w}}=(U, \nu_{\Or_{\bm{w}}})$ with $\bm{w}\in \mathcal{T}_r^N(A)$ and $\nu_{\Or_{\bm{w}}}$ as specified in \eqref{eq_ram_thm}, $d_{\Or_{\bm{w}}}(p,q)<R_N$ for all $p,q\in A$. Note that by Corollary \ref{cor_pick}, the same bound holds for any orbifold with $N$ ramified points in $A$ with ramification degrees between $1$ and $c$. This is because the inclusion map would be holomorphic as their ramification values divide $c!$, and the same argument applies if the orbifold has no ramified points in $A$. Let us define $\tilde{R}\defeq \max_{N\leq M}R_N$. Then, if $\hat{\Or}\defeq(U, \hat{\nu})$ is any orbifold with at most $M$ ramified points, any two at Euclidean distance at least $r$, all lying in $A$ and each of them with ramification value at most $c$, then
\begin{equation}\label{eq_cor_cont}
d_{\hat{\Or}}(p,q)<\tilde{R} \quad \text{ for all } \quad p,q \in A.
\end{equation}
Let us fix any orbifold $\Or\defeq (U,\nu)$ satisfying the hypotheses of the statement of this theorem. Moreover, let us fix $\hat{\Or}= (U,\hat{\nu})$ with $\hat{\nu} \equiv \nu\vert_{ A}$ in $A$ and $\hat{\nu}\equiv 1$ in $U \setminus A$, and note that \eqref{eq_cor_cont} holds for $\hat{\Or}$.

Let $W\defeq\{z\in U : d_{\hat{\Or}}(A, z)<\tilde{R}\}$ and define the orbifold $\Ort\defeq (W,\nu\vert_{ W})$, with $\nu\vert_{ W}$ being the restriction of $\nu$ to $W$. Observe that $A\Subset W$ and that the inclusions $\Ort\hookrightarrow \Or$ and $\Ort\hookrightarrow \hat{\Or}$ are holomorphic. In particular, the boundary of $\Ort$ in $\hat{\Or}$, denoted $\mathbb{B}_\Ort^{\hat{\Or}}$, consists of all ramified points of $\Ort$ lying in $W\setminus A$ together with $\partial W$. Then, by definition of $W$, for all $z\in A$, $d_{\hat{\Or}}(z,\mathbb{B}_\Ort^{\hat{\Or}})<\tilde{R}$, and by Theorem \ref{prop_expansion_Intro}, for all unramified $z\in A$, $\frac{\rho_{\Ort}(z)}{\rho_{\hat{\Or}}(z)}\geq 1+ \frac{2}{e^{\tilde{R}}-1} \eqdef K.$ Moreover, if $\gamma$ is a geodesic in the metric of $\hat{\Or}$ joining two points $p,q\in A$, again by the choice of $W$, $\gamma$ must be totally contained in $W$, and hence in $\widetilde{\Or}$. Thus, 
$$d_{\hat{\Or}}(p,q)=\int \vert \gamma'(t)\vert \rho_{\hat{\Or}}(\gamma(t))dt \geq \frac{1}{K}\int \vert \gamma'(t)\vert \rho_{\Ort}(\gamma(t))dt \geq \frac{1}{K} d_\Ort (p,q).$$ 
By this and by Corollary \ref{cor_pick}, for all $p,q \in A$ 
$$d_\Or(p,q) \leq d_\Ort(p,q) \leq K d_{\hat{\Or}}(p,q) \leq K \cdot \tilde{R} \eqdef R.$$
Since the constant $K$ does not depend on the domain $W$ but only on $\tilde{R}$, the statement follows.
\end{proof}

\section{Uniform expansion} \label{sec_orbifolds}
This section is devoted to the proof of Theorem \ref{thm_main_intro_Orb}: for each strongly postcritically separated function $f\in \B$, we define a pair of hyperbolic orbifolds $(\Ort, \Or)$ so that in particular their underlying surfaces contain $J(f)$ and so that $f\colon \Ort\rightarrow \Or$ is an orbifold covering map. In order to construct these orbifolds, we take into account Corollary \ref{cor_derivative_covering}. That is, a first step towards \textit{expansion} requires, in addition to the conditions above, that the inclusion $\Ort\hookrightarrow \Or$ is holomorphic. Then, the combination of the inclusion being holomorphic and $f$ being a covering map, (i.e. merging formulae \eqref{eq_holom_orb} and \eqref{eq_covering_orb}) implies that the ramification map $\nu$ of $\Or$ must satisfy
\begin{equation}\label{eq_merge_formulae}
\deg(f,z)\cdot \nu(z) \text{ divides } \nu(f(z)) \text{ for all } z\in \Ort.
\end{equation}
In other words, if $z \in \Ort$, then $\deg(f,p)\cdot \nu(p)$ divides $\nu(z)$ for all $p\in f^{-1}(z)$.

\begin{remark} Note that if $J(f)$ is in the underlying surfaces of $\Ort$ and $\Or$, then by \eqref{eq_merge_formulae}, all points in $P(f)\cap J(f)$ are ramified in $\Or$.
\end{remark}

In order to achieve our goal, we have followed Mihaljevi\'c-Brandt's strategy when proving the corresponding statement for strongly subhyperbolic transcendental maps; \cite[Propositions 3.2 and 3.4]{helenaSemi}. The underlying idea is essentially the same as that in Douady and Hubbard's work for subhyperbolic rational maps \cite[p. 22]{Orsaynotes} (see also \cite[\S 19]{milnor_book}): the ramification value of each point in $\Or$ is defined as a multiple of the local degrees of all points on its backward orbit; see \eqref{eq_ram}. In particular, with this definition, all postsingular points of $f$ are ramified. Unlike in the polynomial case, both for strongly subhyperbolic and postcritically separated maps, in addition to those in $P(f)$, more ramified points in $\Or$ are needed in order to guarantee expansion, i.e., to guarantee that the set $\mathbb{B}^\Or_\Ort$ from Definition \ref{def_setD} has ``enough points''. Thus, the set of ramified points of $\Or$ will consist of $P_J$ together with a repelling periodic cycle:
\begin{dfn&prop}[Dynamically associated orbifolds]
\label{prop_Or}
Let $f$ be a strongly postcritically separated map. Then there exist orbifolds $\Or\defeq(S,\nu)$ and $\Ort\defeq(\tilde{S},\tilde{\nu})$ with the following properties:
\begin{enumerate}[label=(\alph*)]
\item \label{item_a_assocorb} Either $S=\C =\tilde{S}$ or $\text{cl}(\tilde{S})\subset S=\C\setminus \overline{U}$, where $U$ is a finite union of bounded Jordan domains.
\item \label{item_b_assocorb} The set of ramified points of $\Or$ equals $P_J \cup B$, where $B$ is a periodic cycle in $J(f)\setminus P_J$.
\item \label{item_c_assocorb} $J(f)\subset \tilde{S}\subset S$ and $P_F \cap S= \emptyset$.
\item \label{item_d_assocorb} $\Or$ and $\Ort$ are hyperbolic orbifolds.
\item \label{item_e_assocorb} $f \colon\Ort\to\Or$ is an orbifold covering map and the inclusion $\Ort \hookrightarrow \Or$ is holomorphic.
\item \label{item_f_assocorb} There exists $p\in S \setminus P_J$ such that $\# \{f^{-1}(p) \cap \tilde{S}\}$ is infinite and $$\# \{z\in f^{-1}(p): \tilde{\nu}(z)\leq \nu(z) \}< \infty.$$
\end{enumerate}
We say that a pair $(\Ort,\Or)$ of Riemann orbifolds is \emph{dynamically associated to $f$} if $\Ort$ and $\Or$ satisfy \ref{item_a_assocorb}-\ref{item_f_assocorb}.
\end{dfn&prop}

\begin{proof}
If $F(f) =\emptyset$, then we define $S \defeq \C$. Otherwise, by Lemma \ref{lem_deff}, $P_F$ is contained in a finite union of attracting basins, and so, by Proposition \ref{prop_Jordan} we can find bounded Jordan domains $U_1, \ldots,U_n$ such that for $U\defeq\cup_{i=1}^{n}U_i$, it holds that $P_F\cup f(U)\Subset U \Subset F(f)$. We then define $S\defeq \C \setminus \overline{U}$. In particular, $S$ is connected and $J(f) \subset S$.
\begin{Claim}
There exists a periodic cycle, that we denote by $B$, contained in $J(f) \setminus P_J$.
\end{Claim}
\begin{subproof}
Note that as $P_J$ is forward invariant, the cycle any repelling periodic point in $J(f) \setminus P_J$ belongs to, is totally contained  $J(f) \setminus P_J$. Since $f$ is an entire transcendental function, $J(f)$ must contain non-degenerate continua \cite{Baker_normality}. Since $P_J$ is discrete, we can choose a bounded piece $\gamma$ of such a continuum, so that $\gamma$ and $P_J$ are at $\epsilon$-Hausdorff distance for some fixed $\epsilon>0$. Then, since $J(f)$ can be characterized as the closure of the set of repelling periodic points of $f$, \cite[Theorem 4]{Bergweilermerophormic}, each point in $\gamma$ is either a repelling periodic point, or an accumulation of those. In any case, we can find a repelling periodic point in the $\epsilon$-neighbourhood of $\gamma$, and so its cycle belongs to $J(f) \setminus P_J$.
\end{subproof}
We define the map $\nu \colon S\rightarrow \N^+$ as
\begin{align}\label{eq_ram}
\nu(z)& \defeq \renewcommand{\arraystretch}{1.5}\left\{\begin{array}{@{}l@{\quad}l@{}}
\lcm\lbrace \deg(f^m,w),\text{ where } f^m(w)=z \text{ for some } m\geq 1\rbrace & \text{if } z\notin B, \\
2 & \text{if } z \in B.
\end{array}\right.\kern-\nulldelimiterspace
\end{align}

Note that no critical point of $S$ belongs to a periodic cycle, since $P_F \subset U$ and by Lemma~\ref{lem_deff}, all periodic cycles in $J(f)$ are repelling. By this, Definition \ref{def_strongps}, and expanding the definition of local degree for an iterate of $f$, there exists a constant $C$ such that for any $w\in S$ and $m\geq 1$, 
\begin{equation}\label{eq_degree}
\deg(f^m,w)= \prod^{m}_{j=1}\deg(f,f^{j}(w)) \leq C.
\end{equation}
Therefore, $\nu(z)\leq \lcm \lbrace 1,2, \ldots, C \rbrace<\infty$ for all $z\in S$. Moreover, the map $\nu$ is defined in \eqref{eq_ram} such that $\nu(z)>1$ if and only if $z$ belongs to $P_J\cup B$. Hence, since $f$ is postcritically separated, $P_J$ is discrete, and thus $\Or\defeq(S,\nu)$ is a Riemann orbifold. In particular, \ref{item_b_assocorb} follows by construction.

The orbifold $\Or$ is hyperbolic: if $S\neq \C$, then this follows from Theorem \ref{thm_uniform}. If on the contrary $S=\C$, by \cite[Theorem A2]{mcmullen1994complex}, the only orbifolds such that $S=\C$ are either hyperbolic, or they are parabolic with signature $(n)$ or $(2,2)$. It is shown in \cite[Proof of Proposition 3.2]{helenaSemi} that 
\begin{itemize}[wide=0pt, leftmargin=\dimexpr\labelwidth + 2\labelsep\relax]
\item for any $n\geq 2$, each orbifold with underlying surface $\C$ and signature $(n)$ must contain an asymptotic value in $S$, and
\item the orbifold with surface $\C$ and signature $(2,2)$ can only occur for polynomials.
\end{itemize}
These two cases lead to contradictions with AV$(f)\cap S=\emptyset$ and $f$ being postcritically separated. Thus, $\Or$ must be hyperbolic. By definition of the map $\nu$, for every $z\in f^{-1}(S)$, $\deg(f,z)$ divides $\nu(f(z))$, and hence we can define
\begin{equation}\label{def_Otilde}
\tilde{S}\defeq f^{-1}(S) \quad \text{ and }\quad \tilde{\nu}(z)\colon\tilde{S}\to\N^+ \quad \text{ with } \quad \tilde{\nu}(z)\defeq \frac{\nu(f(z))}{\deg(f,z)}.
\end{equation}
Since the set of ramified points of $\Or$ is discrete, one can see, using for example the Identity Theorem, that the set $\{z\in\tilde{S} \text{ such that }\tilde{\nu}(z)>1\}$ is also discrete. Thus, $\Ort\defeq (\tilde{S}, \tilde{\nu})$ is an orbifold. By construction, $\text{AV}(f)\cap S=\emptyset$, and so the map $f \colon\tilde{S}\rightarrow S$ is a branched covering. Furthermore, for all $z\in \tilde{S}$, $\deg(f,z)\cdot\tilde{\nu}(z)=\nu(f(z))$, and hence $f \colon\Ort\rightarrow\Or$ is an orbifold covering map.

Recall that $f(U)\Subset U$ whenever $U\neq \emptyset$, which implies $\text{cl}(\tilde{S})\subset S$. Moreover, if $S=\C$, then $J(f)=\C$, and since $\text{AV}(f)\cap J(f)=\emptyset$ by assumption, $\tilde{S}=f^{-1}(\C)=\C$. Thus, \ref{item_a_assocorb} follows. Moreover, since $J(f)$ is a totally invariant set, $J(f)\subset \tilde{S}$, as stated in \ref{item_c_assocorb}. Let $z\in \tilde{S}$. The definition of $\nu$ together with \eqref{eq_degree} imply that $\nu(z)\cdot\deg(f,z)$ divides $\nu(f(z))$. In turn, by \eqref{def_Otilde}, $\nu(f(z))=\tilde{\nu}(z)\cdot\deg(f,z)$. Hence, $\nu(z)$ divides $\tilde{\nu}(z)$ and so the inclusion $\Ort\hookrightarrow\Or$ is a holomorphic map, proving statement \ref{item_e_assocorb}. Since in addition $\Or$ is hyperbolic, by Theorem~\ref{thm_uniform}, each connected component of $\Ort$ must be hyperbolic, and so $\Ort$ is a hyperbolic orbifold. Thus, statement \ref{item_d_assocorb} follows. We are only left to show item \ref{item_f_assocorb}. With that purpose, choose any $p\in B$. In particular $p\in S$, and so $f^{-1}(p) \subset \tilde{S}$. Moreover, since $\text{AV}(f)\cap J(f)=\emptyset$, by Picard's theorem, $\#f^{-1}(p)$ is infinite, see \cite[Theorem 1.14]{Schleicher_entire}. Since $p\in B\subset J(f) \setminus P_J$, $\deg(f,z)=1$ for all $z \in \text{Orb}^{-}(p)$ and in particular, for all $ z \in f^{-1}(p) \setminus B$, it holds that $\tilde{\nu}(z)=2 $ and $\nu(z)=1$. Consequently, \ref{item_f_assocorb} follows and the proof is concluded.
\end{proof}

Note that condition \ref{item_f_assocorb} in the previous proposition implies that for any pair $(\Ort, \Or)$ of orbifolds associated to $f$, the inclusion $\Ort \hookrightarrow \Or$ is not an orbifold covering map, and hence the set $\mathbb{B}^\Or_\Ort$ is non-empty. The next proposition tells us that when in addition $f\in \B$, the set $\mathbb{B}^\Or_\Ort$ contains a sequence of points whose moduli converge to infinity at a specific rate.
\begin{prop}[Unbounded sequence in $\mathbb{B}_\Ort^\Or$]
\label{cor_Kzi} Let $f\in \B$ be strongly postcritically separated and let $(\Ort, \Or)$ be a pair of orbifolds dynamically associated to $f$. Then, there exist a constant $N>1$ and an infinite sequence of points $\{z_i\}_{i\geq 0} \subset \mathbb{B}_\Ort^\Or$ such that $\vert z_i\vert<\vert z_{i+1}\vert\leq N\vert z_i\vert$ for all $i\geq 0$.
\end{prop}
\begin{proof}
Let $\Or=(S, \nu)$ and let $p\in S\setminus S(f)$ be the point in Proposition \ref{prop_Or} for which $\# \{z\in f^{-1}(p): \tilde{\nu}(z)\leq \nu(z) \}< \infty.$ That is, all but finitely many preimages of $p$ belong to $\mathbb{B}_\Ort^\Or$. Since $f\in \B$, we can find a Jordan domain $D$ such that $S(f)\subset D$ and $p\in \C \setminus \overline{D}$. Since $p$ is non-exceptional, each connected component of $f^{-1}(\C \setminus \overline{D})$, that is, each tract of $f$, contains infinitely many preimages of $p$, see \cite[\S2.4]{mio_thesis}. If $\{z_i\}_{i\geq 0}$ is the set of preimages of $p$ in one tract, it follows from estimates on the hyperbolic metric on simply connected domains that there exists a constant $N'>1$ such that $\vert z_{i_k}\vert<\vert z_{i_{k+1}}\vert\leq N'\vert z_{i_k}\vert$ for an infinite subsequence $\{z_{i_k}\}_{i_k\in \N}$. For details on this argument see \cite[Proof of Lemma 5.1]{lasseRidigity} or \cite[Proof of Proposition 3.4]{helena_landing}. Hence, since all but finitely many points of $\{z_i\}_i$ must belong to $\mathbb{B}_\Ort^\Or$, the statement follows.
\end{proof}

Note that Corollary \ref{cor_derivative_covering} applies to any pair of orbifolds $(\Ort, \Or)$ dynamically associated to $f$, and so $\Vert \Deriv f(z)\Vert_{\Or}=\rho_\Ort(z)/\rho_\Or(z)>1$ for all unramified $z\in \Ort$. We aim to prove Theorem \ref{thm_main_intro_Orb} for any such pair of associated orbifolds by finding a sharper uniform lower bound for $\rho_\Ort/\rho_\Or$ combining the following lemma with Theorem \ref{prop_expansion_Intro}. In turn, Lemma \ref{lem_annulus} is a consequence of Proposition \ref{cor_Kzi} together with Theorem \ref{thm_intro_orb} and item \ref{itemd_defsps} in the Definition \ref{def_strongps} of strongly postcritically separated maps:
\begin{lemma}[Distances within annuli are uniformly bounded] \label{lem_annulus} Suppose that $f$ is a strongly postcritically separated function with parameters $(c, \epsilon)$, and let $\Or=(S,\nu)$ and $\Ort=(\tilde{S}, \tilde{\nu})$ be a pair of orbifolds dynamically associated to $f$. Let us fix some constant $K>1$. Then, there exists a constant $R\defeq R(K)>0$ such that if $p,q\in \overline{A(t, Kt)}\subset A(t/K,t K^2) \subset S$ for some $t>0$, then $d_\Or(p,q) \leq R$. If in addition $f\in \B$, then for all $z\in \Ort$, $$d_{\Or}(z, \mathbb{B}^\Or_\Ort)\leq R.$$
\end{lemma}

\begin{proof}
By Proposition \ref{prop_Or}\ref{item_b_assocorb}, the set of ramified points of $\Or$ equals $P_J \cup B$, where $B$ is a periodic cycle in $J(f)\setminus P_J$. Thus, since $f$ is strongly postcritically separated, by Observation~\ref{obs_separation}, there exists a constant $M>0$ so that for each $r>0$ such that $\overline{A(r, Kr)}\subset \Or$, the closed annulus $\overline{A(r, Kr)}$ contains at most $\tilde{M}\defeq M+\#B$ ramified points of $\Or$. For each $b\in B$, let $\epsilon_b>0$ be such that
\begin{equation}\label{eq_epsilon_b}
\text{ if } z,w\in \left((P_J\cup B) \cap \overline{A(K^{-1}\vert b \vert, K\vert b \vert)}\right), \quad \text{ then } \vert z-w\vert\geq \epsilon_b \max\{\vert z \vert, \vert w \vert\}.
\end{equation}
For each $b\in B$, the constant $\epsilon_b$ exists because $\#B$ is finite and $P_J$ is discrete. Note that if $b\in \overline{A(r, Kr)}\subset \Or$ for some $r>0$, then $\overline{A(r, Kr)}\subset \overline{A(K^{-1}\vert b \vert, K\vert b \vert)}$. Let 
$$\tilde{\epsilon}\defeq \min \left\{\epsilon,\min_{b \in B} {\epsilon_b}\right\}.$$
Recall that, by Proposition \ref{prop_Or}\ref{item_a_assocorb}, $S$ is a punctured neighbourhood of infinity, and so we can fix an arbitrary $r>0$ such that $A_r\defeq A(r/K, K^2r)\subset S$. Since 
\begin{equation}\label{eq_union_annuli}
\overline{A_r}= \bigcup_{j=0}^{2} \overline{A(K^{j-1}r, K^{j}r)},
\end{equation}
$A_r$ contains at most $3\tilde{M}$ ramified points of $\Or$. We might assume without loss of generality that $r=1$, since otherwise the same argument applies by scaling by $r$. Let $C\defeq \max_{z\in S}\nu(z)$, and note that by \eqref{eq_degree}, $C<\infty$. Then, by Theorem \ref{thm_intro_orb} applied to the domain $A(1/K, K^2)$, the compact set $\overline{A(1,K)}$ and the parameters $3\tilde{M},C \in \N$ and $\tilde{\epsilon}>0$, we conclude that there exists a constant $R_1$ such that 
\begin{equation}\label{eq_cor_exporb}
d_{\hat{\Or}}(p,q)<R_1 \quad \text{ for all }\ p,q \in \overline{A(1,K)} \ \text{ and all orbifolds } \ \hat{\Or}\defeq (A(1/K, K^2), \nu_{\hat{\Or}}),
\end{equation}
where $\nu_{\hat{\Or}}$ is any ramification map that only assumes values smaller or equal to $C$, and it does so for at most $3\tilde{M}$ points, that are at Euclidean distance at least $\tilde{\epsilon}$ from each other. We shall now complete the proof of the first part of the statement using \eqref{eq_cor_exporb}: for each $t>0$ such that $A_t= A(t/K, K^2t)\subset S$, define the orbifolds $\Or_t\defeq (A_t, \nu\vert_{ A_t})$ and $\Or^t_1\defeq (A_1, \nu^t_1)$, where $\nu\vert_{ A_t}$ is the restriction of the ramification map $\nu$ of $\Or$ to $A_t$, and $\nu^t_1(z)\defeq \nu\vert_{ A_t}(tz)$. Note that by \eqref{eq_union_annuli}, the definition of $\nu^t_1$, \eqref{eq_epsilon_b} and Observation \ref{obs_separation}, both $\Or_t$, $\Or^t_1$ contain at most $3\tilde{M}$ ramified points, any pair at a (Euclidean) distance at least $\tilde{\epsilon}$. Consequently, \eqref{eq_cor_exporb} applies to $\Or^t_1$. Then, the map $\phi_{t} \colon \Or^t_1 \rightarrow \Or_t$ given by $\phi_{t}(z)\defeq tz$ is an orbifold covering map, and since $\Or_t \hookrightarrow \Or$ is holomorphic, by definition of $\Or_t$, Corollary \ref{cor_pick} and \eqref{eq_cor_exporb}, for every $p,q \in \overline{A(t,Kt)}$,
$$d_{\Or}(p,q)\leq d_{\Or_t}(p,q)= d_{\Or^t_1}(\phi^{-1}_{t}(p),\phi^{-1}_{t}(q))<R_1,$$
and the first statement of the lemma is proved.

In order to prove the second part of the lemma, if $f\in \B$, let $\{z_i\}_{i\geq 0} \subset \mathbb{B}_\Ort^\Or$ be the infinite sequence of points from Proposition \ref{cor_Kzi} for which there is $N>1$ such that $\vert z_i\vert<\vert z_{i+1}\vert\leq N\vert z_i\vert$ holds for all $i\geq 0$. Recall that by Proposition \ref{prop_Or}, $S=\C=\tilde{S}$, or $S$ is the complement of a finite union of bounded Jordan domains and $\text{cl}({\tilde{S}})\subset S$. Then, there exists a finite number 
$$I\defeq\min\left\{ j\geq 0: \C \setminus \D_{\frac{\vert z_j\vert}{K}} \subset S\right\},$$
that equals $0$ when $S=\C$. Let $J\defeq\left \lceil{\frac{\log N}{\log K}}\right \rceil$ and for each $i> I$, denote
$$A_{i}\defeq \overline{A(\vert z_{i-1}\vert, N\vert z_{i-1}\vert)}\subseteq \bigcup_{j=1}^{J } \overline{A\left(K^{j-1}\vert z_{i-1} \vert, K^{j}\vert z_{i-1} \vert \right)}.$$
In particular, $z_i\in A_i$ for all $i>I$. Hence, for all $z \in A_{i}\subset S$, using the first part of the lemma, $d_{\Or}(z, \mathbb{B}_\Ort^\Or)\leq J \cdot R_1\eqdef R_2.$ Since the constant $J$ is independent of the index $i>I$, we can conclude that
\begin{equation}\label{eq_cor_exporb2}
d_\Or(z,\mathbb{B}_\Ort^\Or)<R_2 \quad \text{ for all }\quad z\in \bigcup_{i>I}A_i=\C\setminus \D_{\vert z_I \vert}.
\end{equation}
If $\tilde{S} \subset \C\setminus \D_{\vert z_I \vert}$, we are done. Otherwise, recall that either $\C\setminus \D_{\vert z_I \vert} \subset \tilde{S}=\C$, or we have that $\text{cl}(\tilde{S})\subset S$. In any case, we can consider the compact set $\mathcal{K}\defeq \text{cl}(\D_{\vert z_I \vert }\cap \tilde{S})$ and any domain $U$ such that $\mathcal{K} \subset U \subset S$. In particular, if $\Or_{U}\defeq(U, \nu\vert_{ U})$, then the inclusion $\Or_{U} \hookrightarrow \Or$ is holomorphic. Note also that in the first case, $z_I\in \mathcal K$, while in the second, $\partial S \cap \mathcal{K}\neq \emptyset$ and all points in that intersection also belong to $\mathbb{B}_\Ort^\Or$. Consequently, in any case we can choose a point $p\in \mathcal{K} \cap \mathbb{B}_\Ort^\Or$. Let $\tilde{N}$ be the number of ramified points of $\Or_{U}$. If $\tilde{N}>1$, let $\delta$ be the minimum of the (Euclidean) distances between any two ramified points in $U$. Otherwise, if $\tilde{N}$ equals $0$ or $1$, let $\delta$ be any positive real number. Then, by Corollary \ref{cor_pick} and Theorem \ref{thm_intro_orb} applied to $U,\mathcal{K}$ and the parameters $C,\tilde{N}\in \N$ and $\delta>0$, there exists a constant $R_3>0$ such that for all $z\in \text{cl}(\D_{\vert z_I \vert }\cap \tilde{S})$,
\begin{equation*}\label{eq_cor_exporb3}
d_\Or(z,\mathbb{B}_\Ort^\Or) \leq d_{\Or_{U}}(z,\mathbb{B}_\Ort^\Or)\leq d_{\Or_U}(z,p)<R_3.
\end{equation*}
By this together with \eqref{eq_cor_exporb2}, the lemma follows letting $R\defeq \max\{R_1, R_2, R_3\}$.
\end{proof}

\noindent Theorem \ref{thm_main_intro_Orb} now follows easily on combining Lemma \ref{lem_annulus} and Theorem \ref{prop_expansion_Intro}:
\begin{proof}[Proof of Theorem \ref{thm_main_intro_Orb}]
Let $f\in \B$ be strongly postcritically separated. By Proposition \ref{prop_Or}, there exists a pair of hyperbolic orbifolds $\Or\defeq(S,\nu)$ and $\Ort\defeq(\tilde{S},\tilde{\nu})$ such that $J(f)\subset \tilde{S}\subset S$, $f\colon \Ort \rightarrow \Or$ is a covering map and the inclusion $\Ort \hookrightarrow \Or$ is holomorphic. Hence, by Corollary~\ref{cor_derivative_covering}, 
\begin{equation}\label{eq_quotient_deriv}
\Vert \Deriv f(z)\Vert_{\Or}=\frac{\vert f'(z)\vert\rho_{\Or}(f(z))}{\rho_{\Or}(z)} =\frac{\rho_\Ort(z)}{\rho_\Or(z)}.
\end{equation}
Moreover, by Lemma \ref{lem_annulus}, there exists a constant $R$ such that $d_\Or( \mathbb{B}_\Ort^\Or, z)<R$ for all unramified $z \in \Ort$. Thus, by Theorem \ref{prop_expansion_Intro} and using \eqref{eq_quotient_deriv}, $\Vert \Deriv f(z)\Vert_{\Or}\geq (\e^R/\sqrt{\e^{2R}-1}) \eqdef \Lambda>1$ for all unramified $z \in \Ort$, as we wanted to show.
\end{proof}
As a consequence of Theorem \ref{thm_main_intro_Orb}, we obtain the following corollary that relates the $\Or$-length of bounded curves to the $\Or$-length of its successive images.
\begin{cor}[Shrinking of preimages of bounded curves] \label{cor_uniform} Let $f\in \B$ be a strongly postcritically separated map, and let $(\Ort,\Or)$ be a pair of dynamically associated orbifolds. Then, for any curve $\gamma_0 \subset \Or$, for all $k\geq 1$ and each curve $\gamma_k \subset f^{-k}(\gamma_0)$ such that $f^k\vert_{\gamma_k}$ is injective, $$\ell_{\Or}(\gamma_k)\leq \frac{\ell_{\Or}(\gamma_0)}{\Lambda^{k}}$$
for some constant $\Lambda>1$.	
\end{cor}
\begin{proof}
By Theorem \ref{thm_main_intro_Orb} and Corollary \ref{cor_derivative_covering}, there exists a constant $\Lambda$ such that for all unramified $z\in\Ort$,
\begin{equation}\label{eq_exp_lambda2}
\Vert \Deriv f(z)\Vert_{\Or}=\frac{\rho_\Ort(z)}{\rho_\Or(z)}\geq\Lambda>1.
\end{equation}
In particular, recall that the set of ramified points in $\Ort$ is negligible when computing the length of bounded curves, as it is discrete and so has Lebesgue measure $0$. Let $\gamma_0$ be any curve as in the statement. We proceed by induction on $k$. Suppose $k=1$ and let us parametrize the curves $\gamma_0$ and $\gamma_1$ such that $f(\gamma_1(t))=\gamma_0(t)$ for all $t\geq 0$. Since by Proposition \ref{prop_Or} $f\colon \Ort \rightarrow \Or$ is an orbifold covering map, by Theorem \ref{pick}, $\rho_\Ort(\gamma_1(t)) =\vert f'(\gamma_1(t))\vert\cdot \rho_\Or(\gamma_0(t))$ for all $t\geq 0$. Using this and \eqref{eq_exp_lambda2},
\begin{equation*}
\begin{split}\ell_\Or(\gamma_1)&=\int \vert \gamma'_1(t)\vert \rho_{\Or}(\gamma_1(t))dt=\int \vert \gamma'_1(t)\vert \frac{\rho_\Or(\gamma_1(t))}{\rho_\Ort(\gamma_1(t))} \rho_{\Ort}(\gamma_1(t))dt\\
&\leq \frac{1}{\Lambda} \int \vert \gamma'_1(t)\vert \!\!\cdot \!\!\vert f'(\gamma_1(t))\vert \rho_\Or(\gamma_0(t))dt=\frac{1}{\Lambda} \int \vert \gamma'_0(t)\vert\rho_\Or(\gamma_0(t)) dt \leq \frac{\ell_\Or(\gamma_{0})}{\Lambda}.
\end{split}
\end{equation*}
Let us suppose that the statement is true for some $k-1$. Then, if $\gamma_k \subset f^{-k}(\gamma_0)$, $f(\gamma_{k})=\gamma_{k-1}$ for some curve $\gamma_{k-1} \subset f^{-k+1}(\gamma_0)$. By the same argument as before and using the inductive hypothesis, 
\begin{equation*}
\begin{split}
\ell_\Or(\gamma_k)\leq \frac{1}{\Lambda} \int \vert \gamma'_{k-1}(t)\vert\rho_\Or(\gamma_{k-1}(t)) dt=\frac{1}{\Lambda} \ell_\Or(\gamma_{k-1}) \leq \frac{\ell_\Or(\gamma_{0})}{\Lambda^{k}}.\qedhere
\end{split}
\end{equation*}
\end{proof}

\section{Results on the topology of Fatou and Julia sets} \label{sec_Fatou}
In this section we provide the proofs of Theorem \ref{thm_intro1.2} and Corollaries \ref{cor_intro1.8} and \ref{cor_intro1.9}. We note that the arguments in the proofs of the corresponding results for hyperbolic maps in \cite{lasseNuriaWalter}, rely mostly on the maps being \textit{expanding}. That is, in their derivative with respect to the  hyperbolic metric being greater than one in a punctured neighbourhood of infinity that contains their Julia set. Since we have achieved an analogous result for strongly postcritically separated maps in Theorem \ref{thm_main_intro_Orb}, we are able to adapt most of the proofs in \cite{lasseNuriaWalter} with few modifications. We start by borrowing some auxiliary results from \cite{lasseNuriaWalter}. The first one is a well-known result that we cite as stated in {\cite[Lemma 2.7]{lasseNuriaWalter}}.
\begin{lemma}[Coverings of doubly-connected domains] \label{lem_2.7}
Let $U,V\subset\C$ be domains and let $f\colon V\to U$ be a covering map. Suppose that $U$ is doubly-connected. Then either $V$ is doubly-connected and $f$ is a proper map, or $V$ is simply connected and $f$ is a universal cover of infinite degree.
\end{lemma}
The next proposition gathers some well-known facts on the behaviour of entire maps on preimages of simply-connected domains. For ease of reference, the following statement merges \cite[Propositions 2.8 and 2.9]{lasseNuriaWalter}. For the first part, compare to \cite{heins57,herring_propFatou, bolsch}. 
\begin{prop}[Mapping of simply connected sets]\label{prop_5.2_3}
Let $f$ be an entire function, let $D\subset\C$ be a simply connected domain, and let $\widetilde{D}$ be a component of $f^{-1}(D)$. Then either
\begin{enumerate}[label=(\arabic*)]
\item \label{item_1.5.2} $f\colon\widetilde{D}\to D$ is	a proper map and hence has finite degree, or 
\item \label{item_2.5.2} for every $w\in D$ with at most one exception, $\#(f^{-1}(w)\cap \widetilde{D})$ is infinite. In this case, either $\widetilde{D}$ contains an asymptotic curve corresponding to an asymptotic value in $D$, or $\widetilde{D}$ contains infinitely many critical points.
\end{enumerate}
If in addition $D\cap S(f)$ is compact,
\begin{enumerate}[label=(\Alph*)]
\item If $\# (D\cap S(f))\leq 1$, then $\widetilde{D}$ contains at most one critical point of $f$. \label{item_A}
\item In case \ref{item_1.5.2}, if $D$ is a bounded Jordan domain such that $\partial D\cap S(f)=\emptyset$, then $\widetilde{D}$ is also a bounded Jordan domain.\label{item_B}
\item In case \ref{item_2.5.2}, the point $\infty$ is accessible from $\widetilde{D}$.\label{item_C}
\end{enumerate}
\end{prop}

In addition, we will make use of the following result in order to show that the boundaries of certain Fatou components are not locally connected.
\begin{thm}[Boundaries of periodic Fatou components {\cite[Theorem 2.6]{lasseNuriaWalter}}] \label{thm_2.6}
Let $f$ be a transcendental entire function, and suppose that $U$ is an unbounded periodic component of $F(f)$ such that $f^n\vert_U$ does not tend to infinity. Then $\widehat{\C}\setminus U$ is not locally connected at any finite point of $\partial U$.
\end{thm}
The proof of Theorem \ref{thm_intro1.2} will follow easily once we show that whenever condition $(b)$ on its statement holds, every \textit{periodic} Fatou component is bounded. We achieve so in the following theorem. In particular, we note the similarities with \cite[Theorem 1.10]{lasseNuriaWalter}: Theorem~\ref{thm_my1.10} holds for a more general class of maps, but \cite[Theorem 1.10]{lasseNuriaWalter} has the stronger conclusion that periodic Fatou components are quasidiscs. We suspect that this is also the case for the class of maps we study. However, we have not been able to conclude so; see \cite[p. 169]{mio_thesis} for further discussion.
\begin{thm}[Immediate basins of strongly postcritically separated maps]\label{thm_my1.10}
Let $f\in\B$ be strongly postcritically separated and let $D$ be a periodic Fatou component of $f$, of some period $p\geq 1$. Then the following are equivalent:
\begin{enumerate}[label=(\arabic*)]
\item $D$ is a Jordan domain;\label{item_1per}
\item $\widehat{\C}\setminus D$ is locally connected at some finite point of $\partial D$; \label{item_2per}
\item $D$ is bounded; \label{item_3per}
\item the point $\infty$ is not accessible in $D$;\label{item_4per}
\item the orbit of $D$ contains no asymptotic curves and only finitely many critical points;\label{item_5per}
\item $f^p\colon D\to D$ is a proper map;\label{item_6per}
\item for at least two distinct choices of $z\in D$, the set $f^{-p}(z)\cap D$ is finite. \label{item_7per}
\end{enumerate}
\end{thm}

\begin{proof}
Let $f$ and $D$ be as in the statement. In particular, $D$ is simply connected (multiply-connected Fatou components of transcendental entire functions are wandering domains \cite[Theorem 3.1]{Baker_wandering84}). By passing to an iterate, we may assume without loss of generality that $p=1$. Since the complement of a Jordan domain is locally connected at every point, \ref{item_1per}$\Rightarrow$\ref{item_2per} is immediate. If \ref{item_2per} holds, then since by Lemma \ref{lem_deff} all Fatou components of $f$ belong to attracting cycles, by Theorem \ref{thm_2.6} $D$ must be bounded, and so \ref{item_2per} implies \ref{item_3per}. If $D$ is bounded, then $D$ cannot contain a curve to $\infty$, and hence \ref{item_3per}$\Rightarrow$\ref{item_4per}. Since $f$ is postcritically separated and $D\subset F(f)$, $P(f)\cap D$ is compact. Thus, by Proposition \ref{prop_5.2_3}\ref{item_C}, if infinity is not accessible in $D$, then item \ref{item_1.5.2} must occur in Proposition~\ref{prop_5.2_3}, and so $D$ contains only finitely many critical points and no asymptotic values, which is equivalent to $f\colon D\to D$ being a proper map. Thus, \ref{item_4per}$\Rightarrow$\ref{item_5per}$\Leftrightarrow$\ref{item_6per}. Since any proper map has finite degree, \ref{item_6per}$\Rightarrow$\ref{item_7per}.

To conclude the proof, it suffices to show that \ref{item_7per}$\Rightarrow$\ref{item_1per}. With that aim, suppose that \ref{item_7per} holds for $f$. Recall that by Proposition \ref{prop_Jordan}, there exists a bounded Jordan domain $U_0\Subset D$ such that $\overline{f(U_0)}\subset U_0$ and $P(f)\cap D\Subset U_0$. For each $n\geq 1$, let
$$U_n\defeq f^{-n}(U_0)\cap D,$$ and note that by the property $\overline{f(U_0)}\subset U_0$, one can see using induction that
\begin{equation}\label{eq_U_n}
U_{n}\subset U_{n+1} \quad \text{ for all } \quad n\geq 0, \quad \text{ and } \quad D=\bigcup^\infty_{n=0} U_n.
\end{equation}
Since we have assumed that \ref{item_7per} holds for $f$, so does Proposition \ref{prop_5.2_3}\ref{item_1.5.2}, and hence $f\colon D\to D$ is a proper map of some degree $d\geq 1$. Moreover, by definition of $U_0$, for each $n\geq 1$, $f^n\colon D\setminus \overline{U_n}\to A \defeq D\setminus\overline{U_0}$ is a finite-degree covering map (of degree $d^n$) over the doubly-connected domain $A$. By Lemma \ref{lem_2.7}, the domain $D\setminus\overline{U_n}$ is also doubly-connected, and hence $U_n$ is connected for all $n$. Furthermore, since $P(f)\cap D\Subset U_0$, it follows from Proposition \ref{prop_5.2_3}\ref{item_B} applied to $f^n\colon U_n \rightarrow U_0$ that each $U_n$ is a bounded Jordan domain. Hence, for every $n\geq 0$, $f\colon \partial U_{n+1}\to \partial U_n$ is, topologically, a covering map of degree $d$ over a circle.
\begin{Claim}There exists a diffeomorphism $\phi \colon \{z\in\C\colon 1/e<|z|<1\} \to D\setminus \overline{U_0}$ such that 
\begin{equation}\label{eq_diffeo}
f(\phi(z)) = \phi(z^d) \qquad \text{whenever } e^{-1/d}<|z|<1. \end{equation}
\end{Claim}
This claim and its proof appear in \cite[Proof of Theorem 1.10]{lasseNuriaWalter}, and thus we omit the proof. Our next and final goal is to extend continuously the domain of the function $\phi$ to include $\partial \D$. With that aim, for each $\theta\in\R$ and $n\geq 0$, consider the curve
\begin{equation}\label{eq_def_gammatheta}
\gamma_{n,\theta} \defeq \phi( \{ e^{a +i\theta} \colon -d^{-n} \leq a \leq -d^{-(n+1)}\} ).
\end{equation}
Note that by the commutative relation in \eqref{eq_diffeo}, $\gamma_{n,\theta}$ is the preimage of the arc $\gamma_{0,\theta\cdot d^n}$ under some branch of $f^{-n}$, and in particular it is a simple curve with endpoints in $\partial U_n$ and $\partial U_{n+1}$.

Let $\Ort\defeq (\widetilde{S}, \tilde{\nu})$ and $\Or\defeq (S, \nu)$ be a pair of orbifolds dynamically associated to $f$. In particular, by Proposition \ref{prop_Or}\ref{item_a_assocorb}, $S$ can be chosen so that $D\setminus U_0\subset S$. Note that for each $\tilde{\theta}\in \R$, the curve $\gamma_{0,\widetilde{\theta}}$ is contained in the compact set $\overline {U_1\setminus U_0}$. Since $D\setminus U_0\subset F(f) \setminus P(f)$, by Proposition \ref{prop_Or}\ref{item_b_assocorb}, there are no ramified points of $\Or$ in $\overline{U_1\setminus U_0}$. Thus,  the density $\rho_\Or$ of the orbifold metric of $\Or$ attains a maximum value in $\overline{U_1\setminus U_0}$, and since the Euclidean length of the curves $\{\gamma_{0,\widetilde{\theta}}\}_{\tilde{\theta}}$ must be finite, as these curves are the image under $\phi$ of a straight line, there exists a constant $L>0$ such that 
$$\max_{\widetilde{\theta}}\ell_{\Or}(\gamma_{0,\widetilde{\theta}})<L.$$
Moreover, for all $n\geq 0$ and $\theta \in \R$, $\gamma_{n,\theta}\subset D\setminus U_0 \subset S\setminus P(f)$, and so, since $D$ is by assumption invariant, $f^n$ maps $\gamma_{n,\theta}$ injectively to $\gamma_{0,\theta\cdot d^n}$. Hence, we can apply Corollary \ref{cor_uniform} to conclude that there exists a constant $\Lambda>1$ such that 
\begin{equation}\label{eq_contractiontheta}
\ell_{\Or}(\gamma_{n,\theta})\leq \frac{\ell_{\Or}(\gamma_{0,\theta\cdot d^n})}{\Lambda^{n}} \leq \frac{\max_{\widetilde{\theta}}\ell_{\Or}(\gamma_{0,\widetilde{\theta}})}{\Lambda^{n}}\leq \frac{L}{\Lambda^{n}},
\end{equation}
where we note that the upper bound is independent of $\theta$. For each $n\geq 0$, let us define the function 
\[
\sigma_n\colon \R/\Z \rightarrow \partial U_n \quad \text{ as }\quad \sigma_n(t)\defeq \varphi\left(e^{-d^{-n}+ 2\pi t i }\right). \]
Note that for each $t\in \R/\Z$, the curve $\gamma_{n,2\pi t}$ defined in \eqref{eq_def_gammatheta} joins $\sigma_n(t)$ and $\sigma_{n+1}(t)$. Thus, by \eqref{eq_contractiontheta}, $\{\sigma_n\}_n$ forms a Cauchy sequence of continuous functions. Consequently, using \eqref{eq_U_n}, there exists a limit function $\sigma\colon \partial \D \rightarrow \partial D,$ which by \eqref{eq_diffeo} is the continuous extension of $\varphi$ to the unit circle. Hence, $\partial D$ is a continuous closed curve as it is the continuous image of $\partial \D$. In particular, $D$ is bounded. By the maximum principle, $\partial D = \partial \overline{D}$ and $\C\setminus \overline{D}$ has no bounded connected components. Hence, $D$ is a Jordan domain.
\end{proof}
Using the preceding theorem, we are now ready to provide the proofs of our results on the topology of Fatou and Julia sets.
\begin{proof}[Proof of Theorem \ref{thm_intro1.2}]
We start proving that $(a)$ implies $(b)$ by showing the contrapositive. Note that since $f$ is strongly postcritically separated, all asymptotic values of $f$ must lie in $F(f)$, and hence if $\text{AV}(f)\neq \emptyset$, then $F(f)$ must have an unbounded component by definition of asymptotic value. Moreover, if some Fatou component $U$ contains infinitely many critical points, since these are the zeros of the analytic function $f'$, they can only accumulate at infinity, and therefore $U$ is unbounded.

To prove that $(b)$ implies $(a)$, we note that by Lemma \ref{lem_deff}, all Fatou components of $f$ are (pre)periodic. If $(b)$ holds for $f$, that is, $\text{AV}(f)=\emptyset$ and each Fatou component contains at most finitely many critical points, then by Theorem \ref{thm_my1.10} ($\ref{item_5per}\!\!\! \iff \!\!\!\ref{item_1per}$), every periodic Fatou component is a bounded Jordan domain. Next, we see that strictly preperiodic Fatou components are also bounded Jordan domains. If $V$ is any preimage of a periodic Fatou component $U$, then, by assumption, Proposition \ref{prop_5.2_3}\ref{item_2.5.2} cannot hold. Thus, $f\colon V\to U$ must be a proper map. In addition, $V$ is also bounded by Proposition \ref{prop_5.2_3}\ref{item_B}. Proceeding by induction on the pre-period of $V$, the claim follows.
\end{proof}

In order to prove Corollary \ref{cor_intro1.8}, we will make use of a result from \cite{Bergweiler_morosawa}, where the concept of \textit{semihyperbolic} entire maps is introduced:
\begin{defn}[Semihyperbolic functions] An entire function $f$ is \textit{semihyperbolic at a point} $p$ if there exist $r > 0$ and $N \in \N$ such that for all $n \in \N$ and for all components $U$ of $f^{-n}(\D_r(p)) = \{z \in \C : f^n(z) \in \D_r(p)\}$, the function $f^n\vert_U \colon U \rightarrow \D_r(p)$ is a proper map of degree at most $N$. A function $f$ is \textit{semihyperbolic} if $f$ is semihyperbolic at all $p \in J(f)$.
\end{defn}

\begin{prop} If $f$ is strongly postcritically separated, then $f$ is semihyperbolic.
\end{prop}
\begin{proof}
Let us fix $p\in J(f)$. Since $P(f)\cap J(f)$ is discrete and $P(f)\cap F(f)$ is compact, there exists $r>0$ such that $\D_r(p)\cap P(f)$ contains at most the point $p$. By Definition \ref{def_strongps}, there exist constants $C,\mu>0$ such that for all $z\in J(f)$, 
$$\# (\text{Orb}^+(z)\cap \Crit(f)) \leq C \quad \text{ and } \quad \deg(f,z)<\mu.$$
Therefore, for any $n\geq 0$ and any connected component $U$ of $f^{-n}(\D_r(p))$, since we have that $U\cap \text{Orb}^{-}(P(f))\subset \text{Orb}^{-}(p)$, $f^n\vert_U$ is a proper map of degree at most $\mu^C\eqdef N$.
\end{proof}

The following theorem is a version of \cite[Theorem 4]{Bergweiler_morosawa} for our class of maps. In particular, this theorem tells us that if Fatou components are Jordan domains, in certain cases local connectivity of their Julia sets follows. 
\begin{thm}[Bounded components and bounded degree imply local connectivity]\label{thm_morosawalc} 
Let $f\in \B$ be strongly postcritically separated with no asymptotic values. Suppose that every immediate attracting basin of $f$ is a Jordan domain. If there exists $N\in \N$ such that the degree of the restriction of $f$ to any Fatou component is bounded by $N$, then $J(f)$ is locally connected.
\end{thm}
\begin{remark}
We note that \cite[Theorem 2.5]{lasseNuriaWalter} is a version of Theorem \ref{thm_morosawalc} for hyperbolic maps whose proof is based on expansion of hyperbolic maps in a neighbourhood of their Julia set. Therefore and alternatively, we could have presented an analogous proof for functions in class $\B$ that are strongly postcritically separated using Theorem \ref{thm_main_intro_Orb}. 
\end{remark}
\begin{proof}[Proof of Corollary \ref{cor_intro1.8}]
Let $f\in\B$ be strongly postcritically separated with no asymptotic values, and assume that every Fatou component of $f$ contains at most $N$ critical points, counting multiplicity, for some $N\in\N$. Then, hypothesis $(b)$ in Theorem \ref{thm_intro1.2} holds for $f$, and consequently every Fatou component $U$ is a bounded Jordan domain. Moreover, Proposition~\ref{prop_5.2_3}\ref{item_1.5.2} must hold and so the restriction $f\vert_U\colon U\to f(U)$ is a proper map. Since $f$ has no wandering domains (Lemma \ref{lem_deff}), all Fatou components of $f$ are simply connected \cite[Theorem 3.1]{Baker_wandering84}. Then, the Riemann-Hurwitz formula, see \cite[Theorem 7.2]{milnor_book}, tells us that the degree of $f\vert_U$ is bounded by $N+1$. Consequently, local connectivity of $J(f)$ follows from Theorem \ref{thm_morosawalc}.
\end{proof}
\begin{proof}[Proof of Corollary \ref{cor_intro1.9}]
Let $f$ be strongly postcritically separated with no asymptotic values, and assume that every Fatou component contains at most one critical value. Then, by Proposition \ref{prop_5.2_3}\ref{item_A}, each Fatou component also contains at most one critical point. Since, by assumption, the multiplicity of the critical points is uniformly bounded, local connectivity of $J(f)$ is a consequence of Corollary \ref{cor_intro1.8}.
\end{proof}
\section{Pullbacks and post-homotopy classes }\label{sec_homotopies}
Given an entire function $f$ and two simple curves $\gamma, \beta \subset f(\C)\setminus P(f)$, homotopic and with fixed endpoints, by the homotopy lifting property, for each curve in $f^{-1}(\gamma)$, there exists a curve in $f^{-1}(\beta)$ homotopic to it and sharing the same endpoints. In Proposition \ref{cor_homot} we get, by using a modified notion of homotopy, an analogue of this result for a certain class of curves that contain postsingular points. Moreover, in this section we also show that if $f$ is an entire function with dynamic rays in its Julia set and $U$ is a certain bounded domain of any hyperbolic orbifold whose underlying surface intersects $J(f)$, then there exists a constant $\mu$ such that for every piece of dynamic ray contained in $U$, we can find a curve in its ``modified-homotopy'' class with orbifold length at most $\mu$; see Definition \ref{def_ray} and Corollary \ref{cor_homot2}. In particular, this result is crucial to prove the main result in \cite{mio_splitting}.

For completeness and in order to fix notation, we include some definitions regarding homotopy and covering spaces theory that we require, and we refer the reader to \cite[Chapter~1]{hatcher2002algebraic} or \cite[Chapter 9]{Munkres} for an introduction to these topics. In this section, by a \textit{curve} in a space $X$ we mean a continuous map $\gamma :I\rightarrow X$ with $I=[0,1]$, and in particular its image $\gamma(I)$ is bounded. With slight abuse of notation, we also refer by $\gamma$ to $\gamma(I)$, and we denote by $\text{int}(\gamma)$ the curve obtained from $\gamma$ by removing its endpoints. A \textit{homotopy of curves} in $X$ is a family $\{\gamma_t:I\rightarrow X\}_{t\in [0,1]} $ for which the associated map $\overline{\gamma}:I\times [0,1]\rightarrow X$ given by $\overline{\gamma}(s,t) \defeq \gamma_t(s)$ is continuous. Two curves $\alpha$ and $\beta$ are \textit{homotopic in $X$} when there exists a homotopy $\{\gamma_t\}_{t\in [0,1]}$ in $X$ such that $\gamma_0\equiv\alpha$ and $\gamma_1\equiv\beta$. Being homotopic is an equivalence relation on the set of all curves in $X$. Given a covering space $f : \tilde{X} \rightarrow X$, a \textit{lift} of a map $g: Y\rightarrow X $  by $f$ is a map $\tilde{g}: Y \rightarrow \tilde{X}$ such that $f\circ \tilde{g}= g$. The main result that serves our purposes is the following:
\begin{prop}[Homotopy lifting property] \label{prop_hatcher} Given a covering space $f : \tilde{X} \rightarrow X$, a homotopy $\{\gamma_t:Y\rightarrow X\}_{t\in [0,1]}$ and a map $\tilde{\gamma}_0: Y \rightarrow \tilde{X}$ lifting $\gamma_0$, there exists a unique homotopy $\{\tilde{\gamma}_t:Y\rightarrow \tilde{X}\}_{t\in [0,1]}$ that lifts $\{\gamma_t\}_{t\in [0,1]}$.
\end{prop}
\begin{proof}
See \cite[Proposition 1.30]{hatcher2002algebraic} for the proof of the statement whenever the homotopies have fixed endpoints, and \cite[(5.3) Covering Homotopy Theorem]{Greenberg} or \cite[Section 4.2]{hatcher2002algebraic} for the general case. 
\end{proof}

Recall that for an entire function $f$, its singular set $S(f)$ is the smallest closed set for which $f: \C \setminus f^{-1}(S(f)) \rightarrow \C \setminus S(f)$ is a covering map, and regarding the iterates of $f$, for each $k\geq 1$, $S(f^{k})\subseteq P(f)$; see \cite[Proposition 2.13]{mio_thesis}. Consequently, for all $k\geq 1$ and every entire function $f$,
\begin{equation}\label{eq_covering}
f^{k} \colon \C \setminus f^{-k}(P(f))\rightarrow\:\C \setminus P(f) \: \text{ is a covering map}.
\end{equation}
Thus, the homotopy lifting property applies to any homotopy of curves in $\C \setminus P(f)$. We are interested in obtaining an analogous property that applies to certain curves whose image in $\C$ contains postsingular points. We specify now which curves we are interested in:

\begin{discussion}[Definition of the sets $\mathcal{H}^{q}_{p} (W(k))$] Let us fix an entire function $f$ and let $k\in \N$. We suggest the reader keeps in mind the case when $k=0$, since it will be the one of greatest interest for us. Let $W(k)$ be a finite set of (distinct) points in $f^{-k}(P(f))$, totally ordered with respect to some relation ``$\prec$''. That is, $W(k)\defeq(W(k), \prec)=\{w_1, \ldots, w_N \}\subset f^{-k}(P(f))$ such that $w_{j-1}\prec w_j\prec w_{j+1}$ for all $ 1< j< N$. We note that $W(k)$ can be the empty set. Then, for every pair of points\footnote{In particular, $p$ and $q$ might belong to $f^{-k}(P(f))$.} $p,q \in \C\setminus W(k)$, we denote by $\mathcal{H}^{q}_{p} (W(k))$ the collection of all curves in $\C$ with endpoints $p$ and $q$ that join the points in $W(k)$ \textit{in the order} ``$\prec $'', starting from $p$. More formally, $\gamma\in \mathcal{H}^{q}_{p} (W(k))$ if $\text{int}(\gamma)\cap f^{-k}(P(f))=W(k)$ and $\gamma$ can be parametrized so that $\gamma(0)=p$, $\gamma(1)=q$ and $\gamma(\frac{j}{N+1})=w_j$ for all $1\leq j \leq N$. In particular, $\gamma$ can be expressed as a concatenation of $N+1$ curves 
\begin{equation}\label{def_concat}
\gamma=\gamma^{w_{1}}_{p} \bm{\cdot}\gamma^{w_{2}}_{w_1} \bm{\cdot} \cdots \bm{\cdot}\gamma^{q}_{w_N},
\end{equation}
each of them with endpoints in $W(k) \cup\{p,q\}$ and such that $$\text{int}(\gamma^{w_1}_{p}), \text{int}(\gamma^{w_{i+1}}_{w_i}), \text{int}(\gamma^{q}_{w_N}) \subset \C \setminus f^{-k}(P(f))$$
for each $1\leq i \leq N$; see Figure \ref{fig:def_pfhomot}.
\end{discussion}

\begin{figure}[htb]
\begingroup%
\makeatletter%
\providecommand\color[2][]{%
\errmessage{(Inkscape) Color is used for the text in Inkscape, but the package 'color.sty' is not loaded}%
\renewcommand\color[2][]{}%
}%
\providecommand\transparent[1]{%
\errmessage{(Inkscape) Transparency is used (non-zero) for the text in Inkscape, but the package 'transparent.sty' is not loaded}%
\renewcommand\transparent[1]{}%
}%
\providecommand\rotatebox[2]{#2}%
\newcommand*\fsize{\dimexpr\f@size pt\relax}%
\newcommand*\lineheight[1]{\fontsize{\fsize}{#1\fsize}\selectfont}%
\ifx\svgwidth\undefined%
\setlength{\unitlength}{340.15748031bp}%
\ifx\svgscale\undefined%
\relax%
\else%
\setlength{\unitlength}{\unitlength * \real{\svgscale}}%
\fi%
\else%
\setlength{\unitlength}{\svgwidth}%
\fi%
\global\let\svgwidth\undefined%
\global\let\svgscale\undefined%
\makeatother%
\resizebox{0.65\textwidth}{!}{ 
\begin{picture}(1,0.375)%
\lineheight{1}%
\setlength\tabcolsep{0pt}%
\put(0,0){\includegraphics[width=\unitlength,page=1]{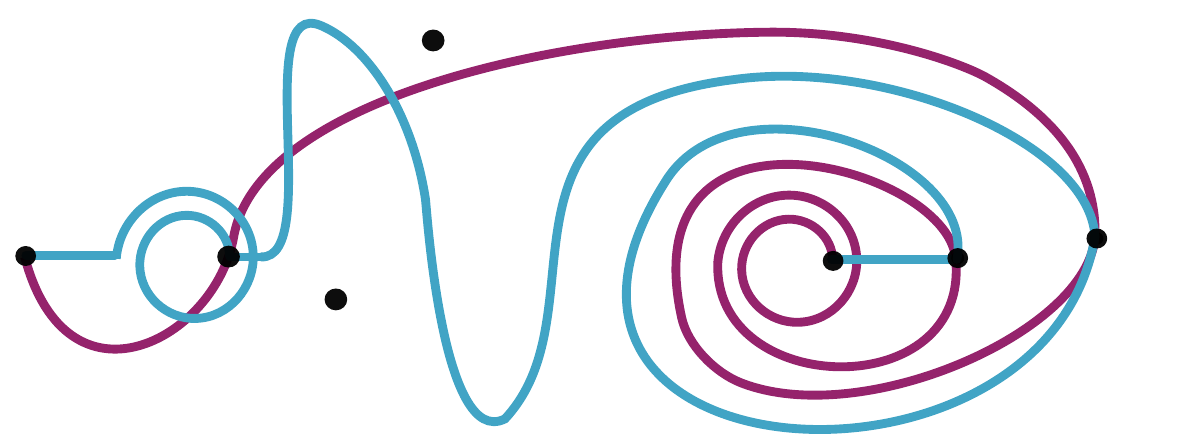}}%
\put(0.06938719,0.03241959){\color[rgb]{0.58431373,0.1372549,0.42352941}\makebox(0,0)[lt]{\lineheight{1.25}\smash{\begin{tabular}[t]{l}\fontsize{15pt}{1em}$\bm{\beta}$\end{tabular}}}}%
\put(0.19285463,0.3230327){\color[rgb]{0.25490196,0.64313725,0.77254902}\makebox(0,0)[lt]{\lineheight{1.25}\smash{\begin{tabular}[t]{l}\fontsize{15pt}{1em}$\bm{\gamma}$\end{tabular}}}}%
\put(0.00314301,0.17159607){\color[rgb]{0,0,0}\makebox(0,0)[lt]{\lineheight{1.25}\smash{\begin{tabular}[t]{l}$p$\end{tabular}}}}%
\put(0.93335476,0.15072401){\color[rgb]{0,0,0}\makebox(0,0)[lt]{\lineheight{1.25}\smash{\begin{tabular}[t]{l}$w_2$\end{tabular}}}}%
\put(0.13715631,0.14550355){\color[rgb]{0,0,0}\makebox(0,0)[lt]{\lineheight{1.25}\smash{\begin{tabular}[t]{l}$w_1$\end{tabular}}}}%
\put(0.82403138,0.15007164){\color[rgb]{0,0,0}\makebox(0,0)[lt]{\lineheight{1.25}\smash{\begin{tabular}[t]{l}$w_3$\end{tabular}}}}%
\put(0.66403745,0.1422822){\color[rgb]{0,0,0}\makebox(0,0)[lt]{\lineheight{1.25}\smash{\begin{tabular}[t]{l}$q$\end{tabular}}}}%
\end{picture}

}
\endgroup
\caption{Example of two curves $\gamma, \beta \in H_{p}^{q}(\{w_1,w_2, w_3\})$ that are post-$k$-homotopic for some $k\geq 1$. Points in $f^{-k}(P(f))$ are represented by black dots.}
\label{fig:def_pfhomot}
\end{figure}
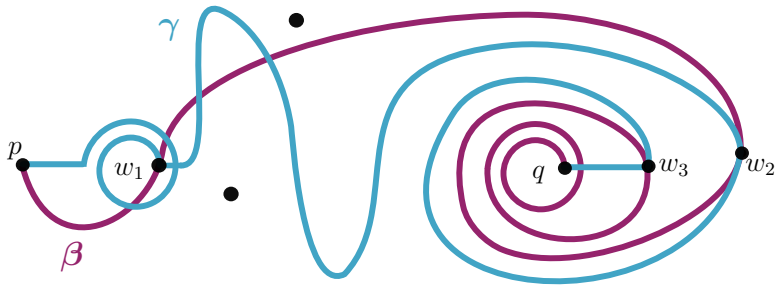
\noindent We use the following notion of homotopy for the sets of curves described:
\begin{defn}[Post-$k$-homotopic curves] \label{def_post0}
Consider $W(k)=\{w_1, \ldots, w_N \}\subset f^{-k}(P(f))$ and two curves $\gamma,\beta\in \mathcal{H}^{w_{N+1}}_{w_0}(W(k))$, for some $\{w_0, w_{N+1}\}\subset \C\setminus W(k)$. We say that $\gamma$ is \emph{post-$k$-homotopic to $\beta$} if for all $0 \leq i\leq N$, $\gamma^{w_{i+1}}_{w_i}$ is homotopic to $\beta^{w_{i+1}}_{w_i} \text{ in } (\C \setminus f^{-k}(P(f))) \cup \{w_i, w_{i+1} \}$.
\end{defn}

\begin{remark}Note that both the curves $\gamma$ and $\beta$ in the definition above belong to $\mathcal{H}^{w_{N+1}}_{w_0}(W(k))$, and so they share the points $w_0, \ldots, w_{N+1}$. In particular, $\gamma$ and $\beta$ share their endpoints, and for each $0 \leq i\leq N$, the respective subcurves $\gamma^{w_{i+1}}_{w_i}$ and $\beta^{w_{i+1}}_{w_i}$ also share the same fixed endpoints $w_i$ and $w_{i+1}$.
\end{remark}
In other words, for each $1\leq i\leq N$, the restrictions of $\gamma$ and $\beta$ between $w_i$ and $w_{i+1}$ are homotopic in the space $(\C \setminus f^{-k}(P(f))) \cup \{w_i, w_{i+1}\}$; see Figure \ref{fig:def_pfhomot}. It is easy to see that this defines an equivalence relation in $\mathcal{H}^q_p (W(k))$, with $p=w_0$ and $q=w_{N+1}$. For each $\gamma\in \mathcal{H}^q_p (W(k))$, we denote by $[\gamma]_{_k}$ its equivalence class. Note that if $W(k)=\emptyset$ and $p,q \in \C \setminus f^{-k}(P(f))$, then for any curve $\gamma\in\mathcal{H}^q_p (W(k))$, $[\gamma]_{_k}$ equals the equivalence class of $\gamma$ in $\C \setminus f^{-k}(P(f))$ in the usual sense. Moreover, if $\gamma$ is any curve that meets only finitely many elements of $f^{-k}(P(f))$, then it belongs to a unique set of the form ``$\mathcal{H}^q_p(W(k))$'' up to reparametrization of $\gamma$, and so its equivalence class $[\gamma]_{_k}$ is defined in an obvious sense. Hence, the notion of post-$k$-homotopy is well-defined for all such curves, and from now on we will sometimes omit the set of curves they belong to.

\noindent The following is an analogue of Proposition \ref{prop_hatcher} for post-$k$-homotopic curves: 
\begin{prop}[Post-homotopy lifting property]\label{cor_homot}
Let $f$ be an entire map and let $C\subset \C$ be a domain so that $f^{-1}(C) \subset C$ and $\text{AV}(f) \cap C=\emptyset.$ Let $\gamma \subset C$ be a bounded curve such that $\#(\gamma \cap P(f))< \infty$. Fix any $k\geq 0$ and any curve $\gamma_k\subset f^{-k}(\gamma)$ for which the restriction $f^k\vert_{ \gamma_k}$ is injective. Then, for each $\beta \in [\gamma]_{_0}$, there exists a unique curve $\beta_k\subset f^{-k}(\beta)$ such that $\beta_k \in [\gamma_k]_{_k}$. In particular, $\beta_k$ and $\gamma_k$ share their endpoints.
\end{prop}
\begin{proof}
Suppose that $\gamma \in \mathcal{H}^{w_{N+1}}_{w_0}(W(0))$, where $w_0$ and $w_{N+1}$ are the endpoints of $\gamma$ and $W(0)=P(f)\cap \text{int}(\gamma)\eqdef\{w_1, \ldots, w_{N}\}$ for some $N>0$. Let $$\tilde{W}(k)\defeq f^{-k}(P(f))\cap \text{int}(\gamma_k)=f^{-k}(W(0))\cap \text{int}(\gamma_k)\eqdef\{v_1, \ldots, v_N\}.$$
In particular, $\gamma_k\in \mathcal{H}^{v_{N+1}}_{v_{0}}(\tilde{W}(k))$ for some $v_0, v_{N+1} \in f^{-k}(\{ w_0, w_{N+1}\}).$
For each $0\leq i\leq N$, we denote by $\gamma^i_k$ the subcurve in $\gamma_k$ with endpoints $v_i$ and $v_{i+1}.$

Similarly, for a fixed $\beta \in [\gamma]_{_0}$ and each $0 \leq i\leq N$, we denote by $\beta^i$ and $\gamma^i$ the respective subcurves in $\beta$ and $\gamma$ with endpoints $w_i$ and $w_{i+1}$. That is, for parametrizations of $\beta$ and $\gamma$ such that $\gamma(\frac{i}{N+1})=w_i=\beta(\frac{i}{N+1})$ for every $0\leq i \leq N+1$, 
$$ \beta^i\defeq \beta\vert_{[\frac{i}{N}, \frac{i}{N+1}]} \quad \text { and } \quad \gamma^i\defeq \gamma\vert_{[\frac{i}{N}, \frac{i}{N+1}]}.$$ For each $\epsilon >0$ small enough, we consider the restrictions $\beta^{i,\epsilon}\defeq \beta\vert_{[\frac{i}{N}+\epsilon, \frac{i}{N+1}-\epsilon]}$ and $\gamma^{i,\epsilon}\defeq \gamma\vert_{[\frac{i}{N}+\epsilon, \frac{i}{N+1}-\epsilon]}$. Then, since $\beta^{i,\epsilon}\subset \beta^i$ and $\gamma^{i,\epsilon}\subset \gamma^i$, 
$\beta^{i,\epsilon}$ is homotopic (in the usual sense) to $\gamma^{i,\epsilon}$ in $(\C \setminus f^{-k}(P(f)) )\cup \{w_i, w_{i+1}\}$. Recall that the notion of homotopy does not demand curves to share their endpoints. Therefore, for each $0\leq i \leq N$, if $\gamma^{i,\epsilon}_k\defeq \gamma^i_k \cap f^{-k}(\gamma^{i,\epsilon})$, by \eqref{eq_covering} and Proposition \ref{prop_hatcher}, there exists a unique curve $\beta^{i,\epsilon}_k \subset f^{-k}(\beta^{i,\epsilon})$ such that 
\begin{equation}\label{eq_homotnew}
\beta^{i,\epsilon}_k \quad \text{ is homotopic to } \gamma^{i,\epsilon}_k \quad \text{ in } \C \setminus f^{-k}(P(f)).
\end{equation}

We shall now see that as $\epsilon \rightarrow 0$, for all $0 \leq i\leq N$, $\beta^{k,\epsilon}_i$ converges to a curve $\beta^k_i$, with endpoints $v_i$ and $v_{i+1}$ and that is homotopic to $\gamma^{k}_i$ in $\C \setminus f^{-k}(P(f)) \cup \{v_i, v_{i+1}\}$. Indeed, note that $f^{-k}(w_i)$ and $f^{-k}(w_{i+1})$ are discrete sets of points, and hence we can find open neighbourhoods $V_i \ni v_i$ and $V_{i+1} \ni v_{i+1}$ such that $V_i \cap f^{-k}(w_i)=\{ v_i\}$ and $V_{i+1} \cap f^{-k}(w_{i+1})=\{ v_{i+1}\}$. By the assumption $\text{AV}(f) \cap C=\emptyset$, using the Open Mapping theorem, we conclude that $f^k\vert_{ V_i}$ is an open map, and so, we can find an open neighbourhood $ W_i \ni w_i$ with $W_i \subset f^k(V_i)$. In particular, $\beta^{i,\epsilon}(t) \in V_i$ for all $t$ sufficiently close to $i/N+\epsilon$, and $\beta^{i,\epsilon}(t) \in V_{i+1}$ for all $t$ sufficiently close to $(i+1)/N+\epsilon$. Thus, by continuity of $f$, as $\epsilon \rightarrow 0$, each of the curves in $\{\beta^{i,\epsilon}\}_i$ converges to a curve with endpoints $v_i$ and $v_{i+1}$, that we denote by $\beta^k_i$. By \eqref{eq_homotnew} and since by construction $V_i\cap f^{-k}(P(f))\cap V_{i+1}=\{v_i,v_{i+1}\}$, we have that $\beta^k_i$ is homotopic to $\gamma^{k}_i$ in $\C \setminus f^{-k}(P(f)) \cup \{v_i, v_{i+1}\}$.

Thus, the curve
$$\beta_k\defeq \{ v_0\} \bm{\cdot} \beta^k_0 \bm{\cdot} \{ v_1\} \bm{\cdot} \cdots \bm{\cdot} \beta^k_N \bm{\cdot} \{v_{N+1}\},$$
satisfies by construction that $\beta_k\in [\gamma]_{_k} $ and $f^k(\beta_k)=\beta$, as required.
\end{proof}
The second goal of this section is to prove Corollary \ref{cor_homot2}. This result asserts that given a function $f$ and a domain $U$ in a hyperbolic orbifold, if certain technical conditions are satisfied, then there is a positive constant $\mu$ such that for any curve $\gamma \subset U$, there exists a curve in $[\gamma]_{_0}$ of orbifold length less than $\mu$. In the next auxiliary proposition we construct curves in any desired post-$0$-homotopy class of arbitrarily small orbifold length for orbifolds with a unique ramified point:
\begin{prop}[Short post-$0$-homotopy curves around a ramified point]\label{prop_circle} Given $\epsilon>0$ and $d\in \N_{\geq 1}$, define the hyperbolic orbifold $\Or\defeq(\D_\epsilon,\nu_d)$ with $\nu_d(0)=d$ and $\nu_d\equiv 1$ elsewhere, and let $\rho_{\Or}(z)dz$ be its orbifold metric. Let $f$ be an entire function such that $P(f)\cap \overline{\D}_{\epsilon} =\lbrace 0\rbrace$. Then, for all $\epsilon'<\epsilon$ small enough, $\ell_\Or(\partial \D_{\epsilon'})<\epsilon/6. $ Moreover, for any curve $\gamma\subset \overline{\D}_{\epsilon'}$, there exists $\tilde{\gamma} \in [\gamma]_{_0}$ satisfying $\ell_\Or(\tilde{\gamma})<\epsilon/6.$
\end{prop}
\begin{remark}
The function $f$ does not play any role in the proof of the proposition, and its role in the statement is to fix post-$0$-homotopy classes. Note that for any function $f$ as in the statement and each curve $\gamma \subset \D_\epsilon^\ast$, $[\gamma]_{_0}$ equals the homotopy class of $\gamma$ (in the usual sense) in the punctured disc $\D^\ast_\epsilon$.
\end{remark}

\begin{proof}[Proof of Proposition \ref{prop_circle}]
Let $\Ort$ be the orbifold with underlying surface $\D$ and with $0$ as its unique ramified point, of degree $d$. We computed in \eqref{eq_density} an explicit formula for the density of its orbifold metric, namely, for each $u\in \D^\ast$, it holds $\rho_{\Ort}(u) du=2\left(d\vert u\vert^{\frac{d-1}{d}}(1-\vert u\vert^{\frac{2}{d}})\right)^{-1
}$. If $\lambda_\epsilon$ is the function that factors by $\epsilon^{-1}$, that is, $\lambda_\epsilon(z)\defeq\epsilon^{-1} z$, then $\lambda_\epsilon \colon \Or \rightarrow \Ort$ is an orbifold covering map. Hence, see Observation \ref{obs_covering_lifting}, for each $z=\epsilon u$, the density of the orbifold metric of $\Or$ is $$\rho_{\Or}(z)=\epsilon^{-1}\rho_{\Ort}(u/\epsilon)=2\left( \epsilon^{1/d}\vert z\vert^{\frac{d-1}{d}}\left(1-\epsilon^{-2/d}\vert z\vert^{\frac{2}{d}}\right)\right)^{-1}.$$ Observe that for any $\epsilon_1<\epsilon$, the function $\rho_{\Or}$ is constant when restricted to $\partial \D_{\epsilon_1}$, and $\ell_{\Or}(\partial \D_{\epsilon_1})=2\pi\epsilon_1 \rho_{\Or}(\epsilon_1)$, as a function of $\epsilon_1$, is strictly decreasing and converging to $0$ whenever $\epsilon_1 \to 0$. Thus, the first part of the statement follows.

In order to prove the second part of the statement, note that for any $z\in \D_\epsilon$, the $\Or$-length of the radial line joining $0$ to $z$, a segment that we denote by $[0,z]$, is at most $\vert z \vert \rho_{\Or}(z)$. By a \textit{radial line} we mean any subcurve of a straight line in $\overline{\D}$ joining the origin to $\partial \D$. Thus, the $\Or$-length of the segment $[0,z]$ also converges to $0$ as $\vert z \vert \to 0$. Hence, we can fix any $\epsilon'<\epsilon$ such that
\begin{equation}\label{eq_e18}
\ell_\Or([0, \epsilon'])<\epsilon/18.
\end{equation}
Let $\gamma\subset \D_{\epsilon'}$ with endpoints $p$ and $q$. If $\gamma$ contains the point $0$, then $\gamma \in H^q_p(\{0\})$ and the concatenation of the radial lines joining $p$ and $q$ to $0$, that is, $\tilde{\gamma}\defeq [p,0] \bm{\cdot}[0,q]$, satisfies $\tilde{\gamma}\in [\gamma]_{_0}$ and by \eqref{eq_e18}, $\ell_\Or(\tilde{\gamma})<\epsilon/9<\epsilon/6$. Otherwise, $\gamma \subset H^q_p(\emptyset)$. Thus, the curves in $[\gamma]_{_0}$ are exactly those homotopic to $\gamma$ (in the usual sense) in $\D_{\epsilon'} \setminus \lbrace 0 \rbrace$ with fixed endpoints. Note that roughly speaking, the homotopy class of such a curve is determined by the number $n$ of times that the curve ``loops'' around $0$ following an orientation. Hence, we are aiming to construct a representative of any such class with a bound on its orbifold length, namely $\epsilon/6$. In a rough sense, for each $n\geq 0$, we define a representative $\gamma_n^+$ as follows: we start at the point $p$ and follow the radial line towards the origin until we meet a circle centred at the origin of some radius $\epsilon_n$ small enough. Then, we follow anticlockwise an arc of this circle until meeting the point on the radial line from $0$ to $q$. Then, we follow the circle of radius $\epsilon_n$ anticlockwise $n$ times. Finally, we follow the radial line to $q$. Similarly, we define a curve $\gamma_n^-$ starting at $q$ and following the circle of radius $\epsilon_n$ clockwise $n$ times.

More formally, for each natural $n\geq 0$, by the observations made at the beginning of the proof, we can choose $\epsilon_n< \epsilon'$ such that
\begin{equation}\label{eq_18n1}
\ell_{\Or}(\partial \D_{\epsilon_n})<\frac{\epsilon}{18(n+1)}.
\end{equation}
Define $[p,x(n)]$ and $[y(n),q]$ as the restriction of the radial lines from $p$ to $0$ and $0$ to $q$ with respective endpoints $\{x(n),y(n)\}\subset \partial\D_{\epsilon_n}$. Let $\alpha_n^+$ and $\beta_n^-$ be the arcs in $\partial \D_{\epsilon_n}$ that connect $x(n)$ to $y(n)$ in positive and negative orientation respectively; see Figure \ref{fig:homot_epsilon}. Let $\partial \D_{\epsilon_n}^+$ and $\partial \D_{\epsilon_n}^-$ be the loops starting at $y(n)$ positively and negatively oriented respectively. We define the curves $\gamma^+_n$ and $\gamma^-_n$ as the concatenations
\begin{equation*}
\begin{split}
\gamma^+_n&\defeq [p,x(n)]\bm{\cdot} \alpha_n^+ \bm{\cdot} \underbrace{\partial \D_{\epsilon_n}^+\bm{\cdot} \cdots \bm{\cdot} \partial \D_{\epsilon_n}^+}_{ n \text{ times }} \bm{\cdot} [y(n),q] \quad \text{ and } \\
\gamma^-_n&\defeq [p,x(n)]\bm{\cdot} \beta_n^- \bm{\cdot} \underbrace{\partial \D_{\epsilon_n}^-\bm{\cdot} \cdots \bm{\cdot} \partial \D_{\epsilon_n}^-}_{ n \text{ times }} \bm{\cdot} [y(n),q].
\end{split}
\end{equation*}

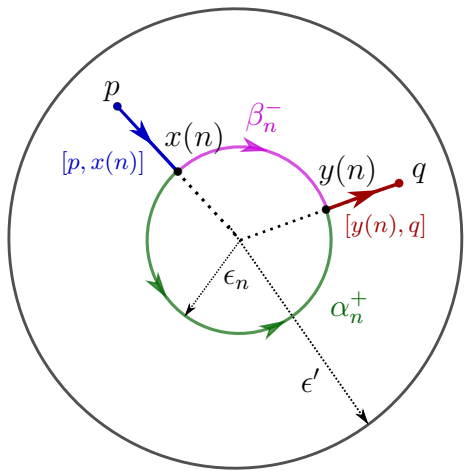
\begin{figure}[htb]
\begingroup%
\makeatletter%
\providecommand\color[2][]{%
	\errmessage{(Inkscape) Color is used for the text in Inkscape, but the package 'color.sty' is not loaded}%
	\renewcommand\color[2][]{}%
}%
\providecommand\transparent[1]{%
	\errmessage{(Inkscape) Transparency is used (non-zero) for the text in Inkscape, but the package 'transparent.sty' is not loaded}%
	\renewcommand\transparent[1]{}%
}%
\providecommand\rotatebox[2]{#2}%
\newcommand*\fsize{\dimexpr\f@size pt\relax}%
\newcommand*\lineheight[1]{\fontsize{\fsize}{#1\fsize}\selectfont}%
\ifx\svgwidth\undefined%
\setlength{\unitlength}{226.46421748bp}%
\ifx\svgscale\undefined%
\relax%
\else%
\setlength{\unitlength}{\unitlength * \real{\svgscale}}%
\fi%
\else%
\setlength{\unitlength}{\svgwidth}%
\fi%
\global\let\svgwidth\undefined%
\global\let\svgscale\undefined%
\makeatother%
\begin{picture}(1,0.76992088)%
\lineheight{1}%
\setlength\tabcolsep{0pt}%
\put(0,0){\includegraphics[width=\unitlength,page=1]{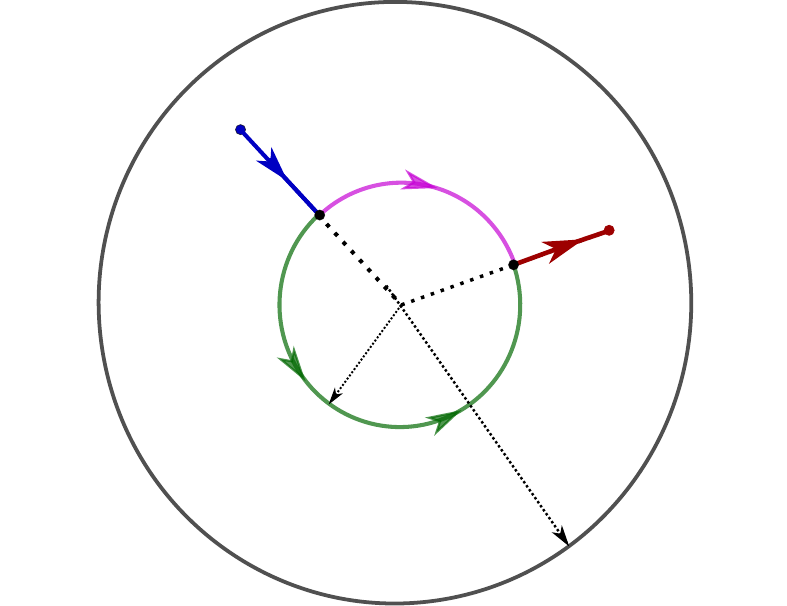}}%
\put(0.28388524,0.62451042){\color[rgb]{0,0,0}\makebox(0,0)[lt]{\lineheight{1.25}\smash{\begin{tabular}[t]{l}$p$\end{tabular}}}}%
\put(0.79583535,0.48664155){\color[rgb]{0,0,0}\makebox(0,0)[lt]{\lineheight{1.25}\smash{\begin{tabular}[t]{l}$q$\end{tabular}}}}%
\put(0.48185038,0.30867773){\color[rgb]{0,0,0}\makebox(0,0)[lt]{\lineheight{1.25}\smash{\begin{tabular}[t]{l}$\epsilon_n$\end{tabular}}}}%
\put(0.61112363,0.13001984){\color[rgb]{0,0,0}\makebox(0,0)[lt]{\lineheight{1.25}\smash{\begin{tabular}[t]{l}$\epsilon'$\end{tabular}}}}%
\put(0.38459563,0.53829906){\color[rgb]{0,0,0}\makebox(0,0)[lt]{\lineheight{1.25}\smash{\begin{tabular}[t]{l}$x(n)$\end{tabular}}}}%
\put(0.64035766,0.48353159){\color[rgb]{0,0,0}\makebox(0,0)[lt]{\lineheight{1.25}\smash{\begin{tabular}[t]{l}$y(n)$\end{tabular}}}}%
\put(0.21189424,0.50037018){\color[rgb]{0,0,0.76078431}\makebox(0,0)[lt]{\lineheight{1.25}\smash{\begin{tabular}[t]{l}\fontsize{9pt}{1em}$[p,x(n)]$\end{tabular}}}}%
\put(0.68329975,0.39505964){\color[rgb]{0.61568627,0,0}\makebox(0,0)[lt]{\lineheight{1.25}\smash{\begin{tabular}[t]{l}\fontsize{9pt}{1em}$[y(n),q]$\end{tabular}}}}%
\put(0.66003351,0.25662962){\color[rgb]{0,0.4,0}\makebox(0,0)[lt]{\lineheight{1.25}\smash{\begin{tabular}[t]{l}$\alpha_n^{+}$\end{tabular}}}}%
\put(0.51928283,0.57479722){\color[rgb]{0.77254902,0,0.83137255}\makebox(0,0)[lt]{\lineheight{1.25}\smash{\begin{tabular}[t]{l}$\beta_n^{-}$\end{tabular}}}}%
\end{picture}%
\endgroup%
\caption{Construction of representatives $\gamma_n^-$ and $\gamma_n^+$ for each post-$0$-homotopy class of curves in Proposition \ref{prop_circle} as a concatenation of oriented curves.}
	\label{fig:homot_epsilon}
\end{figure}

By the choices of $\epsilon'$ and $\epsilon_n$ in \eqref{eq_e18} and \eqref{eq_18n1}, $\max\{\ell_{\Or}(\gamma^-_n),\ell_{\Or}(\gamma^+_n) \}<\epsilon/6$, and thus, for each homotopy class of curves in $\D_{\epsilon'} \setminus \{0\}$, we have constructed a representative with the desired $\Or$-length. The statement now follows.
\end{proof}	

Given an entire function $f$, in order to construct in Corollary \ref{cor_homot2} curves of any post-$0$-homotopy class with uniformly bounded orbifold length in a compact set $U$, we will assume that there are dynamic rays landing at every point in $P(f)\cap U$. The reason for this is that we will use those dynamic rays as a boundary that other dynamic rays cannot cross more than once. Then, Corollary \ref{cor_homot2} will be a consequence of the more general Theorem \ref{new_lem_hom}, that shows that we can find curves of uniformly bounded orbifold length in any desired post-$0$-homotopy class, lying in any simply connected domain $C$ for which $P(f)\cup \overline{C}\subset \partial C$ and $\# (P(f)\cap \partial C)$ is finite. 

If $\pi: \D\rightarrow C$ is the Riemann map for some simply connected domain $C$, whenever $\partial C$ is locally connected, by Carathéodory-Torhorst's Theorem\footnote{\label{footTor}This theorem is commonly attributed only to Carathéodory, although its first full proof seems to date back to 1921 and was given by Marie Torhorst \cite{Torhorst}; see \cite[\S2]{lasse_prime_ends} for further discussion.} \cite[Theorem 2.1]{pommerenke_boundary}, $\pi$ extends continuously to a surjective map $\pi: \overline{\D}\rightarrow \overline{C}$, that we call the \textit{extended Riemann map.} Note that in that case, there might exist curves $\gamma \subset \partial C$ for which there is not a curve $\beta\subset \pi^{-1}(\gamma)$ satisfying $\pi(\beta)=\gamma$. For example, let $C$ be a disc $D$ minus a cross ``$+$'' that intersects $\partial D$ at a single point. Then, the horizontal segment of the cross would be an example of such a curve in $\partial C$. We will exclude those ``pathological cases'' in our result:
\begin{thm}[Curves in post-$0$-homotopy classes with uniformly bounded lengths] \label{new_lem_hom} 
Let $f$ be an entire map and let $\Or=(S, \nu)$ be a hyperbolic orbifold with $S\subset \C$. Let $C\subset S\setminus P(f)$ be a simply connected domain such that $C\Subset S$, $\partial C$ is locally connected and $\partial C \cap P(f)$ is finite. Let $\pi: \overline{\D}\rightarrow \overline{C}$ be the extended Riemann map. Then, there exists a constant $\eta\defeq \eta(C)>0$ with the following property. Let $\gamma$ be any injective curve such that either $\text{int}(\gamma) \subset C$, or $\gamma \subset \partial C$ and there exists a curve $\beta\subset \pi^{-1}(\gamma)$ satisfying $\pi(\beta)=\gamma$. Then, there exists a curve $\tilde{\gamma} \in [\gamma]_{_0}$ such that $\ell_{\Or}(\tilde{\gamma})\leq \eta.$
\end{thm}

\begin{proof}
For any two points $z,w \in \overline{\D}$, we denote by $[z,w]$ the straight segment joining them. We start by finding for each curve of the form $\pi([z,w])\subset \overline{C}$, a curve in its post-$0$-homotopy class of (uniformly) bounded $\Or$-length. With that aim, let $L: \overline{\D} \times \overline{\D} \rightarrow \R_{\geq 0}$ be given by
\begin{equation}\label{eq_L}
L(z,w)\defeq\inf_{\beta \in [\pi([z,w])]_{_0}}\ell_\Or(\beta).
\end{equation}

We claim that $L$ achieves a maximum value $\mu$ in $\overline{\D}\times \overline{\D}$. To prove this, firstly we note that since any geodesic in $\Or$ joining two points in $\overline{C}$ has by definition finite $\Or$-length, $L(z,w)<\infty$ for all $(z,w) \in \overline{\D} \times \overline{\D}$. Then, we show in the next claim that $L$ is upper semicontinuous, and the existence of the maximum follows from the combination of these two facts. 
\begin{Claim} \label{ClaimLem}The function $L$ is upper semicontinuous. 
\end{Claim}		
\begin{subproof} Let $(z,w) \in \overline{\D}\times \overline{\D}$ and $\epsilon>0$ be arbitrary but fixed. We want to show that there exists a neighbourhood $U(z)\times U(w)$ of $(z,w)$ such that for every $(\tilde{z},\tilde{w})\in U(z)\times U(w)$, $L(\tilde{z},\tilde{w})< L(z,w)+\epsilon.$ Since by assumption $\#(\overline{C}\cap P(f))<\infty$, $d\defeq \max_{z\in\overline{C}\cap P(f)}\nu(z) <\infty$. Let us choose $\epsilon'<\epsilon$ small enough so that the estimates provided by Proposition \ref{prop_circle} with the parameters $\epsilon$ and $d!$ hold. Moreover, since the set of ramified points of $\Or$ is discrete, we can choose $\epsilon'<\epsilon/3$ such that $\D_{\epsilon'}(\pi(z)) \cup \D_{\epsilon'}(\pi(w)) \subset \Or$ and the only possible ramified points in $\D_{\epsilon'}(\pi(z)) \cup \D_{\epsilon'}(\pi(w))$ are $\pi(z)$ and $\pi(w)$. We also choose $\epsilon'$ small enough such that $\D_{\epsilon'}(\pi(z)) \cap \D_{\epsilon'}(\pi(w))=\emptyset$. For the rest of the proof of the claim, we assume that $\pi(z)$ and $\pi(w)$ are ramified points of $\Or$ of degree $d!$, since by Corollary \ref{cor_pick}, any estimates on lengths of curves obtained in this setting also hold for the original ramification values of $\pi(z)$ and $\pi(w)$, that lie between $1$ and $d$.
	
By continuity of $\pi$, we can find connected neighbourhoods $U(z)\ni z, U(w)\ni w$, relatively open in $\D$ and satisfying the following properties:
\begin{itemize}[noitemsep,wide=0pt, leftmargin=\dimexpr\labelwidth + 2\labelsep\relax]
\item $\pi(U(z)) \cup \pi(U(w))\subset (\D_{\epsilon'}(\pi(z)) \cup \D_{\epsilon'}(\pi(w)))\cap \overline{C}$.
\item For any $(\tilde{z}, \tilde{w})\in U(z)\times U(w)$, there exists a curve in $\pi^{-1}(\partial \D_{\epsilon'}(\pi(z)))\cap \D$, that we denote by $\xi^{\tilde{z}}$, that joins the first point of intersection of $[z, w]$ with $\pi^{-1}(\partial \D_{\epsilon'}(\pi(z)))$, with the first point of intersection of $[\tilde{z}, \tilde{w}]$ with $\pi^{-1}(\partial \D_{\epsilon'}(\pi(z)))$. Similarly, there is an arc $\xi^{\tilde{w}}$ in $\pi^{-1}(\partial \D_{\epsilon'}(\pi(w)))\cap \D$ with analogous properties; see Figure \ref{fig:lem_homot}.
\end{itemize}
	
	\begin{figure}[htb]
	\begingroup%
	\makeatletter%
	\providecommand\color[2][]{%
		\errmessage{(Inkscape) Color is used for the text in Inkscape, but the package 'color.sty' is not loaded}%
		\renewcommand\color[2][]{}%
	}%
	\providecommand\transparent[1]{%
		\errmessage{(Inkscape) Transparency is used (non-zero) for the text in Inkscape, but the package 'transparent.sty' is not loaded}%
		\renewcommand\transparent[1]{}%
	}%
	\providecommand\rotatebox[2]{#2}%
	\newcommand*\fsize{\dimexpr\f@size pt\relax}%
	\newcommand*\lineheight[1]{\fontsize{\fsize}{#1\fsize}\selectfont}%
	\ifx\svgwidth\undefined%
	\setlength{\unitlength}{425.19685039bp}%
	\ifx\svgscale\undefined%
	\relax%
	\else%
	\setlength{\unitlength}{\unitlength * \real{\svgscale}}%
	\fi%
	\else%
	\setlength{\unitlength}{\svgwidth}%
	\fi%
	\global\let\svgwidth\undefined%
	\global\let\svgscale\undefined%
	\makeatother%
	\begin{picture}(1,0.53333333)%
	\lineheight{1}%
	\setlength\tabcolsep{0pt}%
	\put(0,0){\includegraphics[width=\unitlength,page=1]{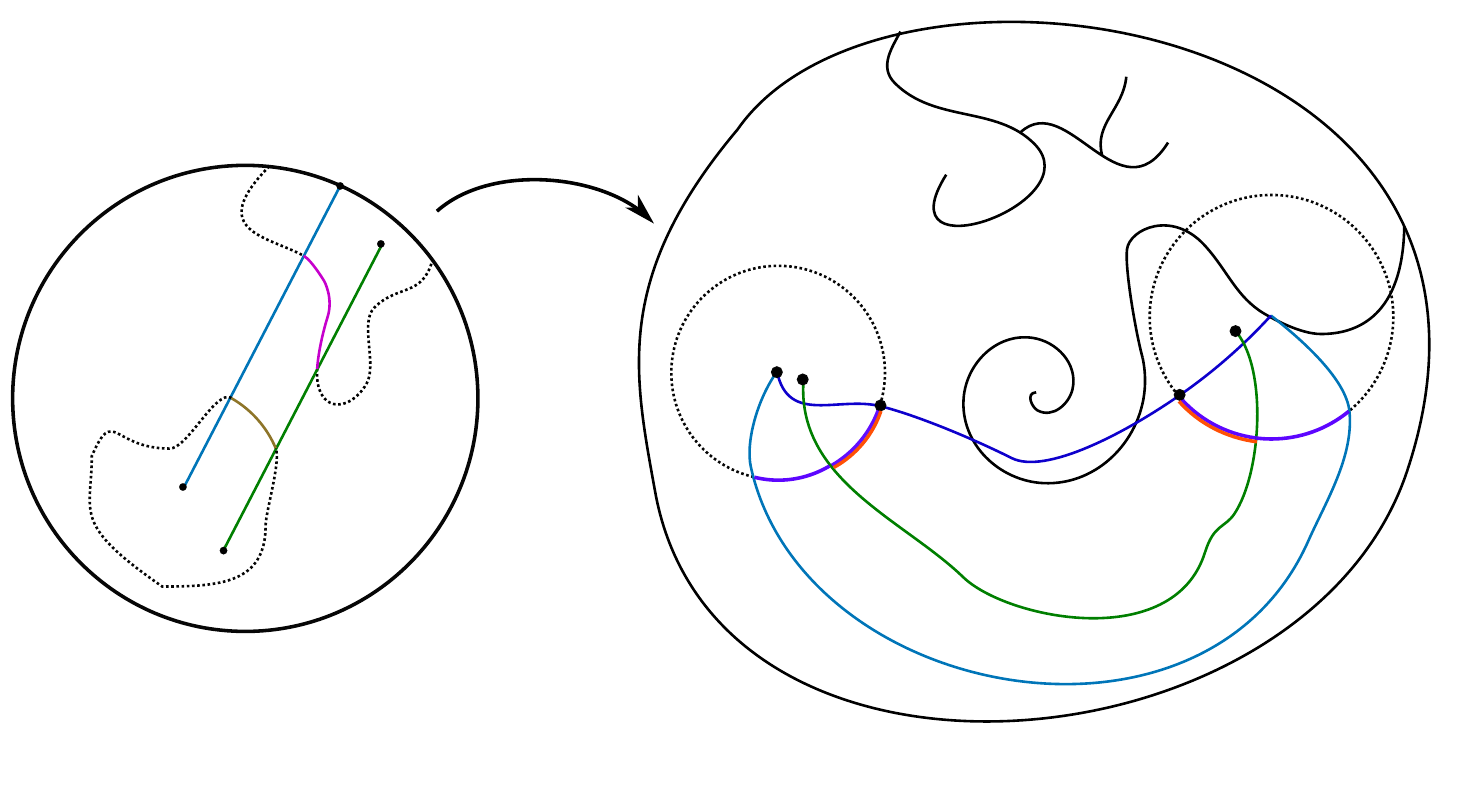}}%
	\put(0.05436037,0.41991304){\color[rgb]{0,0,0}\makebox(0,0)[lt]{\lineheight{1.25}\smash{\begin{tabular}[t]{l}$\overline{\mathbb{D}}$\end{tabular}}}}%
	\put(0.84936531,0.49435073){\color[rgb]{0,0,0}\makebox(0,0)[lt]{\lineheight{1.25}\smash{\begin{tabular}[t]{l}$\overline{C}$\end{tabular}}}}%
	\put(0,0){\includegraphics[width=\unitlength,page=2]{lemma_homot1.pdf}}%
	\put(0.22001214,0.41576166){\color[rgb]{0,0,0}\makebox(0,0)[lt]{\lineheight{1.25}\smash{\begin{tabular}[t]{l}$w$\end{tabular}}}}%
	\put(0.25815651,0.34645493){\color[rgb]{0,0,0}\makebox(0,0)[lt]{\lineheight{1.25}\smash{\begin{tabular}[t]{l}$\tilde{w}$\end{tabular}}}}%
	\put(0.10448885,0.18768942){\color[rgb]{0,0,0}\makebox(0,0)[lt]{\lineheight{1.25}\smash{\begin{tabular}[t]{l}$z$\end{tabular}}}}%
	\put(0.12806169,0.15391693){\color[rgb]{0,0,0}\makebox(0,0)[lt]{\lineheight{1.25}\smash{\begin{tabular}[t]{l}$\tilde{z}$\end{tabular}}}}%
	\put(0.20645698,0.2332133){\color[rgb]{0,0.49803922,0}\makebox(0,0)[lt]{\lineheight{1.25}\smash{\begin{tabular}[t]{l}\fontsize{9pt}{1em}$[\tilde{z},\tilde{w}]$\end{tabular}}}}%
	\put(0.11701106,0.30735251){\color[rgb]{0,0.45882353,0.72156863}\makebox(0,0)[lt]{\lineheight{1.25}\smash{\begin{tabular}[t]{l}\fontsize{9pt}{1em}$[z,w]$\end{tabular}}}}%
	\put(0.14776905,0.22458039){\color[rgb]{0.55686275,0.45882353,0.16078431}\makebox(0,0)[lt]{\lineheight{1.25}\smash{\begin{tabular}[t]{l}$\xi^{\tilde{z}}$\end{tabular}}}}%
	\put(0.21962243,0.35297553){\color[rgb]{0.77647059,0,0.8}\makebox(0,0)[lt]{\lineheight{1.25}\smash{\begin{tabular}[t]{l}$\xi^{\tilde{w}}$\end{tabular}}}}%
	\put(0.85040794,0.32654332){\color[rgb]{0,0,0}\makebox(0,0)[lt]{\lineheight{1.25}\smash{\begin{tabular}[t]{l}\fontsize{9pt}{1em}$\pi(w)$\end{tabular}}}}%
	\put(0.79157303,0.31880503){\color[rgb]{0,0,0}\makebox(0,0)[lt]{\lineheight{1.25}\smash{\begin{tabular}[t]{l}\fontsize{9pt}{1em}$\pi(\tilde{w})$\end{tabular}}}}%
	\put(0.81656709,0.26083227){\color[rgb]{0,0,0}\makebox(0,0)[lt]{\lineheight{1.25}\smash{\begin{tabular}[t]{l}$\fontsize{9pt}{1em}y$\end{tabular}}}}%
	\put(0.79770804,0.22887451){\color[rgb]{1,0.29803922,0}\makebox(0,0)[lt]{\lineheight{1.25}\smash{\begin{tabular}[t]{l}$\gamma_y$\end{tabular}}}}%
	\put(0.54239296,0.28370738){\color[rgb]{0,0,0}\makebox(0,0)[lt]{\lineheight{1.25}\smash{\begin{tabular}[t]{l}\fontsize{9pt}{1em} $\pi(\tilde{z})$\end{tabular}}}}%
	\put(0.55462659,0.2414167){\color[rgb]{0,0.25490196,0}\makebox(0,0)[lt]{\lineheight{1.25}\smash{\begin{tabular}[t]{l}\fontsize{8pt}{1em}$\lambda(\tilde{z})$\end{tabular}}}}%
	\put(0.60647118,0.25815654){\color[rgb]{0,0,0}\makebox(0,0)[lt]{\lineheight{1.25}\smash{\begin{tabular}[t]{l}$\fontsize{9pt}{1em}x$\end{tabular}}}}%
	\put(0.58837601,0.22309896){\color[rgb]{1,0.29803922,0}\makebox(0,0)[lt]{\lineheight{1.25}\smash{\begin{tabular}[t]{l}$\gamma_x$\end{tabular}}}}%
	\put(0.53002334,0.18385808){\color[rgb]{0.34901961,0,1}\makebox(0,0)[lt]{\lineheight{1.25}\smash{\begin{tabular}[t]{l}$\delta_x$\end{tabular}}}}%
	\put(0.69967843,0.13033699){\color[rgb]{0,0.49411765,0}\makebox(0,0)[lt]{\lineheight{1.25}\smash{\begin{tabular}[t]{l}\fontsize{9pt}{1em} $\pi([\tilde{z},\tilde{w}])$\end{tabular}}}}%
	\put(0.64748007,0.08644327){\color[rgb]{0,0.45882353,0.72156863}\makebox(0,0)[lt]{\lineheight{1.25}\smash{\begin{tabular}[t]{l}\fontsize{9pt}{1em} $\pi([z,w])$\end{tabular}}}}%
	\put(0.72686697,0.24263778){\color[rgb]{0.05490196,0,0.8}\makebox(0,0)[lt]{\lineheight{1.25}\smash{\begin{tabular}[t]{l}$\beta$\end{tabular}}}}%
	\put(0.86454475,0.20945232){\color[rgb]{0.34901961,0,1}\makebox(0,0)[lt]{\lineheight{1.25}\smash{\begin{tabular}[t]{l}$\delta_y$\end{tabular}}}}%
	\put(0.5121224,0.22034893){\color[rgb]{0.55686275,0.45882353,0.16078431}\makebox(0,0)[lt]{\lineheight{1.25}\smash{\begin{tabular}[t]{l}\fontsize{6pt}{1em}$\pi(\xi^{\tilde{z}})$\end{tabular}}}}%
	\put(0.85643694,0.25264343){\color[rgb]{0.77647059,0,0.8}\makebox(0,0)[lt]{\lineheight{1.25}\smash{\begin{tabular}[t]{l}\fontsize{6pt}{1em}$\pi(\xi^{\tilde{w}})$\end{tabular}}}}%
	\put(0.48002963,0.2813311){\color[rgb]{0,0,0}\makebox(0,0)[lt]{\lineheight{1.25}\smash{\begin{tabular}[t]{l}\fontsize{9pt}{1em}$\pi(z)$\end{tabular}}}}%
	\put(0.35435731,0.42724404){\color[rgb]{0,0,0}\makebox(0,0)[lt]{\lineheight{1.25}\smash{\begin{tabular}[t]{l}$\pi$\end{tabular}}}}%
	\put(0.11,0.39){\color[rgb]{0.4,0.4,0.4}\makebox(0,0)[lt]{\lineheight{1.25}\smash{\begin{tabular}[t]{l}\fontsize{9pt}{1em}$U(w)$\end{tabular}}}}%
	\put(0.18687377,0.18023402){\color[rgb]{0.4,0.4,0.4}\makebox(0,0)[lt]{\lineheight{1.25}\smash{\begin{tabular}[t]{l}\fontsize{9pt}{1em}$U(z)$\end{tabular}}}}%
	\end{picture}%
	\endgroup%
			\caption{Proof of upper semicontinuity of the function $L$. Points in $P(f)$ are represented by red stars.}
			\label{fig:lem_homot}
		
	\end{figure}
	
	In particular, $\pi(\xi^{\tilde{z}})$ and $\pi(\xi^{\tilde{w}})$ are arcs in $\partial\D_{\epsilon'}(\pi(z)) \cap C$ and $\partial \D_{\epsilon'}(\pi(w)) \cap C$ joining $\pi([\tilde{z},\tilde{w}])$ and $\pi([z,w])$. Let
	$\lambda(\tilde{z})$ be the restriction of $\pi([\tilde{z},\tilde{w}])$ between $\pi(\tilde{z})$ and the endpoint of $\pi(\xi^{\tilde{z}})$ that also belongs to $\pi([\tilde{z},\tilde{w}])$. In particular, $\lambda(\tilde{z})$ belongs to $\overline{\D}_{\epsilon'}(\pi(z))$, and thus, by Proposition \ref{prop_circle}, there exists $\tilde{\lambda}(\tilde{z}) \in [\lambda(\tilde{z})]_{_0}$ satisfying $\ell_\Or(\tilde{\lambda}(\tilde{z}))<\epsilon/6.$ Analogously, if $\lambda(\tilde{w})$ is the restriction of $\pi([\tilde{z},\tilde{w}])$ between $\pi(\tilde{w})$ and the endpoint of $\pi(\xi^{\tilde{w}})$ that also belongs to $\pi([\tilde{z},\tilde{w}])$, then there exists $\tilde{\lambda}(\tilde{w}) \in [\lambda(\tilde{w})]_{_0}$ such that $\ell_\Or(\tilde{\lambda}(\tilde{w}))<\epsilon/6.$ We also define $\lambda(z)$ (resp. $\lambda(w)$) as the restriction of $\pi([z,w])$ between $\pi(z)$ (resp. $\pi(w)$) and the endpoint of $\pi(\xi^{z})$ (resp. $\pi(\xi^{w})$) that also belongs to $\pi([z,w])$.
	
	\noindent Consider the subcurves 
	\begin{equation*}
	\begin{split}
	\lceil\pi([\tilde{z},\tilde{w}])\rceil&\defeq \pi([\tilde{z},\tilde{w}]) \setminus 
	(\lambda(\tilde{z}) \cup \lambda(\tilde{w})) \quad \text{ and } \\
	\lceil\pi([z,w])\rceil &\defeq \pi([z,w]) \setminus 
	(\lambda(z) \cup \lambda(w)).
	\end{split}
	\end{equation*}
	In particular, for each of the just-defined restrictions, one of their endpoints is an endpoint of $\pi(\xi^{\tilde{z}})$, and the other one is an endpoint of $\pi(\xi^{\tilde{w}})$. Since all curves with fixed endpoints totally contained in a simply connected domain are homotopic, see \cite[Proposition 1.6]{hatcher2002algebraic}, any two curves totally contained in
	$C$ are homotopic in $\C\setminus P(f)$, and thus the concatenation
	\begin{equation}\label{eq_conc1}
	\pi(\xi^{\tilde{z}}) \bm{\cdot} \lceil\pi([z,w])\rceil \bm{\cdot} \pi(\xi^{\tilde{w}})\quad \text{ is post-$0$-homotopic to } \quad \lceil\pi([\tilde{z},\tilde{w}])\rceil. 
	\end{equation}
	Let us choose\footnote{We believe that the infimum in \eqref{eq_L} is in fact a minimum, that is, we can choose a curve $\beta$ that is an \textit{orbifold geodesic} in the corresponding post-$0$-homotopy class with minimum length. Nonetheless, a reference has not been located and its existence is not required for our purposes.} any curve $\beta \in [\pi([z,w])]_{_0}$ so that $\ell_\Or(\beta)< L(z,w)+ \epsilon/3$. Let $x$ be the first point in $\beta \cap \partial \D_{\epsilon'}(\pi(z))$ and $y$ be last point in $\beta \cap \partial \D_{\epsilon'}(\pi(w))$ with respect to a parametrization of $\beta$ from $\pi(z)$ to $\pi(w)$, and let $\lceil\beta \rceil$ be the restriction of $\beta$ between those points; see Figure~\ref{fig:lem_homot}. Let us choose a pair of arcs $\delta_x \subset \partial \D_{\epsilon'}(\pi(z))$ and $\delta_y \subset \partial \D_{\epsilon'}(\pi(w))$ connecting respectively $x$ and $y$ to the single points in the intersections $\pi(\xi^{\tilde{z}})\cap \lceil\pi([\tilde{z},\tilde{w}])\rceil $ and $\pi(\xi^{\tilde{w}}) \cap \lceil\pi([\tilde{z},\tilde{w}])\rceil$, in such a way that the regions that those arcs together with $\lceil\beta \rceil$ and $\lceil\pi([\tilde{z},\tilde{w}])\rceil$ enclose, do not contain $\pi(z)$ nor $\pi(w)$. Since by assumption $\beta \in [\pi([z,w])]_{_0}$, by construction, the concatenation
	$$ \delta_x \bm{\cdot} \lceil\pi([z,w])\rceil \bm{\cdot} \delta_y \quad \text{ is post-$0$-homotopic to } \quad \lceil\beta\rceil.$$
	Consequently, if $\gamma_x\subset (\pi(\xi^{\tilde{z}})\cup \delta_x)$ and $\gamma_y\subset (\pi(\xi^{\tilde{w}})\cup \delta_y)$ are the curves joining the endpoints of $\lceil\pi([\tilde{z},\tilde{w}])\rceil$ and $\lceil \beta \rceil$, then, using \eqref{eq_conc1}, 
	$$\gamma_2 \defeq \gamma_x \bm{\cdot} \lceil \beta \rceil \bm{\cdot} \gamma_y \quad \text{ is post-$0$-homotopic to }\quad \lceil\pi([\tilde{z},\tilde{w}])\rceil. $$
	By construction and using Proposition \ref{prop_circle}, $$\ell_\Or(\gamma_2)\leq \ell_\Or(\partial \D_{\epsilon'}(\pi(z)))+ \ell_\Or(\beta)+\ell_\Or(\partial \D_{\epsilon'}(\pi(w))) < \ell_\Or(\beta)+ \epsilon/3.$$ Finally, the concatenation 
	$$\gamma \defeq \tilde{\lambda}(\tilde{z}) \bm{\cdot} \gamma_2 \bm{\cdot} \tilde{\lambda}(\tilde{w}) \quad \text{ is post-$0$-homotopic to }\quad \pi([\tilde{z},\tilde{w}]) $$
	and $L(\tilde{z}, \tilde{w})\leq \ell_{\Or}(\gamma) < L(z,w)+\epsilon$.
\end{subproof}	

If $\mu$ is the maximum value that $L$ attains in $\overline{\D} \times \overline{\D}$, then for every point $(z,w)\in \overline{\D}\times \overline{\D}$, we can find a curve $\beta\in [\pi([z,w])]_{_0}$ such that $\ell_\Or(\beta)< 2\mu$. Let $\gamma$ be an injective curve as in the statement. We start by considering both of the cases when int$(\gamma)\subset C$, and when $\gamma\subset \partial C$ and in addition int$(\gamma)\cap P(f)=\emptyset$. Let $p$ and $q$ be the endpoints of $\gamma$ and let $z,w \in \pi^{-1}(\{p,q\})$ be the endpoints of a curve in $\overline{\D}$ that is mapped injectively to $\gamma$ under $\pi$ . Note that such curve always exists: when int$(\gamma)\subset C$, it is the curve that contains the unique preimage $\pi^{-1}(\text{int}(\gamma))$, and when $\gamma\subset \partial C$, there are two such curves, that in particular share one of their endpoints. In both cases, $\pi([z,w])$ and $\gamma$ belong to a simply connected domain contained in $(\C \setminus P(f)) \cup \{p,q\}$. Thus, if we consider the set $W\defeq \gamma\cap P(f)$, which might be either empty or contain one of both of the endpoints $\{p,q\}$ of $\gamma$, we have, using again \cite[Proposition 1.6]{hatcher2002algebraic}, that $\pi([z,w]),\gamma\in \mathcal{H}^q_p(W(0))$ are post-$0$-homotopic, and in particular, $\pi([z,w])\in [\gamma]_{_0}$. Thus, there exists $\tilde{\gamma} \in [\pi([z,w])]_{_0}= [\gamma]_{_0}$ such that $\ell_\Or(\tilde{\gamma} )\leq 2\mu$.

We are left to consider the case when $\gamma \subset \partial C$ and there is a curve $\beta\subset \pi^{-1}(\gamma)$ satisfying $\pi(\beta)=\gamma$. Let $p$ and $q$ be the endpoints of $\gamma$ and suppose that $\gamma\in H^{q}_{p}(W)$ for $W\defeq P(f)\cap \gamma=\{w_1, \ldots, w_N\}$. Let us parametrize $\gamma$ so that $\gamma(0)=p$, $\gamma(1)=q$ and $\gamma(\frac{j}{N+1})=w_j$ for all $1\leq j \leq N$. In particular, following \eqref{def_concat}, we can express $\gamma$ as a concatenation $\gamma=\gamma_0\cdot\gamma_2\cdots \gamma_N$, where $\gamma_i\defeq \gamma\vert_{ [i/N, (i+1)/N ]}$. Note that for each $0\leq i\leq N$, $\gamma_i$ satisfies the hypotheses of the case considered above, that is, $\gamma_i \subset \partial C$ and int$(\gamma_i) \cap P(f)=\emptyset$. Thus, for each $i$, there exists a curve $\tilde{\gamma}_i \in [\gamma_i]_{_0}$ with $\ell_\Or(\tilde{\gamma}_i )\leq 2\mu$. Then, the concatenation $\tilde{\gamma}=\tilde{\gamma}_0\cdot\tilde{\gamma}_2\cdots \tilde{\gamma}_N$ satisfies $\tilde{\gamma} \in [\gamma]_{_0}$ and $\ell_\Or(\tilde{\gamma})\leq 2\mu \vert P(f)\vert\eqdef\nu $. Letting $\eta\defeq \max \{\nu,2\mu \}$, the theorem follows.
\end{proof}

Before stating our last result, we provide a formal definition of dynamic ray:
\begin{defn}[Dynamic rays {\cite[Definition 2.2]{RRRS}}]\label{def_ray}
Let $f\in \B$. A \emph{ray tail} of $f$ is an injective curve $\gamma :[t_0,\infty)\rightarrow I(f)$, with $t_0>0$, such that
\begin{itemize}
\item for each $n\geq 1$, $t\mapsto f^{n}(\gamma(t))$ is injective with $\lim_{t \rightarrow \infty} f^{n}(\gamma(t))=\infty$. 
\item $f^{n}(\gamma(t))\rightarrow \infty$ uniformly 
in $t$ as $n\rightarrow \infty$.
\end{itemize}
A \emph{dynamic ray} of $f$ is a maximal injective curve $\gamma :(0,\infty)\rightarrow I(f)$ such that the restriction $\gamma_{|[t,\infty)}$ is a ray tail for all $t > 0$. We say that $\gamma$ \emph{lands} at $z$ if $\lim_{t \rightarrow 0^+} \gamma(t)=z$, and we call $z$ the \emph{endpoint} of $\gamma$.
\end{defn}

\begin{observation}[Properties of dynamic rays] \label{obs_prop_rays} With our definition, dynamic rays might contain singular values and postsingular points, as occurs for the map $f=\cosh$, see \cite{mio_cosine,mio_splitting}. In addition, dynamic rays might overlap pairwise in subcurves delimited by (preimages) of critical points; see \cite[Chapter 4]{mio_thesis} or \cite[proof of Proposition 2.3]{RRRS} for further discussion. Moreover, any two rays intersect in a connected set, since otherwise, their intersection would enclose a domain that escapes uniformly to infinity, contradicting that $I(f)$ has empty interior as $f\in \mathcal{B}$, see \cite{eremenkoclassB}.
\end{observation}
\begin{cor}[Pieces of rays with uniformly bounded length]\label{cor_homot2}
Let $f\in \B$, let $\Or=(S, \nu)$ be a hyperbolic orbifold with $S\subset \C$, and let $U\Subset S$ be a simply connected domain with locally connected boundary. Assume that $P(f)\cap \overline{U} \subset J(f)$, $\#( P(f)\cap \overline{U})$ is finite and there exists a dynamic ray or ray tail landing at each point in $P(f)\cap U$. Then, there exists a constant $L_U \geq 0$, depending only on $U$, such that for any (connected) piece of ray tail $\xi \subset U$, there exists $\delta\in [\xi]_{_0}$ with $\ell_{\Or}(\delta)\leq L_U.$
\end{cor}
\begin{proof}
Let $P(f) \cap U\eqdef\{p_1,\ldots, p_N\}$ for some $N<\infty$. We start defining a set $X \supset (P(f) \cap U)$ using pieces of dynamic rays. By assumption, for each $1\leq i\leq N,
$ there exists at least one dynamic ray  or ray tail landing at each $p_i\in P(f) \cap U$. We choose any such ray and let $\Gamma_i$ be its parametrization including its landing point. Then, denote by $\gamma_i$ the unique connected component of $\Gamma_i\cap U$ that contains its landing point $p_i$. Note that $\gamma_i$ might contain some other points in $P(f)\cap U$. We construct the set $X$ inductively: define $X_1\defeq \gamma_1 \cup \partial U$. For each $2\leq j\leq N$, if $p_j\in X_{j-1}$, we define $X_j=X_{j-1}.$ Otherwise, let $X_j$ be the union of $X_{j-1}$ with the connected component of $\gamma_j \setminus X_{j-1}$ containing $p_j$, which by Observation \ref{obs_prop_rays}, is a bounded piece of dynamic ray. By construction, $X\defeq X_N$ is the union of $\partial U$ with a collection of $\tilde{N}\leq N$ connected components, each of them consisting of a concatenation of finitely many pieces of rays, and so that $U\setminus X$ is simply connected. That is, the set $X$ can be written as $X=\bigcup^{\tilde{N}}_{k=1} T_k \cup \partial U$, where each $T_k$ is topologically a \textit{tree} with finitely many edges $\lbrace e^k_1,\ldots,e^k_{m(k)} \rbrace$. We denote $M\defeq \max_{k\leq \tilde{N}}m(k)$.

Let $C\defeq U\setminus X$ and note that by construction, $C$ is a simply connected domain such that $C\cap P(f)=\emptyset$, and moreover, since $\partial C=X\cup \partial U$, $\partial C$ is locally connected. Hence, the set $C$ satisfies the hypotheses in Theorem \ref{new_lem_hom}. We claim that if $\xi$ is a piece of ray tail in $U=\overline{C}\setminus \partial U$, then $\xi$ is of one of the following three types:\\
\textbf{Type 1.} $\text{int}(\xi)\subset C$.\\
\textbf{Type 2.} $\xi\subset X$ and $\xi$ is a concatenation of at most $M$ curves $\{\alpha_i\}_{i\leq M}$, so that for each $\alpha_i$, there exists a curve $\beta_i \subset \partial \D$ such that $\pi(\beta_i)=\alpha_i$, where $\pi \colon \overline{\D} \rightarrow \overline{C}$ is the extended Riemann map.\\
\textbf{Type 3.} $\xi$ is a concatenation of at most $2\tilde{N}+1$ curves of types $1$ and $2$.\\
Indeed, if $\xi\subset \partial C$, by assumption, $\xi\subset T_k$ for some $k$, and so $\xi$ is contained in a concatenation of some of the edges $\lbrace e^k_1,\ldots,e^k_{m(k)} \rbrace$ of $T_k$. For each $i$, $\pi^{-1}(e^k_i)$ is either an arc that maps $2$-to-$1$ to $e^k_i$, or consists of two different arcs, each of them mapping $1$-to-$1$ to $e^k_i$. Thus, $\xi$ is of type $2$. Let us now analyse the case when $\xi\subset \overline{C}\setminus \partial U$ is a piece of dynamic ray which is not of type $1$ nor $2$. Then $\xi \cap T_k\neq \emptyset$ for some $k$. Since $T_k$ is a union of pieces of dynamic rays, all of its points but maybe some endpoints escape uniformly to infinity. Hence, by Observation~\ref{obs_prop_rays}, $\xi \cap T_k$ is connected. This means that $\xi \cap X$ is a collection of at most $\tilde{N}$ curves, preceded and/or followed by subcurves of $\xi$ with interior in $C$. Thus, $\xi$ is of type $3$. Consequently, by Theorem \ref{new_lem_hom}, there exists a constant $\eta$ such that if $\xi$ is a piece of dynamic ray in $U$, then there exists a curve $\tilde{\xi} \in [\xi]_{_0}$ such that $\ell_{\Or}(\tilde{\xi})\leq \max\{M, 2\tilde{N}+1\}\eta\eqdef\mu.$
\end{proof}

\bibliographystyle{alpha}
\bibliography{biblioComplex}
\end{document}